\documentclass[10pt,twoside]{article}
\usepackage{geometry}
\usepackage{amsmath,amsthm,amsfonts,amssymb,bm,stmaryrd,pifont,mathtools,leftidx}
\usepackage{times,mathrsfs}
\usepackage{graphicx,psfrag}
\usepackage[all]{xy}
\usepackage{wrapfig}
\usepackage{stmaryrd}

\usepackage{hyperref}

\SetSymbolFont{stmry}{bold}{U}{stmry}{m}{n}
\usepackage{titlesec}
\titleformat{\section}{\raggedright\bfseries}{$\bm{\Rmnum{\thesection}}$}{1em}{}
\titleformat{\subsection}{\raggedright\bfseries}{\thesubsection}{1em}{}
\usepackage{tikz}
\usetikzlibrary{arrows,backgrounds,shadows,decorations.markings,matrix}

\makeatletter
\newcommand{\f}[2]{\frac{#1}{#2}}
\newcommand{\rmnum}[1]{\mathrm{\romannumeral #1}}
\newcommand{\Rmnum}[1]{\mathrm{\expandafter\@slowromancap\romannumeral #1@}}
\newcommand{\df}{\mathsf{d}}
\newcommand{\id}{\mathsf{id}}

\newcommand{\ha}{\alpha}
\newcommand{\hb}{\beta}
\newcommand{\p}{\partial}
\newcommand{\pd}[2]{\frac{\partial#1}{\partial#2}}

\newcommand{\hess}{\mathsf{Hess}}

\renewcommand{\hom}{\mathsf{Hom}}
\renewcommand{\dim}[1]{\mathsf{dim}_{#1}}
\renewcommand{\exp}{\mathsf{exp}}
\renewcommand{\mod}{\mathsf{~mod~}}
\renewcommand{\deg}{\mathsf{deg~}}
\renewcommand{\det}{\mathsf{det~}}
\newcommand{\res}{\mathsf{Res}}
\newcommand{\sgn}[1]{\mathsf{sgn}(#1)}
\newcommand{\sm}[1]{\makebox[0.5\width]{#1}}         
\newcommand{\C}{\mathbb{C}}

\newcommand{\Q}{\mathbb{Q}}
\newcommand{\Z}{\mathbb{Z}}
\newcommand{\ainfty}{$\mathsf{A_\infty}$-}

\newcommand{\into}{\hookrightarrow}

\newcommand{\lra}[1]{\langle #1 \rangle}

\newcommand{\lrf}[1]{\left\{#1\right\}}
\newcommand{\lr}[1]{\left(#1\right)}
\providecommand{\abs}[1]{\left\lvert#1\right\rvert}

\newcommand{\hch}{\mathsf{HH}^\bullet}               
\newcommand{\fix}{\mathsf{Fix}}                      
\newcommand{\jac}{\mathsf{Jac}}                      
\newcommand{\e}{\mathsf{e}}                          
\newcommand{\bdf}{\mathsf{b}}                        
\newcommand{\val}[1]{\mathsf{Val}(#1)}
\newcommand{\pmlr}[2]{\leftidx{_+}(#1,#2)_-}
\newcommand{\pplr}[2]{\leftidx{_+}(#1,#2)_+}
\newcommand{\mmlr}[2]{\leftidx{_-}(#1,#2)_-}

\makeatother
\theoremstyle{plain}

\newtheorem{thm}{Theorem}[section]
\newtheorem{prop}[thm]{Proposition}
\newtheorem{lem}[thm]{Lemma}
\newtheorem{coro}[thm]{Corollary}
\theoremstyle{definition}
\newtheorem{defn}{Definition}[section]
\newtheorem{ex}{Example}[section]
\theoremstyle{remark}
\newtheorem*{rem}{Remark}

\begin{document}
\numberwithin{equation}{section}
\tikzset{-latex-/.style={decoration={
  markings,
  mark=at position #1 with {\arrow{latex}}},postaction={decorate}}}
\title{{G-twisted braces and orbifold Landau-Ginzburg Models}}
\author{Weiqiang He, Si Li and Yifan Li}
\date{}

\newcommand{\Addresses}{{
  \bigskip
  \footnotesize

  Weiqiang He, \textsc{YMSC, Tsinghua University,
    Beijing 100084, China}\par\nopagebreak
  \textit{E-mail address}: \texttt{wqhe@math.tsinghua.edu.cn}

  \medskip

  Si Li, \textsc{YMSC, Tsinghua University,
    Beijing 100084, China}\par\nopagebreak
  \textit{E-mail address}: \texttt{sili@mail.tsinghua.edu.cn}

  \medskip

  Yifan Li, \textsc{YMSC, Tsinghua University,
    Beijing 100084, China}\par\nopagebreak
  \textit{E-mail address}: \texttt{yf-li14@mails.tsinghua.edu.cn}

}}

\maketitle

\begin{abstract} Given an algebra with group $G$-action, we construct brace structures for its $G$-twisted Hochschild cochains. An an application, we construct $G$-Frobenius algebras for orbifold Landau-Ginzburg B-models and present explicit orbifold cup product formula for all invertible polynomials.

\end{abstract}
\tableofcontents

\section{Introduction}
Consider a triple $(A,W,G)$ (we call $G$-twisted curved algebra), where $A$ is an associative algebra with a compatible $G$-action and $W$ is a $G$-invariant central element of $A$. Associated to $A$, we have Hochschild cochains $C^\bullet(A,A)$ where many homological constructions are captured by brace structures \cite{G93,GJ94,VG95}. 
In this paper, we generalize such brace structures to $G$-twisted Hochschild cochains
$
C^\bullet(A,A[G])
$ (Definition \ref{def: twisted brace structure}) and prove a $G$-twisted version of higher pre-Jacobi identities (Proposition \ref{thm:twisted high pre-j identity}).

 Let
$$
\p_H: C^\bullet(A,A[G])\to C^{\bullet+1}(A,A[G])
$$
be the Hochschild differential. Let $\mathsf{HH}^\bullet(A,A[G])$ be the Hochschild cohomology with a natural cup product $\cup$. We show that $\cup$ satisfies a $G$-twisted commutativity relation (see also \cite{S17}) and the $G$-invariant subspace $\mathsf{HH}^\bullet(A,A[G])^G$ is a Gerstenhaber algebra (Theorem \ref{thm: dga on G-twisted Hochschild}). When $G$ is a finite group, a result of \c{S}tefan (Cor. 3.4 in \cite{St95}) implies that $\hch(A[G],A[G])$ and $\hch(A,A[G])^G$ are isomorphic as graded vector spaces (see also \cite{DE05}). Using $G$-twisted brace structures, we extend this to an isomorphism as Gerstenhaber algebras (Theorem \ref{thm: comparison of Hochschild}).

Let $W$ be a $G$-invariant central element of $A$.  It leads to a curving differential (Definition \ref{def: differentials})
$$
  \df_W: C^\bullet(A,A[G])\to C^{\bullet-1}(A,A[G])
$$
such that $(C^\bullet(A,A[G]),\p_H, \df_W)$ forms a mixed complex. All the above constructions apply to $(\p_H+\df_W)$-cohomologies which are instead $\Z/2\Z$-graded.

Our study of $(A,G,W)$ is motivated by mirror symmetry between two singularity theories: Landau-Ginzburg (LG) A-model and Landau-Ginzburg (LG) B-model. Such LG/LG mirror symmetry is parallel to the well-studied Calabi-Yau/Calabi-Yau and Toric/Landau-Ginzburg mirror symmetry.  The basic data for LG model consists of $(A,W,G)$ where $W$ is a holomorphic function (called the \emph{superpotential})
$$
W: X\to \C
$$
on a complex variety $X$ and  $G$ is a group (called the \emph{orbifold group}) acting on $X$ preserving $W$.  $A=\mathcal O(X)$ is the structure ring of $X$. A version of compact type Hochschild cohomology (\cite{HMS74,PP10,P11}, see Definition \ref{defn-Hochschild-cohomology}) turns out to be relevant for the state spaces of orbifold LG models \cite{V89,IV90}.

In this paper we focus on the case when $W:\mathbb{C}^N \to\mathbb{C}$ is a weighted homogeneous polynomial
$$
  W(\lambda^{q_i}x_i)=\lambda W(x_i), \quad \forall \lambda \in \C^*
$$
with an isolated critical point at the origin and contains no monomials of the form $x^ix^j$ for $i\neq j$. Here $q_i\in (0,{\f{1}{2}}]\cap \Q$ is called the weight of $x_i$. The orbifold group $G$ will be a subgroup of $G_W$ where
\begin{equation*}
G_{W} =
\left\{(\lambda_1,\dots,\lambda_N)\in(\mathbb{C}^{\times})^N\Big\vert\,W(\lambda_1\,x_1,\dots,\lambda_N\,x_N)=W(x_1,\dots,x_N)\right\}.
\end{equation*}
There exists a construction of LG/LG mirror pairs which originates from a physical construction of Berglund-H\"ubsch \cite{BH93} and is further completed by Krawitz \cite{K09}. In this construction,  the polynomial $W$ is required to be \emph{invertible} \cite{K94,CR11,CIR14}, 
i.e., the number of variables must equal to the number of monomials of $W$. By rescaling the variables, we can always write $W$ as
\begin{equation*}
W=\sum_{i=1}^{N}\prod_{j=1}^{N}x_j^{a_{ij}}.
\end{equation*}
We denote its \emph{exponent matrix} {by $E_W =\left(a_{ij}\right)_{N\times N}$. The mirror polynomial of $W$ is
\[
W^T=\sum_{i=1}^{N}\prod_{j=1}^{N}x_j^{a_{ji}},
\]
i.e., the exponent matrix $E_{W^T}$ of the mirror polynomial is the transpose matrix of $E_W$. The mirror construction between orbifold groups $G^T\subset G_{W^T}$ and $G\subset G_W$ is more involved but explicitly known. It has the property that bigger $G$ corresponds to smaller $G^T$ and vice versa.

A general mirror theorem is proved (see\cite{HLSW15} and references therein for  related works) between LG A-model (FJRW theory \cite{FJR13}) of $(W^T, G_{W^T})$ and LG B-model (Saito-Givental theory \cite{S83,G01}) of $(W, G=\text{trivial})$. 
Correlation functions for LG B-model when $G$ is not a trivial group is less known.

As an application of methods developed in this paper, we construct ${G}$-Frobenius algebraic structure in the sense of Kaufmann \cite{K03,K06} for orbifold Landau-Ginzburg B-models (Theorem \ref{thm: invertible LG gives a G-Frob algebra.}) (see also \cite{BTW16} for certain axiomatic discussion of Kaufman's definition in orbifold LG models). Moreover, it is the first time we are able to compute explicit orbifold cup product formula for all invertible polynomials. This computation is based on a construction of explicit homotopy retract (Section \ref{section-Koszul})
\[\begin{tikzpicture}
\matrix (a) [matrix of math nodes, row sep=3.5em, column sep=4em, text height=2ex, text depth=0.9ex]
{\big(C^{\bullet}(\overline{A},A[G]),\p_H+\df_W\big) & \big(K^{\bullet}(A,A[G]),\p_K + \tilde{\df}_W\big)\\};
\path[-latex]
(a-1-1.2) edge node[above,font=\scriptsize] {$\tilde{\Phi}^*$} (a-1-2.178.5);
\path[latex-]
(a-1-1.-2) edge node[below,font=\scriptsize] {${\Upsilon}^*$} (a-1-2.-178.5);
\draw[-latex]
(a-1-1.-170)  arc (-40:-320:11pt);
\node[left,outer sep=20pt,font=\scriptsize] at (a-1-1.-180) {$\tilde{\mathsf{H}}^*$};
\end{tikzpicture}\]
between $G$-twisted reduced Hochschild complex $C^{\bullet}(\overline{A},A[G])$ and $G$-twisted Koszul complex $K^{\bullet}(A,A[G])$. This homotopy retract is obtained from a version of homological perturbation from a construction of Shepler and Witherspoon \cite{SW11} in the case without $W$. The result is summarized as follows.

According to \cite{KS92}, an invertible polynomial is classified as a direct sum of three
types of elementary invertible polynomials:
\begin{description}
\item [(a)] Fermat type: $W = x_1^n$.
\item [(b)] Loop type: $W=x_{1}^{n_1}x_2+x_2^{n_2}x_3+\cdots + x_N^{n_N}x_1$.
\item [(c)] Chain type: $W= x_{1}^{n_1}x_2+x_2^{n_2}x_3+\cdots + x_N^{n_N}$.
\end{description}
Let $g\in G_W$, $g(x_i)=\lambda_i x_i$ where not all $\lambda_i$'s are $1$. We define a version of $g$-twisted Hessian $\hess^g(W)$ (Definition \eqref{defn-Hession}). For example, for loop type polynomial,  the $g$-twisted Hessian has the form
$$
\hess^{g}(W)= \f{(-1)^{N+1}+n_1n_2\cdots
n_N}{(1-\lambda_1)(1-\lambda_2)\cdots(1-\lambda_N)}x_1^{n_1-1}x_2^{n_2-1}\cdots x_N^{n_N-1}.
$$
Let $\mathfrak{1}_g$ be the generator of the $g$-twisted sector (see \eqref{eq: generator}). Then our orbifold cup product formula reads
$$
  \mathfrak{1}_g \cup \mathfrak{1}_{g^{-1}}=(-1)^{\f{N(N-1)}{2}}\hess^{g}(W) \mathfrak{1}_{\e}
$$
where $\e\in G$ is the identity.  This explicitly determines the full orbifold cup product  (Theorem \ref{thm: cup product formula for General invertible polynomial}).

Alternatively, there exists categorical approach to orbifold Landau-Ginzburg models in terms of the dg-category $\text{MF}(A[G],W)$ of matrix factorizations.  In the case when $G$ is trivial, Hochschild cohomology is computed by Dyckerhoff in terms of compact generators \cite{D11}. Alternately, there is another approach via curved algebra which we follow in this paper. In \cite{PP10}, Polishchuk and Positselski identify Hochschild cohomology and compact type Hochschild cohomology of $\text{MF}(A, W)$.  Segal \cite{S13} constructs a quasi-isomorphism between compact type Hochschild complex of $\text{MF}(A, W)$ and that of the curved algebra $(A,W)$. C\u{a}ld\u{a}raru and Tu compute the compact type Hochschild (co)-homology of $(A,W)$ explicitly in \cite{CT13}. In the orbifold case when $G$ is nontrivial, Hochschild cohomology is computed in terms of compact generators \cite{PV12} or directly in terms of curved algebras\cite{CT13,S13}. These two approaches are identified in \cite{T14}. Based on the result of \cite{PP10}, Shklyarov \cite{S17} shows that the categorical cup product of matrix factorizations is identical to the Hochschild cup product of $G$-twisted curved algebras.

On the other hand, it is very difficult to compute orbifold cup product through categories. In \cite{S17}, Shklyarov deduces a formula of cup product in terms of certain complicated unknown coefficient $\sigma_{g,h}$. In appendix of \cite{S17}, Basalaev and Shklyarov actually obtain interesting closed formula of cup product in some special cases (two variable chain polynomial and Fukaya category of surface).

In contrast to Shklyarov's approach, we use an explicit homotopy retract between Hochschild complex and Koszul complex to deduce a combinatorial formula. This allows us to explicitly compute orbifold cup product for all invertible polynomials.  Our result confirms a conjecture by Basalaev and Shklyarov in Appendix A of \cite{S17}. It is a very interesting question to compare our results with Shklyarov's formula.

\noindent{\bf Acknowledgements.}
We would like to thank Xiaojun Chen, Yu Liu, Junwu Tu, Bin Wang, Guodong Zhou for helpful discussions.
S. L. is partially supported by Grant 20151080445 of Independent
Research Program at Tsinghua University. W. H. is partially supported by Tsinghua Postdoc Grant 100410019.

\section{\boldmath$G$\unboldmath-twisted Hochschild complexes for curved algebras}

In this paper we fix a field $k$ of characteristic zero.  All vector spaces and algebras are defined over $k$. Let $V$ be a $\Z$ or $\Z/2\Z$-graded $k$-vector space. Let $V_d=\{x\in V\mid |x|=d\}\subset V$ denote the subspace of degree $d$ elements. Let $V[m]$ denote the degree shifting by $m$ such that
$$
     V[m]_d= V_{d+m}.
$$
When $V$ is $\Z/2\Z$-graded, the above degree shifting is understood modulo $2\Z$. For example
$$
  V[1]_{\bar 0}=V_{\bar 1}, \quad V[1]_{\bar 1}=V_{\bar 0}, \quad \text{where}\ \quad \bar 0,\bar 1\in \Z/2\Z.
$$

\begin{defn} A \emph{curved algebra} is a pair $(A,W)$ where $A$ is an associative algebra and $W$ is a central element of $A$. We will sometimes denote the pair $(A,W)$ for a curved algebra by $A_W$.
\end{defn}

\begin{defn}\label{def:G-curved algebra}
Let $G$ be a group. A \emph{$G$-twisted curved algebra} is a triple $(A,W, G)$ where $A$ is an associative algebra with $G$-equivariant product and $W$ is a $G$-invariant central element of $A$. $\forall g \in G$ and $a \in A$, we denote by $\leftidx{^{g}}a$ the left action of $a$ by $g$. The identity element of $G$ is always denoted by $\e$.
\end{defn}

\begin{defn}\label{def:G-twisted curved alg} Given a $G$-twisted curved algebra $(A,W,G)$, we define its \emph{$G$-orbifolding} to be the $G$-twisted curved algebra $(A[G], W, G)$ where
\begin{description}
\item[(a)] $A[G]=A\otimes_k k[G]$ is the crossed product algebra where $k[G]$ is the group algebra. The product is
$$
(a_1g_1)\cdot (a_2g_2) \coloneqq  (a_1\cdot \leftidx{^{g_1}}a_2)g_1g_2,\quad \text{where}\quad a_1,a_2\in A, g_1,g_2\in G.
$$

\item [(b)] $G$-action: \begin{align*}
G\times A[G] & \rightarrow A[G]\\
(g,ag') & \mapsto g.(ag') \coloneqq (\leftidx{^{g}}a) gg'g^{-1}
\end{align*}

\item [(c)] $G$-invariant central element  $=W\e$.
\end{description}
\end{defn}

In this section, we construct $G$-twisted brace structures ({Definition \ref{def: twisted brace structure}}) on $G$-twisted Hochschild cochains ({Definition \ref{G-twisted Hochschild cochains}}). This generalizes the usual brace operations on Hochschild cochains of associative algebras. As an application, we obtain a $G$-twisted commutative structure on $G$-twisted Hochschild cohomology  ({Theorem \ref{thm: dga on G-twisted Hochschild}}) and a comparison theorem between two versions of $G$-twisted Gerstenhaber algebras ({Theorem \ref{thm: comparison of Hochschild}} and {Theorem \ref{thm: comparison of curved Hochschild}}). This will be applied in the next section to establish our main results on $G$-Frobenius algebras of orbifold Landau-Ginzburg models.


\subsection{Hochschild cochains and brace structure}
\begin{defn}\label{def:Hochschild cochain}
Given an associative algebra $A$, we denote \emph{Hochschild cochains} and \emph{compact type Hochschild cochains} of $A$  by
\begin{align}
C^\bullet(A,A) =\mathop{\prod}\limits_{p=0}^\infty C^p(A,A), \quad
C^\bullet_c(A,A) = \mathop{\bigoplus}\limits_{p=0}^\infty C^p(A,A),
\end{align}
where
\begin{align*}
 C^p(A,A)=\hom_k(A^{\otimes p}, A).
\end{align*}
Given $\phi \in C^p(A,A)$, we write $|\phi|=p$ for its degree.
\end{defn}

Following  \cite{G93,GJ94,GV94,CS99,T04}, we can define higher operations on Hochschild cochains as a generalization of Gerstenhaber product introduced in \cite{G63}. Given $\phi\in C^p(A,A)$, $\phi_i\in C^{p_i}(A,A)$ for $i=1,\cdots,k$ and $a_1,a_2\cdots \in A$, we define (if $k\leq p$)
\begin{align}\label{eq: brace on Hoch cochains}
&\phi\{\phi_1,\phi_2, \cdots \phi_k\}(a_1 \otimes a_2 \otimes \cdots )\nonumber\\
= & \sum\limits_{\bm{I}\in\mathcal{I}} (-1)^{\sum\limits_{j=1}^k(i_j-1)(\abs{\phi_j}-1)} \phi\big(a_1 \otimes \cdots \otimes a_{i_1-1} \otimes \phi_1(a_{i_1}\otimes \cdots) \otimes \cdots\nonumber\\
 &\qquad\qquad\qquad\qquad\qquad\quad~~ \cdots\otimes a_{i_j-1}\otimes \phi_j(a_{i_j}\otimes \cdots)\otimes \cdots\big),
\end{align}
where
\begin{equation*}
\mathcal{I} = \lrf{\bm{I} = ( i_1 < i_2 < \cdots < i_k)\mid i_1>0 \text{ and } \forall 1 \leqslant j < k, i_{j} + \abs{\phi_j} \leqslant i_{j+1} }.
\end{equation*}
If $k>p$, we set  $\phi\{\phi_1,\phi_2, \cdots \phi_k\}=0$. The map
\begin{align*}
C^\bullet(A,A) & \rightarrow  C^\bullet(C^\bullet(A,A),C^\bullet(A,A))\\
\phi & \mapsto  \phi\lrf{\cdots}
\end{align*}
gives brace structures \cite{G93, GJ94,VG95}. It satisfies the following \emph{higher pre-Jacobi identity}
\begin{align}\label{eq: Higher prejacobi}
\nonumber&\phi \lrf{\phi_1,\phi_2 \cdots \phi_n}\lrf{\psi_1,\psi_2 \cdots \psi_m}\\
=&\sum \pm \phi\lrf{\psi_1\cdots \psi_{j_1-1},\phi_1\lrf{\psi_{j_1}\cdots}\cdots \psi_{j_n-1},\phi_n\lrf{\psi_{j_n}\cdots}\cdots \psi_m}.
\end{align}

When $k = 1$, we get the Gerstenhaber product
\begin{equation}
\phi \lrf{\phi_1} = \phi \circ \phi_1.
\end{equation}

Pre-Jacobi identity implies that the Gerstenhaber bracket on $C^\bullet(A,A)$ defined by
\begin{equation}\label{eq: Gerstenhaber bracket}
\lrf{\phi_1,\phi_2}= \phi_1 \lrf{\phi_2} - (-1)^{ (\abs{\phi_1} - 1) (\abs{\phi_2} - 1) } \phi_2\lrf{\phi_1},\quad \forall \phi_1,\phi_2\in C^\bullet(A,A),
\end{equation}
gives a graded Lie algebra structure on $C^\bullet(A,A)[{1}]$.

 The product $\cdot$ on $A$ gives rise to a Hochschild cochain
$$
  m_2\in C^2(A, A), \quad m_2(a_1, a_2)=a_1\cdot a_2, \quad a_i\in A.
$$
The associativity of the product is equivalent to
$$
  \{m_2, m_2\}=0.
$$

\begin{defn} The \emph{cup product} $\cup$ and \emph{Hochschild differential} $\p_H$ on $C^\bullet(A,A)$ are defined by
\begin{align}
  \phi_1\cup \phi_2= (-1)^{|\phi_1|(\abs{\phi_2}-1)} m_2 \{\phi_1,\phi_2\}, \quad    \p_H (\phi) \coloneqq (-1)^{\abs{\phi}-1} \{m_2, \phi \}.
\end{align}
\end{defn}

Higher 
pre-Jacobi identity for brace operations implies the following higher homotopies,
\begin{align}\label{untwisted higher homotopy}
&\p_H(\phi)\lrf{\phi_1,\phi_2,\cdots, \phi_n} \nonumber \\
=&~ (-1)^{\xi_n} \p_H \big(\phi\lrf{\phi_1,\cdots, \phi_n}\big) - \sum_{k=1}^n(-1)^{\xi_k}\phi\lrf{\phi_1,\cdots, \phi_{k-1},\p_H(\phi_k),\phi_{k+1},\cdots, \phi_n}\nonumber\\
&+(-1)^{\abs{\phi_1}\xi'_1} \phi_1 \cup \phi\lrf{\phi_2,\phi_3,\cdots, \phi_n} - (-1)^{\abs{\phi}\abs{\phi_n}+\xi_{n-1}\xi'_{n-1}} \phi\lrf{\phi_1,\phi_2,\cdots, \phi_{n-1}} \cup \phi_n\nonumber\\
&+\sum_{k=1}^{n-1} (-1)^{\xi_k+\abs{\phi_k}\abs{\phi_{k+1}}} \phi\lrf{ \phi_1,\cdots, \phi_k \cup \phi_{k+1},\cdots,
\phi_{n}},
\end{align}
with
\begin{equation*}
\begin{dcases}\xi_k = \abs{\phi_1}+\abs{\phi_2}+\cdots+ \abs{\phi_k}-k,\\
\xi'_k = \abs{\phi_{k+1}}+\abs{\phi_{k+2}}+\cdots+ \abs{\phi_n}-(n-k).
\end{dcases}
\end{equation*}

Here are some identities for lower braces
\begin{itemize}

\item $n=1$.  We find
\begin{align}\label{Ger2}
    (-1)^{\abs{\phi_1}}\p_H(\phi\{\phi_1\}) + (\p_H\phi)\{\phi_1\} - (-1)^{|\phi_1|} \phi\{\p_H\phi_1\} = \phi_1 \cup \phi-(-1)^{|\phi||\phi_1|}\phi \cup \phi_1.
\end{align}
This says that $\cup$ on $C^\bullet(A,A)$ is commutative up to homotopy.  Switching the role of $\phi, \phi_1$ and comparing the difference, we find the compatibility of $\p_H$ with the Gerstenhaber bracket \cite{G63},
\begin{align}\label{Ger3}
 \p_H\{\phi,\phi_1\} = (-1)^{|\phi_1|-1}\{\p_H \phi, \phi_1\} + \{\phi, \p_H \phi_1\}.
\end{align}
\item $n=2$. We find
\begin{align}\label{Ger4}
  &\p_H(\phi\{\phi_1,\phi_2\}) - (-1)^{|\phi_2|-1}\phi\{\p_H\phi_1,\phi_2\} - \phi\{\phi_1,\p_H\phi_2\} - (-1)^{\abs{\phi_1} + \abs{\phi_2}} \p_H(\phi)\{\phi_1,\phi_2\}\nonumber \\
   = & (-1)^{(|\phi_1|-1)|\phi_2|}\left( \phi\{\phi_1\cup \phi_2\} - (-1)^{(|\phi|-1)\abs{\phi_2}}\phi\{\phi_1\}\cup \phi_2-\phi_1\cup \phi\{\phi_2\} \right).
\end{align}
Set $\phi=m_2$ and use $m_2\{m_2\}=0$, we find
\begin{align}\label{Ger1}
\p_H(\phi_1\cup \phi_2)=\p_H(\phi_1)\cup \phi_2+(-1)^{|\phi_1|}\phi_1\cup \p_H(\phi_2).
\end{align}
This says that the triple
$$
  (C^\bullet(A,A), \p_H, \cup)
$$
defines a differential graded algebra (dga).
\end{itemize}

\begin{defn}\label{def:Gerstenhaber algebra}
A \emph{Gerstenhaber algebra} $(A,\cdot,\{~~,~~\})$ is a $\mathbb{Z}$ (or $\mathbb{Z}/2\Z$)-graded commutative algebra with a graded Lie structure on $A[1]$ satisfying the shifted Poisson identities. Thus, for $a,b,c\in A$, we have
\begin{itemize}
  \item $|a\cdot b| = |a|+|b|$ and $a\cdot b = (-1)^{|a||b|}b\cdot a$.
  \item $|\{a,b\}|  = |a|+|b|-1$, $\{a,b\} = -(-1)^{(|a|-1)(|b|-1)}\{b,a\}$ and $\{a,\{b,c\}\}=\{\{a,b\},c\}+(-1)^{(|a|-1)(|b|-1)}\{b,\{a,c\}\}$.
  \item $\{a\cdot b, c\} = \{a,b\}\cdot c+(-1)^{\abs{a}(|c|-1)}a\cdot\{b,c\}$.
\end{itemize}
\end{defn}

The direct consequence of \eqref{eq: Gerstenhaber bracket} \eqref{Ger2} \eqref{Ger3}\eqref{Ger4}\eqref{Ger1}  says that  the Hochschild cohomology
$$
 \mathsf{HH}^\bullet(A)\coloneqq\mathsf{H}^\bullet(C^\bullet(A,A),  \p_H)
$$
together with the cup product and Gerstenhaber bracket form a Gerstenhaber algebra.

\subsection{Curved algebras and mixed complex}
We consider a curved algebra $(A,W)$ where $W$ is a central element of $A$.  The curving $W$ defines
\begin{align*}
   m_0\in C^0(A,A), \quad m_0(1)=W.
\end{align*}
\begin{defn}\label{def: differentials} We define the \emph{curving differential} $\df_W$ on Hochschild cochains of $A$ by
\begin{align}
  \df_W = (-1)^{p}\{m_0,-\}: C^{p}(A,A)\to C^{p-1}(A,A) 
\end{align}
and the \emph{curved Hochschild differential} by
\begin{align}
  \p_H^W = \p_H+ \df_W.
\end{align}
\end{defn}
The following identities hold
$$
  \p_H^2=\df_W^2=(\p_H^W)^2=0.
$$
Therefore the triple $\{C^\bullet(A,A), \p_H, \df_W\}$ defines a mixed complex. We are interested in the cohomology for the mixed differential $\p_H^W$, which is sensitive to the topology we use. Following  \cite{CW10,CT13,S13}, the appropriate complex for Landau-Ginzburg models turns out to be the one of compact type above as induced in \cite{HMS74,PP10,P11}.

\begin{defn}We define the \emph{compact type Hochschild cohomology} for a curved algebra $(A,W)$ by
\begin{align}
   \mathsf{HH}_c(A_W)= \mathsf{H} \left(C^\bullet_c(A,A), \p_H^W \right).
\end{align}
$\mathsf{HH}_c(A_W)$ is $\Z/2\Z$-graded in terms of the parity of the degree
$$
\mathsf{HH}_c(A_W)=\mathsf{HH}^{\bar 0}_c(A_W)\oplus \mathsf{HH}^{\bar 1}_c(A_W).
$$
\end{defn}

\begin{prop} $\mathsf{HH}_c(A_W)$ is a $\Z/2\Z$-graded Gerstenhaber algebra.
\end{prop}
\begin{proof}
Equations  \eqref{Ger2} \eqref{Ger3}\eqref{Ger4}\eqref{Ger1} still hold if $\p_H$ is replaced by $\p_H^W$. It follows that the cup product $\cup$ and the Gerstenhaber bracket $\{-,-\}$ induce the Gerstenhaber algebra structure on $\mathsf{HH}_c(A_W)$.
\end{proof}

\boldmath
\subsection{$G$-curved algebras and $G$-twisted brace structures}
\unboldmath

\begin{defn}\label{G-twisted Hochschild cochains}
Let $(A,W,G)$ be a $G$-twisted curved algebra. We define the \emph{$G$-twisted Hochschild cochains and compact type Hochschild cochains} by
\begin{align}
C^\bullet(A,A[G]) =\mathop{\prod}\limits_{p=0}^\infty C^p(A,A[G]), \quad
C^\bullet_c(A,A[G]) = \mathop{\bigoplus}\limits_{p=0}^\infty C^p(A,A[G]),
\end{align}
where
\begin{align*}
 C^p(A,A[G])=\hom_k(A^{\otimes p}, A[G]).
\end{align*}
\end{defn}
There is a natural $G$-action on $G$-twisted Hochschild cochains
\begin{align*}
G\times C^p(A,A[G]) & \rightarrow C^p(A,A[G]),\\
(g,\phi) & \mapsto g^*(\phi),
\end{align*}
given by
\begin{equation}\label{eq: group action on twisted cochains}
g^*(\phi)(a_1\otimes a_2\otimes \cdots \otimes a_p) =g \cdot \phi(\leftidx{^{g^{-1}}}a_1\otimes \leftidx{^{g^{-1}}}a_2 \otimes \cdots \otimes \leftidx{^{g^{-1}}} a_p) \cdot g^{-1}.
\end{equation}

\boldmath
\subsubsection*{$G$-twisted braces}
\unboldmath

\begin{defn}[$G$-twisted braces]\label{def: twisted brace structure}
Given $\phi\in C^p(A,Ag)$, $\phi_i\in C^{p_i}(A,Ag_i)$ for $i=1,\cdots,k$ and $a_1,a_2\cdots \in A$, we define (if $k\leq p$)
\begin{align}\label{eq: twisted brace on Hoch cochains}
&\phi\{\phi_1,\phi_2, \cdots \phi_k\}(a_1 \otimes a_2 \otimes \cdots )\nonumber\\
\coloneqq & \sum\limits_{\bm{I}\in\mathcal{I}} (-1)^{\sum\limits_{j=1}^k(i_j-1)(\abs{\phi_j}-1)} \phi^\circ\Big(a_1 \otimes \cdots \otimes a_{i_1-1} \otimes \phi^\circ_1(a_{i_1}\otimes \cdots \otimes a_{i_1+\abs{\phi_1}-1}) \otimes \leftidx{^{g_1}}a_{i_1+\abs{\phi_1}}\otimes\cdots\nonumber\\
 &\qquad\qquad\qquad\qquad\quad~~ \cdots\otimes \leftidx{^{g_{j-1}\cdots g_1}}a_{i_j-1}\otimes \phi^\circ_j(\leftidx{^{g_{j-1}\cdots g_1}}a_{i_j}\otimes \cdots)\otimes \leftidx{^{g_j\cdots g_1}}a_{i_j+\abs{\phi_j}}\otimes \cdots\nonumber\\
 &\qquad\qquad\qquad\qquad\quad~~ \cdots\otimes \phi^\circ_k(\leftidx{^{g_{k-1}\cdots g_1}}a_{i_k}\otimes \cdots)\otimes \leftidx{^{g_k\cdots g_1}}a_{i_k+\abs{\phi_j}}\otimes \cdots\Big)gg_k\cdots g_1,
\end{align}
where
\begin{equation*}
\mathcal{I} \coloneqq \lrf{\bm{I} = ( i_1 < i_2 < \cdots < i_k)\mid 0 < i_1 \text{ and } \forall 1 \leqslant j < k, i_{j} + \abs{\phi_j} \leqslant i_{j+1} }.
\end{equation*}
Here $\phi^\circ \in C^\bullet(A,A)$ such that $\phi = \phi^\circ g\in C^\bullet(A,Ag)$. Similarly for $\phi_i^\circ$. We will set  $\phi\{\phi_1,\phi_2, \cdots \phi_k\}=0$, if $k>p$. This extends linearly to brace structures on $G$-twisted Hochschild cochains
$$
   C^\bullet(A,A[G])\to C^\bullet\left(C^\bullet(A,A[G]),C^\bullet(A,A[G])\right), \quad \phi \mapsto \phi\{\cdots\},
$$
which we call $G$-twisted braces.
\end{defn}

\begin{rem} We may use tree graphs to express terms above. For example, we write
\[(-1)^{(i_1-1)(|\phi_1|-1)+(i_2-1)(|\phi_2|-1)}\tikz[scale=0.06,baseline=-5]{
\draw[thick] (10,0) node[right] {$gg_2g_1$} -- (0,0) -- ++ (-10,16);
\draw[thick] (0,0) -- ++ (-10,8);
\draw[thick] (0,0) -- ++ (-10,0)node[left,font=\tiny,inner sep = 1pt] {$g_{\sm{1}}$};
\draw[thick] (0,0) -- ++ (-10,-8);
\draw[thick] (0,0) -- ++ (-10,-16) node[left,font=\tiny,inner sep = 1pt] {$g_{\sm{2}}g_{\sm{1}}$};
\draw[thick] (-10,8) -- ++ (-10,-5);
\draw[thick] (-10,8) -- ++ (-10,0);
\draw[thick] (-10,8) -- ++ (-10,5) node[font=\tiny,below,outer sep = 0pt,inner sep = 1pt ] {$i_{\sm{1}}$};
\draw[thick] (-10,-8) -- ++ (-10,-5) node[left,font=\tiny,inner sep = 1pt] {$g_{\sm{1}}$};
\draw[thick] (-10,-8) -- ++ (-10,0) node[left,font=\tiny,inner sep = 1pt] {$g_{\sm{1}}$};
\draw[thick] (-10,-8) -- ++ (-10,5) node[font=\tiny,below,inner sep = 1pt] {$i_{\sm{2}}$};
\node[left,font=\tiny,inner sep = 1pt] at (-20,-3) {$g_{\sm{1}}$};
\filldraw[fill=black] (-10,16) circle (10pt);
\filldraw[fill=black] (-10,-16) circle (10pt);
\filldraw[fill=black] (-10,0) circle (10pt);
\filldraw[fill=black] (-20,-13) circle (10pt);
\filldraw[fill=black] (-20,-3) circle (10pt);
\filldraw[fill=black] (-20,-8) circle (10pt);
\filldraw[fill=black] (-20,13) circle (10pt);
\filldraw[fill=black] (-20,3) circle (10pt);
\filldraw[fill=black] (-20,8) circle (10pt);
\node[font=\tiny,rectangle,draw=black,fill=white,inner sep=0.5] at (0,0) {$\phi$};
\node[font=\tiny,rectangle,draw=black,fill=white,inner sep=0.5] at (-10,8) {$\phi_1$};
\node[font=\tiny,rectangle,draw=black,fill=white,inner sep=0.5] at (-10,-8) {$\phi_2$};
\filldraw[fill=white,draw=black] (10,0) circle (10pt);
},\]
for
\begin{align*}
&(-1)^{(i_1-1)(|\phi_1|-1)+(i_2-1)(|\phi_2|-1)} \phi^\circ \big(a_1\otimes \cdots \otimes a_{i_1-1}\otimes\phi_1^\circ(a_{i_1}\otimes \cdots)\otimes\leftidx{^{g_1}}a_{i_1+\abs{\phi_1}}\otimes \cdots\nonumber\\
&\qquad\qquad\qquad \cdots \otimes\leftidx^{g_1}a_{i_2-1}\otimes \phi_2^\circ (\leftidx{^{g_1}}a_{i_2}\otimes \cdots )\otimes  \leftidx{^{g_2g_1}}a_{i_2+\abs{\phi_2}}\otimes \cdots\big)gg_2g_1,
\end{align*}
as a term in $\phi\lrf{\phi_1,\phi_2}(a_1,a_2,\cdots )$.
As another example, the following graph
\[(-1)^{(i_1-1)(|\phi_1|-1)+(i_2-1)(|\phi_2|-1)}\tikz[scale=0.06,baseline=-5]{
\draw[thick] (10,0) node[right] {$gg_1g_2$} -- (0,0) -- ++ (-10,16);
\draw[thick] (0,0) -- ++ (-10,8);
\draw[thick] (0,0) -- ++ (-10,0)node[left,font=\tiny,inner sep = 1pt] {$g_{\sm{1}}$};
\draw[thick] (0,0) -- ++ (-10,-8) -- ++ (-10,0);
\draw[thick] (0,0) -- ++ (-10,-16) node[left,font=\tiny,inner sep = 1pt] {$g_{\sm{1}}g_{\sm{2}}$};
\draw[thick] (-10,8) -- ++ (-10,-5);
\draw[thick] (-10,8) -- ++ (-10,0);
\draw[thick] (-10,8) -- ++ (-10,5) node[font=\tiny,below,outer sep = 0pt,inner sep = 1pt ] {$i_{\sm{1}}$};
\draw[thick] (-20,-8) -- ++ (-10,-5);
\draw[thick] (-20,-8) -- ++ (-10,0);
\draw[thick] (-20,-8) -- ++ (-10,5) node[font=\tiny,below,inner sep = 1pt] {$i_{\sm{2}}$};
\filldraw[fill=black] (-10,16) circle (10pt);
\filldraw[fill=black] (-10,-16) circle (10pt);
\filldraw[fill=black] (-10,0) circle (10pt);
\filldraw[fill=black] (-10,-8) circle (10pt);
\filldraw[fill=black] (-30,-13) circle (10pt);
\filldraw[fill=black] (-30,-3) circle (10pt);
\filldraw[fill=black] (-30,-8) circle (10pt);
\filldraw[fill=black] (-20,13) circle (10pt);
\filldraw[fill=black] (-20,3) circle (10pt);
\filldraw[fill=black] (-20,8) circle (10pt);
\node[font=\tiny,rectangle,draw=black,fill=white,inner sep=0.5] at (0,0) {$\phi$};
\node[font=\tiny,rectangle,draw=black,fill=white,inner sep=0.5] at (-10,8) {$\phi_1$};
\node[font=\tiny,rectangle,draw=black,fill=white,inner sep=0.5] at (-20,-8) {$\leftidx{^{g_{\sm{1}}}}\phi_2$};
\filldraw[fill=white,draw=black] (10,0) circle (10pt);
},\]
for
\begin{align*}
&(-1)^{(i_1-1)(|\phi_1|-1)+(i_2-1)(|\phi_2|-1)} \phi^\circ \big(a_1\otimes \cdots \otimes a_{i_1-1}\otimes\phi_1^\circ(a_{i_1}\otimes \cdots)\otimes\leftidx{^{g_1}}a_{i_1+\abs{\phi_1}}\otimes \cdots\nonumber\\
&\qquad\qquad\qquad \cdots \otimes\leftidx^{g_1}a_{i_2-1}\otimes \leftidx{^{g_1}}\phi_2^\circ (a_{i_2}\otimes \cdots )\otimes  \leftidx{^{g_1g_2}}a_{i_2+\abs{\phi_2}}\otimes \cdots\big)gg_1g_2,
\end{align*}
appears as a term in $\phi\lrf{\phi_1}\lrf{\phi_2}(a_1,a_2,\cdots )$.
\end{rem}

\begin{rem}This brace structure is a generalization of twisted versions of Gerstenhaber products in \cite{HT10,SW12}.
\end{rem}

\begin{prop}[$G$-twisted higher pre-Jacobi identities]\label{thm:twisted high pre-j identity}
Given $\phi\in C^\bullet(A,Ag)$, $\phi_i\in  C^\bullet(A,Ag_i)$ with $ 1\leqslant i\leqslant n$ and $\psi_j\in C^\bullet(A,Ah_j)$ with $1\leqslant j\leqslant m$, we have the following $G$-twisted version of higher pre-Jacobi identities
\begin{align}\label{eq:twisted high pre-j identity}
&\phi \lrf{\phi_1,\phi_2 \cdots \phi_n}\lrf{\psi_1,\psi_2 \cdots \psi_m}\nonumber\\
=&\sum_{\bm{j}\in\mathcal{J}}  (-1)^{\sum\limits_{i=1}^n\xi_i(\abs{\phi_i}-1)} \phi\Big\{\psi_1,\cdots \psi_{j_1-1},\phi_1\lrf{\psi_{j_1}, \cdots \psi_{j'_1-1}}, g_1^*(\psi_{j'_1}),\cdots\nonumber\\
&\cdots (g_{i-1} \cdots  g_1)^*(\psi_{j_i-1}),\phi_i\lrf{(g_{i-1}\cdots g_1)^*(\psi_{j_i}), \cdots (g_{i-1}\cdots g_1)^*(\psi_{j'_i-1})}, \cdots\nonumber\\
&\cdots \phi_n\lrf{(g_{n-1}\cdots g_1)^*(\psi_{j_n}), \cdots (g_{n-1}\cdots g_1)^*(\psi_{j_n-1})}, \cdots(g_{n}\cdots g_1)^*(\psi_m)\Big\},
\end{align}
where
\begin{equation*}
\mathcal{J} \coloneqq \lrf{\bm{j} = ( j_1 \leqslant j'_1 < \cdots < j_n\leqslant j'_n)\mid 1 \leqslant j_1, j'_n\leqslant m \text{ and } \forall 1 \leqslant i < n, j'_{i} \leqslant \abs{\phi_i} + j_{i} },
\end{equation*}
and
\begin{equation*}
\xi_i \coloneqq \abs{\psi_1}+\abs{\psi_2}+\cdots \abs{\psi_{j_i-1}}-j_i+1.
\end{equation*}
\end{prop}

\begin{proof} We illustrate by an example when $n=m=2$. The left hand side of (\ref{eq:twisted high pre-j identity}) consists of terms like
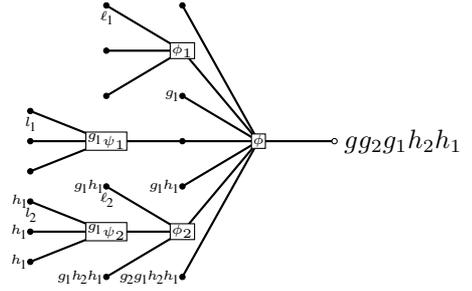
\[\tikz[scale=0.10,baseline=-5]{
\draw[thick] (10,0) node[right] {$gg_2g_1h_2h_1$} -- (0,0) -- ++ (-10,18);
\draw[thick] (0,0) -- ++ (-10,6) node[left,font=\tiny,inner sep = 1pt] {$g_{\sm{1}}$};
\draw[thick] (0,0) -- ++ (-10,12);
\draw[thick] (0,0) -- ++ (-10,-6)node[left,font=\tiny,inner sep = 1pt] {$g_{\sm{1}}h_{\sm{1}}$};
\draw[thick] (0,0) -- ++ (-10,0) -- ++ (-10,0);
\draw[thick] (0,0) -- ++ (-10,-18) node[left,font=\tiny,inner sep = 1pt] {$g_{\sm{2}}g_{\sm{1}}h_{\sm{2}}h_{\sm{1}}$};
\draw[thick] (0,0) -- ++ (-10,-12);
\draw[thick] (-10,12) -- ++ (-10,-6);
\draw[thick] (-10,12) -- ++ (-10,0);
\draw[thick] (-10,12) -- ++ (-10,6) node[font=\tiny,below,outer sep = 0pt,inner sep = 2pt ] {$\ell_{\sm{1}}$};
\draw[thick] (-10,-12) -- ++ (-10,-6) node[left,font=\tiny,inner sep = 1pt] {$g_{\sm{1}}h_{\sm{2}}h_{\sm{1}}$};
\draw[thick] (-10,-12) -- ++ (-10,0) -- ++ (-10,0) node[left,font=\tiny,inner sep = 1pt] {$h_{\sm{1}}$};
\draw[thick] (-10,-12) -- ++ (-10,6) node[font=\tiny,below,inner sep = 2pt] {$\ell_{\sm{2}}$};
\draw[thick] (-20,-12) -- ++ (-10,4) node[font=\tiny,below,inner sep = 2pt] {$l_{\sm{2}}$};
\draw[thick] (-20,-12) -- ++ (-10,-4) node[left,font=\tiny,inner sep = 1pt] {$h_{\sm{1}}$};
\draw[thick] (-20,0) -- ++ (-10,-4);
\draw[thick] (-20,0) -- ++ (-10,0);
\draw[thick] (-20,0) -- ++ (-10,4) node[font=\tiny,below,,inner sep = 2pt] {$l_{\sm{1}}$};
\node[left,font=\tiny,inner sep = 1pt] at (-20,-6) {$g_{\sm{1}}h_{\sm{1}}$};
\node[left,font=\tiny,inner sep = 1pt] at (-30,-8) {$h_{\sm{1}}$};
\filldraw[fill=black] (-10,18) circle (10pt);
\filldraw[fill=black] (-10,6) circle (10pt);
\filldraw[fill=black] (-10,0) circle (10pt);
\filldraw[fill=black] (-10,-6) circle (10pt);
\filldraw[fill=black] (-10,-18) circle (10pt);
\filldraw[fill=black] (-20,18) circle (10pt);
\filldraw[fill=black] (-20,12) circle (10pt);
\filldraw[fill=black] (-20,6) circle (10pt);
\filldraw[fill=black] (-20,-18) circle (10pt);
\filldraw[fill=black] (-20,-6) circle (10pt);
\filldraw[fill=black] (-30,4) circle (10pt);
\filldraw[fill=black] (-30,0) circle (10pt);
\filldraw[fill=black] (-30,-4) circle (10pt);
\filldraw[fill=black] (-30,-8) circle (10pt);
\filldraw[fill=black] (-30,-12) circle (10pt);
\filldraw[fill=black] (-30,-16) circle (10pt);
\node[font=\tiny,rectangle,draw=black,fill=white,inner sep=0.5] at (0,0) {$\phi$};
\node[font=\tiny,rectangle,draw=black,fill=white,inner sep=0.5] at (-10,12) {$\phi_1$};
\node[font=\tiny,rectangle,draw=black,fill=white,inner sep=0.5] at (-10,-12) {$\phi_2$};
\node[font=\tiny,rectangle,draw=black,fill=white,inner sep=0.5] at (-20,0) {$\leftidx{^{g_{\sm{1}}}}\psi_1$};
\node[font=\tiny,rectangle,draw=black,fill=white,inner sep=0.5] at (-20,-12) {$\leftidx{^{g_{\sm{1}}}}\psi_2$};
\filldraw[fill=white,draw=black] (10,0) circle (10pt);
}.\]
The right half side of (\ref{eq:twisted high pre-j identity}) consists of terms like
\[\tikz[scale=0.10,baseline=-5]{
\draw[thick] (10,0) node[right] {$gg_2g_1h_2h_1$} -- (0,0) -- ++ (-10,18);
\draw[thick] (0,0) -- ++ (-10,6) node[left,font=\tiny,inner sep = 1pt] {$g_{\sm{1}}$};
\draw[thick] (0,0) -- ++ (-10,12);
\draw[thick] (0,0) -- ++ (-10,-6)node[left,font=\tiny,inner sep = 1pt] {$g_{\sm{1}}h_{\sm{1}}g_{\sm{1}}^{\sm{-\!1}}g_{\sm{1}}$};
\draw[thick] (0,0) -- ++ (-10,0);
\draw[thick] (0,0) -- ++ (-10,-18) node[anchor = north east,font=\tiny,inner sep = 0pt,outer sep = 0pt] {$g_{\sm{2}}g_{\sm{1}}h_{\sm{2}}g_{\sm{1}}^{\sm{-\!1}}g_{\sm{1}}h_{\sm{1}}$};
\draw[thick] (0,0) -- ++ (-10,-12);
\draw[thick] (-10,12) -- ++ (-10,-6);
\draw[thick] (-10,12) -- ++ (-10,0);
\draw[thick] (-10,12) -- ++ (-10,6) node[font=\tiny,below,outer sep = 0pt,inner sep = 2pt ] {$\ell_{\sm{1}}$};
\draw[thick] (-10,-12) -- ++ (-10,-5) node[left,font=\tiny,inner sep = 1pt] {$g_{\sm{1}}h_{\sm{2}}g_{\sm{1}}^{\sm{-\!1}}g_{\sm{1}}h_{\sm{1}}$};
\draw[thick] (-10,-12) -- ++ (-10,0) -- ++ (-10,0) node[left,font=\tiny,inner sep = 1pt] {$g_{\sm{1}}h_{\sm{1}}g_{\sm{1}}^{\sm{-\!1}}g_{\sm{1}}$};
\draw[thick] (-10,-12) -- ++ (-10,5) node[font=\tiny,below,inner sep = 2pt] {$\ell_{\sm{2}}$};
\draw[thick] (-20,-12) -- ++ (-10,4) node[font=\tiny,below,inner sep = 2pt] {$l_{\sm{2}}$};
\draw[thick] (-20,-12) -- ++ (-10,-4) node[left,font=\tiny,inner sep = 1pt] {$g_{\sm{1}}h_{\sm{1}}g_{\sm{1}}^{\sm{-\!1}}g_{\sm{1}}$};
\draw[thick] (-10,0) -- ++ (-10,-4) node[left,font=\tiny,inner sep = 1pt] {$g_{\sm{1}}$};
\draw[thick] (-10,0) -- ++ (-10,0) node[left,font=\tiny,inner sep = 1pt] {$g_{\sm{1}}$};
\draw[thick] (-10,0) -- ++ (-10,4) node[font=\tiny,below,,inner sep = 2pt] {$l_{\sm{1}}$};
\node[left,font=\tiny,inner sep = 1pt] at (-20,-7) {$g_{\sm{1}}h_{\sm{1}}g_{\sm{1}}^{\sm{-\!1}}g_{\sm{1}}$};
\node[left,font=\tiny,inner sep = 1pt] at (-30,-8) {$g_{\sm{1}}h_{\sm{1}}g_{\sm{1}}^{\sm{-\!1}}g_{\sm{1}}$};
\node[left,font=\tiny,inner sep = 1pt] at (-20,4) {$g_{\sm{1}}$};
\filldraw[fill=black] (-10,18) circle (10pt);
\filldraw[fill=black] (-10,6) circle (10pt);
\filldraw[fill=black] (-10,0) circle (10pt);
\filldraw[fill=black] (-10,-6) circle (10pt);
\filldraw[fill=black] (-10,-18) circle (10pt);
\filldraw[fill=black] (-20,18) circle (10pt);
\filldraw[fill=black] (-20,12) circle (10pt);
\filldraw[fill=black] (-20,6) circle (10pt);
\filldraw[fill=black] (-20,-17) circle (10pt);
\filldraw[fill=black] (-20,-7) circle (10pt);
\filldraw[fill=black] (-20,4) circle (10pt);
\filldraw[fill=black] (-20,0) circle (10pt);
\filldraw[fill=black] (-20,-4) circle (10pt);
\filldraw[fill=black] (-30,-8) circle (10pt);
\filldraw[fill=black] (-30,-12) circle (10pt);
\filldraw[fill=black] (-30,-16) circle (10pt);
\node[font=\tiny,rectangle,draw=black,fill=white,inner sep=0.5] at (0,0) {$\phi$};
\node[font=\tiny,rectangle,draw=black,fill=white,inner sep=0.5] at (-10,12) {$\phi_1$};
\node[font=\tiny,rectangle,draw=black,fill=white,inner sep=0.5] at (-10,-12) {$\phi_2$};
\node[font=\tiny,rectangle,draw=black,fill=white,inner sep=0.5] at (-10,0) {$g_{\sm{1}}^*(\psi_1)$};
\node[font=\tiny,rectangle,draw=black,fill=white,inner sep=0.5] at (-20,-12) {$g_{\sm{1}}^*(\psi_2)$};
\filldraw[fill=white,draw=black] (10,0) circle (10pt);
}.\]
By (\ref{eq: group action on twisted cochains}), we see (\ref{eq:twisted high pre-j identity}) hold.
\end{proof}

\begin{lem}\label{lem: twisted brace structures on twisted cochain is G-equivariant}
The twisted brace structures are $G$-equivariant with respect to the $G$-action defined in (\ref{eq: group action on twisted cochains}), i.e., $\forall h\in G$ and $\phi,\phi_1,\phi_2,\cdots \in C^\bullet(A,A[G])$, we have
\begin{equation}\label{eq: twisted brace structures on twisted cochain is G-equivariant}
h^*\lr{\phi\lrf{\phi_1,\phi_2,\cdots}} = h^*(\phi)\lrf{h^*(\phi_1),h^*(\phi_2),\cdots}.
\end{equation}
\end{lem}

\begin{proof}We illustrate for the case $\phi\{\phi_1,\phi_2\}$. Terms in both sides are given by
\[\tikz[scale=0.1,baseline=-5]{
\draw[thick] (10,0) node[right] {$hgg_2g_1h^{-1}$} -- (0,0) -- ++ (-10,16) node[left,font=\tiny,inner sep = 1pt] {$h^{\sm{-1}}$};
\draw[thick] (0,0) -- ++ (-10,8);
\draw[thick] (0,0) -- ++ (-10,0) node[left,font=\tiny,inner sep = 1pt] {$g_{\sm{1}}h^{\sm{-1}}$};
\draw[thick] (0,0) -- ++ (-10,-8);
\draw[thick] (0,0) -- ++ (-10,-16) node[left,font=\tiny,inner sep = 1pt] {$g_{\sm{2}}g_{\sm{1}}h^{\sm{-1}}$};
\draw[thick] (-10,8) -- ++ (-10,-5) node[left,font=\tiny,inner sep = 1pt] {$h^{\sm{-1}}$};
\draw[thick] (-10,8) -- ++ (-10,0) node[left,font=\tiny,inner sep = 1pt] {$h^{\sm{-1}}$};
\draw[thick] (-10,8) -- ++ (-10,5) node[font=\tiny,below,outer sep = 0pt,inner sep = 2pt ] {$i_{\sm{1}}$};
\draw[thick] (-10,-8) -- ++ (-10,-5) node[left,font=\tiny,inner sep = 1pt] {$g_{\sm{1}}h^{\sm{-1}}$};
\draw[thick] (-10,-8) -- ++ (-10,0) node[left,font=\tiny,inner sep = 1pt] {$g_{\sm{1}}h^{\sm{-1}}$};
\draw[thick] (-10,-8) -- ++ (-10,5) node[font=\tiny,below,inner sep = 2pt] {$i_{\sm{2}}$};
\node[left,font=\tiny,inner sep = 1pt] at (-20,-3) {$g_{\sm{1}}h^{\sm{-1}}$};
\node[left,font=\tiny,inner sep = 1pt] at (-20,13) {$h^{\sm{-1}}$};
\filldraw[fill=black] (-10,16) circle (10pt);
\filldraw[fill=black] (-10,-16) circle (10pt);
\filldraw[fill=black] (-10,0) circle (10pt);
\filldraw[fill=black] (-20,-13) circle (10pt);
\filldraw[fill=black] (-20,-3) circle (10pt);
\filldraw[fill=black] (-20,-8) circle (10pt);
\filldraw[fill=black] (-20,13) circle (10pt);
\filldraw[fill=black] (-20,3) circle (10pt);
\filldraw[fill=black] (-20,8) circle (10pt);
\node[font=\tiny,rectangle,draw=black,fill=white,inner sep=0.5] at (0,0) {$\leftidx{^h}\phi$};
\node[font=\tiny,rectangle,draw=black,fill=white,inner sep=0.5] at (-10,8) {$\phi_1$};
\node[font=\tiny,rectangle,draw=black,fill=white,inner sep=0.5] at (-10,-8) {$\phi_2$};
\filldraw[fill=white,draw=black] (10,0) circle (10pt);
}=\tikz[scale=0.1,baseline=-5]{
\draw[thick] (10,0) node[right] {$hgh^{-1}hg_2h^{-1}hg_1h^{-1}$} -- (0,0) -- ++ (-10,16) node[left,font=\tiny,inner sep = 1pt] {$h^{\sm{-1}}$};
\draw[thick] (0,0) -- ++ (-10,8);
\draw[thick] (0,0) -- ++ (-10,0) node[left,font=\tiny,inner sep = 1pt] {$h^{\sm{-1}}hg_{\sm{1}}h^{\sm{-1}}$};
\draw[thick] (0,0) -- ++ (-10,-8);
\draw[thick] (0,0) -- ++ (-10,-16) node[left,font=\tiny,inner sep = 1pt] {$h^{\sm{-1}}hg_{\sm{2}}h^{\sm{-1}}hg_{\sm{1}}h^{\sm{-1}}$};
\draw[thick] (-10,8) -- ++ (-10,-5) node[left,font=\tiny,inner sep = 1pt] {$h^{\sm{-1}}$};
\draw[thick] (-10,8) -- ++ (-10,0) node[left,font=\tiny,inner sep = 1pt] {$h^{\sm{-1}}$};
\draw[thick] (-10,8) -- ++ (-10,5) node[font=\tiny,below,outer sep = 0pt,inner sep = 2pt ] {$i_{\sm{1}}$};
\draw[thick] (-10,-8) -- ++ (-10,-5) node[left,font=\tiny,inner sep = 1pt] {$h^{\sm{-1}}hg_{\sm{1}}h^{\sm{-1}}$};
\draw[thick] (-10,-8) -- ++ (-10,0) node[left,font=\tiny,inner sep = 1pt] {$h^{\sm{-1}}hg_{\sm{1}}h^{\sm{-1}}$};
\draw[thick] (-10,-8) -- ++ (-10,5) node[font=\tiny,below,inner sep = 2pt] {$i_{\sm{2}}$};
\node[left,font=\tiny,inner sep = 1pt] at (-20,-3) {$h^{\sm{-1}}hg_{\sm{1}}h^{\sm{-1}}$};
\node[left,font=\tiny,inner sep = 1pt] at (-20,13) {$h^{\sm{-1}}$};
\filldraw[fill=black] (-10,16) circle (10pt);
\filldraw[fill=black] (-10,-16) circle (10pt);
\filldraw[fill=black] (-10,0) circle (10pt);
\filldraw[fill=black] (-20,-13) circle (10pt);
\filldraw[fill=black] (-20,-3) circle (10pt);
\filldraw[fill=black] (-20,-8) circle (10pt);
\filldraw[fill=black] (-20,13) circle (10pt);
\filldraw[fill=black] (-20,3) circle (10pt);
\filldraw[fill=black] (-20,8) circle (10pt);
\node[font=\tiny,rectangle,draw=black,fill=white,inner sep=0.5] at (0,0) {$\leftidx{^h}\phi$};
\node[font=\tiny,rectangle,draw=black,fill=white,inner sep=0.5] at (-10,8) {$\leftidx{^{h^{\sm{-1}}h}}\phi_1$};
\node[font=\tiny,rectangle,draw=black,fill=white,inner sep=0.5] at (-10,-8) {$\leftidx{^{h^{\sm{-1}}h}}\phi_2$};
\filldraw[fill=white,draw=black] (10,0) circle (10pt);
}.\]
\end{proof}

\boldmath
\subsubsection*{$G$-twisted Hochschild differential and cup product}
\unboldmath
Now we introduce differentials on $G$-twisted Hochschild cochains. We will always identify
$$
  C^\bullet(A,A)\cong C^\bullet(A, A\e) \hookrightarrow  C^\bullet(A ,A[G]) 
$$
as the identity sector of $C^\bullet(A ,A[G])$. In particular, we identify the product and the curving of $A$
$$
  m_2\in C^2(A,A\e), \quad m_0\in C^0(A, A\e),
$$
as $G$-invariant cochains in $C^\bullet(A,A[G])$.
\begin{defn} We define the \emph{Hochschild differential} $\p_H$ and the \emph{curving differential} $\df_W$ on $G$-twisted Hochschild cochains by
$$
 \p_H(\phi) = (-1)^{\abs{\phi}-1}m_2\{\phi\}-\phi\{m_2\}, \quad \df_W(\phi) = \phi\{m_0\}, \quad \phi\in C^\bullet(A,A[G]).
$$
We also denote
$$
  \p_H^W = \p_H+ \df_W.
$$
\end{defn}

\begin{defn}We define the \emph{cup product} on $G$-twisted Hochschild cochains by
$$
   \phi_1\cup \phi_2=(-1)^{|\phi_1|(\abs{\phi_2}-1)}m_2\{\phi_1, g_1^*\phi_2\}, \quad \phi_i\in C^\bullet(A, Ag_i).
$$
Note that  this is the usual cup product that arises from the algebra structure on $A[G]$ by Definition \ref{def: twisted brace structure}.

\end{defn}

\begin{lem}\label{lem: twisted ginfty algebra} The following $G$-twisted version of higher homotopy identities holds,
\begin{align}
&\p_H(\phi)\lrf{\phi_1,\phi_2,\cdots, \phi_n} \nonumber \\
=&~ (-1)^{\xi_n} \p_H \big(\phi\lrf{\phi_1,\cdots, \phi_n}\big) - \sum_{k=1}^n(-1)^{\xi_k}\phi\lrf{\phi_1,\cdots, \phi_{k-1},\p_H(\phi_k),\phi_{k+1},\cdots, \phi_n}\nonumber\\
&+(-1)^{\abs{\phi_1}\xi'_1} \phi_1 \cup (g_1^{-1})^*\phi\lrf{\phi_2,\phi_3,\cdots, \phi_n}\nonumber\\
&+\sum_{k=1}^{n-1} (-1)^{\xi_k+\abs{\phi_k}\abs{\phi_{k+1}}} \phi\lrf{ \phi_1,\cdots, \phi_k \cup (g_i^{-1})^*\phi_{k+1},\cdots,
\phi_{n}}\nonumber\\
&- (-1)^{\abs{\phi}\abs{\phi_n}+\xi_{n-1}\xi'_{n-1}} \phi\lrf{\phi_1,\phi_2,\cdots, \phi_{n-1}} \cup ((g_{n-1}\cdots g_1)^{-1})^*\phi_n,
\end{align}
with
\begin{equation*}
\begin{dcases}\xi_k \coloneqq \abs{\phi_1}+\abs{\phi_2}+\cdots+ \abs{\phi_k}-k,\\
\xi'_k \coloneqq \abs{\phi_{k+1}}+\abs{\phi_{k+2}}+\cdots+ \abs{\phi_n}-(n-k).
\end{dcases}
\end{equation*}
Here $\phi\in C^\bullet(A,Ag),\phi_i\in C^\bullet(A,Ag_i)$. The same is true if $\p_H$ is replaced by $\p_H^W$.
\end{lem}

\begin{proof} The lemma follows from the twisted higher pre-Jacobi identities and the $G$-invariance of $m_2, m_0$.

\end{proof}

\begin{coro}
The cup product $\cup$ on $C^\bullet(A,A[G])$ satisfies the following twisted commutativity up to homotopy: for any $\phi\in C^\bullet(A,Ag),\phi_1\in C^\bullet(A,Ag_1)$
\begin{align}\label{eq: twisted commutative on cochains}
(-1)^{\abs{\phi_1}}\p_H(\phi\{\phi_1\}) + (\p_H\phi)\{\phi_1\} - (-1)^{|\phi_1|} \phi\{\p_H\phi_1\} = \phi_1 \cup (g_1^{-1})^*\phi-(-1)^{|\phi||\phi_1|}\phi \cup \phi_1.
\end{align}
\end{coro}

\begin{proof}This follows from Lemma \ref{lem: twisted ginfty algebra} in the case $n=1$.
\end{proof}

\begin{coro} The triple $\{C^\bullet(A,A[G]), \cup, \p_H\}$ defines a differential graded algebra. If we replace $\p_H$ with $\p_H^W$, we get a $\Z/2\Z$-graded differential graded algebra.

\end{coro}
\begin{proof} Let $\phi\in C^\bullet(A,Ag),\phi_i\in C^\bullet(A,Ag_i)$. Lemma \ref{lem: twisted ginfty algebra} implies
\begin{align*}
  &\p_H(\phi\{\phi_1,\phi_2\}) - (-1)^{|\phi_2|-1}\phi\{\p_H\phi_1,\phi_2\} - \phi\{\phi_1,\p_H\phi_2\} - (-1)^{\abs{\phi_1} + \abs{\phi_2}} \p_H(\phi)\{\phi_1,\phi_2\}\nonumber \\
   = & (-1)^{(|\phi_1|-1)|\phi_2|}\left( \phi\{\phi_1\cup (g_1^{-1})^*\phi_2\} - (-1)^{(|\phi|-1)\abs{\phi_2}}\phi\{\phi_1\}\cup (g_1^{-1})^*\phi_2-\phi_1\cup (g_1^{-1})^*\phi\{\phi_2\} \right).
\end{align*}
Set $\phi=m_2$ and use $m_2\{m_2\}=0$, we find
\begin{align}\label{eq: dga structure}
\p_H(\phi_1\cup \phi_2)=\p_H(\phi_1)\cup \phi_2+(-1)^{|\phi_1|}\phi_1\cup \p_H(\phi_2).
\end{align}
The proof for $\p_H^W$ is similar.
\end{proof}

\begin{defn}\label{defn-Hochschild-cohomology} Let
\begin{align*}
   \mathsf{HH}(A,A[G])= \mathsf{H}(C^\bullet(A,A[G]), \p_H), \quad    \mathsf{HH}_c(A,A[G])= \mathsf{H}(C_c^\bullet(A,A[G]), \p_H)
\end{align*}
denote the $G$-twisted Hochschild cohomologies and similarly for the curved case
\begin{align*}
   \mathsf{HH}(A_W,A_W[G])= \mathsf{H}(C^\bullet(A,A[G]), \p_H+\df_W), \quad    \mathsf{HH}_c(A_W,A_W[G])= \mathsf{H}(C_c^\bullet(A,A[G]), \p_H+\df_W).
\end{align*}
All the above cohomologies carry a natural $G$-action induced by \eqref{eq: group action on twisted cochains}.
\end{defn}

\begin{thm}\label{thm: dga on G-twisted Hochschild} The cup product $\cup$ defines $\Z$-graded algebras on $\mathsf{HH}(A,A[G]), \mathsf{HH}_c(A,A[G])$  and $\Z/2\Z$-graded algebras on $\mathsf{HH}(A_W,A_W[G]),\mathsf{HH}_c(A_W,A_W[G])$ satisfying the twisted commutativity relation
$$
    [\phi_1]\cup [\phi_2]=(-1)^{|\phi_1||\phi_2|}[\phi_2]\cup (g_2^{-1})^*[\phi_1], \quad \phi_i\in C^\bullet(A,Ag_i).
$$
Moreover, their $G$-invariant subspaces, $\mathsf{HH}(A,A[G])^G$, $\mathsf{HH}_c(A,A[G])^G$, $\mathsf{HH}(A_W,A_W[G])^G$ and\newline $\mathsf{HH}_c(A_W,A_W[G])^G$, inherit natural Gerstenhaber algebra structures.
\end{thm}

\begin{proof} The induced cup product on cohomologies and twisted commutativity follow from \eqref{eq: twisted commutative on cochains} \eqref{eq: dga structure}. To see Gerstenhaber algebra structures on $G$-invariant cohomologies, we consider for example
$$
 \mathsf{HH}(A,A[G])^G= \mathsf{H}(C^\bullet(A,A[G])^G, \p_H),
$$
where $C^\bullet(A,A[G])^G$ are $G$-invariant cochains. On $C^\bullet(A,A[G])^G$, the $G$-twisted higher pre-Jacobi identities \eqref{eq:twisted high pre-j identity} reduce to the same form as the untwisted one \eqref{untwisted higher homotopy}, from which we deduce the Gerstenhaber algebra structures by \eqref{eq: Gerstenhaber bracket} \eqref{Ger2} \eqref{Ger3}\eqref{Ger4}\eqref{Ger1}.

\end{proof}

\begin{rem}
  The $G$-twisted commutativity of the cup product is also obtained by Shklyarov in \cite{S17}.
\end{rem}}

\begin{prop}\label{lem:subgp lem} Let $H$ be a subgroup of $G$, then the inclusion
\begin{equation*}C^\bullet({A},A[H])\hookrightarrow C^\bullet({A},A[G])\end{equation*}
induce an embedding of $\Z/2\Z$-graded algebras
\begin{equation*}
 \mathsf{HH}_c(A_W,A_W[H])\hookrightarrow \mathsf{HH}_c(A_W,A_W[G]).\end{equation*}
 The same is true if we consider $\mathsf{HH}(A,A[H]), \mathsf{HH}(A_W,A_W[H])$.
\end{prop}

\begin{proof} It is easy to see that $\mathsf{HH}_c(A_W,A_W[H])$ is a $\Z/2\Z$-graded subspace of $\mathsf{HH}_c(A_W,A_W[G])$  and closed under the cup product.

\end{proof}

\subsection{A comparison between Gerstenhaber algebras}
In this subsection, $G$ will be a finite group. In this case,  \c{S}tefan \cite{St95} proved that there is an isomorphism as graded vector spaces between $\hch(A[G])$ and $\hch(A,A[G])^G$ (see also \cite{DE05}). Using $G$-twisted brace structures, we extend this to an isomorphism between Gerstenhaber algebras.

Consider the following map \cite{B03,HT10,SW12,CT13}
\begin{align}
\Psi \colon  C^\bullet(A,A[G])^G \rightarrow&~ C^\bullet(A[G],A[G]),\nonumber\\
\phi \quad \mapsto & \quad\Psi(\phi),\nonumber
\end{align}
where for $a_1g_1,a_2g_2,\cdots a_pg_p\in A[G]$,
\begin{equation}\label{defn-Psi-G}
\Psi(\phi)(a_1g_1 \otimes \cdots a_2g_2 \otimes \cdots \otimes a_{p}g_{p})=\phi( a_1\otimes \leftidx{^{g_1}a_2}\otimes\cdots\otimes  \leftidx{^{g_1g_2\cdots g_{p-1}}}a_p) g_1g_2\cdots g_p.
\end{equation}

\begin{lem}\label{lem: psi is ginfty morphism}
$\Psi$ preserves brace structures: for any $\phi, \phi_i \in C^\bullet(A,A[G])^G$,
\begin{equation}\label{eq: Psi couples braces}
\Psi(\phi)\lrf{\Psi(\phi_1),\Psi(\phi_2),\cdots} = \Psi\big(\phi\lrf{\phi_1,\phi_2,\cdots}\big).
\end{equation}
\end{lem}
\begin{proof}
Let $\phi_k$ be a $G$-invariant $G$-twisted cochain. We write
\begin{equation*}
\phi_k = \sum_{g\in G}\phi^g_kg, \quad \phi^g_k  \in C^\bullet(A,A).
\end{equation*}
$G$-invariance implies that
\begin{equation*}
h^*(\phi_k^g) = \phi_k^{hgh^{-1}}.
\end{equation*}

Given input $a_1h_1\otimes a_2h_2\otimes \cdots$, the left half side of (\ref{eq: Psi couples braces}) gives a sum of terms like
\begin{align*}
&\pm \phi^g \Big(a_1\otimes \leftidx{^{h_1}}a_2\otimes \cdots \otimes \leftidx{^{\hslash_{i_1}}}\phi^{\hslash_{i_1}^{-1}g_1\hslash_{i_1}}_1\big(a_{i_1}\otimes\leftidx{^{h_{i_1}}} a_{i_1+1}\otimes \cdots \big)\otimes \cdots\nonumber\\ &\qquad\cdots \leftidx{^{g_{k-1} g_{k-2} \cdots g_1 \hslash_{i_k} }}\phi_k^{\lr{g_{k-1}\cdots g_1 \hslash_{i_k} }^{-1}g_k{g_{k-1} \cdots g_1  \hslash_{i_k} }}(a_{i_k}\otimes \cdots)\otimes \cdots \Big)g\cdots g_2g_1h_1h_2\cdots\nonumber\\
=&\pm \phi^g \Big(a_1\otimes \leftidx{^{h_1}}a_2\otimes \cdots \otimes \phi^{g_1}_1\big(\leftidx{^{\hslash_{i_1}}}a_{i_1}\otimes\leftidx{^{\hslash_{i_1}h_{i_1}}} a_{i_1+1}\otimes \cdots \big)\otimes \cdots\nonumber\\ &\qquad\quad\cdots \phi_k^{g_k}(\leftidx{^{g_{k-1} g_{k-2} \cdots g_1 \hslash_{i_k} }}a_{i_k}\otimes \cdots)\otimes \cdots \Big)g\cdots g_2g_1h_1h_2\cdots,
\end{align*}
where $\hslash_i \coloneqq h_1h_2\cdots h_i$. The right half side of (\ref{eq: Psi couples braces}) gives a sum of terms like
\begin{align*}
&\pm \phi^g \Big(a_1\otimes \leftidx{^{h_1}}a_2\otimes \cdots \otimes \phi^{g_1}_1\big(\leftidx{^{\hslash_{i_1}}}a_{i_1}\otimes\leftidx{^{\hslash_{i_1}h_{i_1}}} a_{i_1+1}\otimes \cdots \big)\otimes \cdots\nonumber\\ &\qquad\quad\cdots \phi_k^{g_k}(\leftidx{^{g_{k-1} g_{k-2} \cdots g_1 \hslash_{i_k} }}a_{i_k}\otimes \cdots)\otimes \cdots \Big)g\cdots g_2g_1h_1h_2\cdots.
\end{align*}
The Lemma follows.
\end{proof}

\begin{thm}\label{thm: comparison of Hochschild}Let $G$ be a finite group.  Then $\Psi$ defined by \eqref{defn-Psi-G} induces an isomorphism between $\Z$-graded Gerstenhaber algebras
$$
  \Psi\colon \hch(A,A[G])^G\to  \hch(A[G]).
$$
\end{thm}
\begin{proof} It is easy to see that $\Psi$ is compatible with Hochschild differential $\p_H$. By \cite{B03}, $\Psi$ induces a $\Z$-graded vector space isomorphism between $\hch(A,A[G])^G$ and  $\hch(A[G])$. The theorem is now a formal consequence of Lemma \ref{lem: psi is ginfty morphism}.
\end{proof}

Now we consider the curved case.

\begin{thm}\label{thm: comparison of curved Hochschild} Let $G$ be a finite group, $(A,W,G)$ be a $G$-twisted curved algebra.  Then $\Psi$ defined by \eqref{defn-Psi-G} induces an isomorphism between $\Z/2\Z$-graded Gerstenhaber algebras
$$
  \Psi\colon \mathsf{HH}_c(A_W,A_W[G])^G\to  \mathsf{HH}_c(A_W[G]).
$$
\end{thm}
\begin{proof} By Lemma \ref{lem: psi is ginfty morphism}, we only need to prove that
$$
  \Psi\colon C^\bullet_c(A,A[G])^G\to C^\bullet_c(A[G])
$$
induce a vector space isomorphism on $(\p_H+\df_W)$-cohomology. Introduce a formal variable $u$ of degree $2$ and we extend $\Psi$ to a cochain map between $\Z$-graded complex
\begin{align}\label{mixed-complex}
\Psi\colon (C^\bullet_c(A,A[G])^G[u,u^{-1}], \p_H+u \df_W)\to (C^\bullet_c(A[G], A[G])[u,u^{-1}], \p_H+u \df_W).
\end{align}
For any $k\in \Z$, there are natural identifications
\begin{align*}
   \mathsf{HH}^{\bar 0}_c(A_W,A_W[G])^G&= \mathsf{H}^{2k}(C^\bullet_c(A,A[G])^G[u], \p_H+u \df_W)\\
  \mathsf{HH}^{\bar 1}_c(A_W,A_W[G])^G&= \mathsf{H}^{2k+1}(C^\bullet_c(A,A[G])^G[u], \p_H+u \df_W)
\end{align*}
and similarly for $\mathsf{HH}_c(A_W[G])$. It suffices to show that \eqref{mixed-complex} is a quasi-isomorphism.

Consider the decreasing filtration $\cdots \subset F^{p+1}\subset F^{p}\subset F^{p-1}\subset\cdots$ where
$$
    F^p = u^p C^\bullet_c[u].
$$
Here $C^\bullet_c$ denotes $C^\bullet_c(A,A[G])^G$ or $C^\bullet_c(A[G],A[G])$. Since we work with compact type complex, this filtration is exhaustive (see \cite{W95} for details) 
and there is an associated convergent spectral sequence. $\Psi$ in \eqref{mixed-complex} induces an isomorphism between $E_1$-pages which are computed by $\p_H$-cohomologies. It follows that \eqref{mixed-complex} defines a quasi-isomorphism.
\end{proof}


\section{Orbifold Landau-Ginzburg B-models}

In this section we study the orbifold Landau-Ginzburg model associated to a $G$-twisted curved algebra
$$
(A=\C[x_1,\cdots,x_N], W, G),
$$
where  $W$ is a weighted homogeneous  invertible polynomial and $G$ is a subgroup of $(\C^*)^N$ preserving $W$. We show that the cohomology $\mathsf{HH}_c(A_W,A_W[G])$ is a $G$-Frobenius algebra. In particular, the $G$-invariant subspace $\mathsf{HH}_c(A_W,A_W[G])^G$ has an induced Frobenius algebra structure. We give closed formulae 
for the cup product on $\mathsf{HH}_c(A_W,A_W[G])^G$ for all Fermat, Loop and Chain types.  This is computed via an explicit homotopy between Koszul resolution and bar resolution.

For convenience, we fix the following notations in this section.
\begin{defn} For $g\in G$, we let
$$
  \fix(g)=\{v\in \C^N\mid \leftidx{^g}v=v\}\subset \C^N
$$
be the fixed locus of $g$ and
$$
A_g=\C[\fix(g)]   
$$
denote polynomial functions on $ \fix(g)$. We write
$$
N_g=\dim{\C}\fix(g), \quad W_g=W|_{\fix(g)}\in A_g.
$$
We also denote $\bm{I}_g=\{i_1<i_2<\cdots< i_{N-N_g}\}$ where $\{x_{i_k}\}$'s are variables such that $\leftidx{^g}x_{i_k}\neq x_{i_k}$. $\bm{I}_g$ will be called the moving index of $g$.
\end{defn}

\begin{defn} Let $x_1^{\gamma_1}\cdots x_N^{\gamma_N}$ be a monomial in $A$, and $g\in G$. We define
  \begin{equation}\rho_{i}(g)\big(x_1^{\gamma_1}\cdots x_N^{\gamma_N}\big) = (\leftidx{^g}x_1)^{\gamma_1}\cdots (\leftidx{^g}x_{i-1})^{\gamma_{i-1}}x_{i}^{\gamma_i}\cdots x_N^{\gamma_N},\end{equation}
and the quantum differential operator
\begin{equation}\label{eq: q-diff operator}\p^g_{x_i}\left(x_1^{\gamma_1}\cdots x_N^{\gamma_N}\right)=\begin{dcases}
[\gamma]_{\lambda_i}~x_1^{\gamma_1}\cdots x_i^{\gamma_i-1}\cdots x_N^{\gamma_N}&\text{if } \gamma_i>0,\\
0&\text{else}.
\end{dcases}\end{equation}
Here $\lambda_i$ is the weight: $\leftidx{^g}x_i=\lambda_i  x_i$. $[\gamma]_\lambda$ ($\gamma\geqslant 1$) is defined by
\begin{equation}\label{eqn-quantum-number}
[\gamma]_{\lambda}= 1+\lambda+\lambda^2+\cdots +\lambda^{\gamma-1}.\end{equation}
Both $\rho_i(g)$ and $\p^g_{x_i}$ extend linearly to operators on $A$.
\end{defn}

\subsection{Bar resolution vs. Koszul resolution}\label{section-Koszul}
For simplicity, we work with the reduced bar resolution (or `normalized bar resolution', see \cite{L13}), 
\[\begin{tikzpicture}
\matrix (a) [matrix of math nodes, row sep=3.5em, column sep=2.0em, text height=1.7ex, text depth=0.9ex,font=\scriptsize]
{\cdots&A\otimes \overline{A}^{\otimes 3}\otimes A&A\otimes \overline{A}^{\otimes 2}\otimes A&A\otimes \overline{A}\otimes A&A\otimes A &A&0,\\};
\path[latex-]
(a-1-7) edge (a-1-6);
\path[latex-]
(a-1-6) edge node[above,font=\scriptsize] {$m_2$} (a-1-5);
\path[latex-]
(a-1-5) edge node[above,font=\scriptsize] {$\bdf$} (a-1-4);
\path[latex-]
(a-1-4) edge node[above,font=\scriptsize] {$\bdf$} (a-1-3);
\path[latex-]
(a-1-3) edge node[above,font=\scriptsize] {$\bdf$} (a-1-2);
\path[latex-]
(a-1-2) edge node[above,font=\scriptsize] {$\bdf$} (a-1-1);
\end{tikzpicture}\]
where $\overline{A} = A/\mathbb{C} = \mathbb{C}[x^i]/\mathbb{C}$. It gives a free resolution of $A$ as $A^e$-modules (equivalently $A$-bimodules) where $A^e=A\otimes A^{\mathsf{op}}$. Applying the functor $\hom_{A^e}(\cdot,A[G])$, we obtain the (reduced) G-twisted Hochschild cochain complex
$$
  \p_H\colon C^p(\overline{A},A[G])\to C^{p+1}(\overline{A},A[G]).
$$
Here  $C^p(\overline{A},A[G])= \hom(\overline{A}^{\otimes p}, A[G])$ and $\p_H$ is the Hochschild differential which is also well-defined on the reduced complex. We will denote
\begin{align}
C^\bullet(\overline{A},A[G]) =\mathop{\prod}\limits_{p=0}^\infty C^p(\overline{A},A[G]), \quad
C^\bullet_c(\overline{A},A[G]) = \mathop{\bigoplus}\limits_{p=0}^\infty C^p(\overline{A},A[G]).
\end{align}
The curving differential $\df_W$ is defined the same as before
$$
\df_W\colon C^{p}(\overline A, A[G])\to C^{p-1}(\overline A, A[G]).
$$

Since $A=\C[x_i]$, we have a simpler $A^e$-module resolution of $A$ via the Koszul resolution
\[\begin{tikzpicture}
\matrix (a) [matrix of math nodes, row sep=3em, column sep=1.6em, text height=1.5ex, text depth=0.25ex,font=\scriptsize]
{\cdots&\bigoplus\limits_{i<j<k}e^i e^j e^k A^e&\bigoplus\limits_{i<j}e^i e^jA^e&\bigoplus\limits_{i} e^iA^e&A^e&A& 0,\\};
\path[-latex]
(a-1-1) edge node[above,font=\scriptsize] {$\tilde{\bdf}$} (a-1-2);
\path[-latex]
(a-1-2) edge node[above,font=\scriptsize] {$\tilde{\bdf}$} (a-1-3);
\path[-latex]
(a-1-3) edge node[above,font=\scriptsize] {$\tilde{\bdf}$} (a-1-4);
\path[-latex]
(a-1-4) edge node[above,font=\scriptsize] {$\tilde{\bdf}$} (a-1-5);
\path[-latex]
(a-1-5) edge node[above,font=\scriptsize] {$m$} (a-1-6);
\path[-latex]
(a-1-6) edge (a-1-7);
\end{tikzpicture}\]
where $e^i$'s are odd variables: $\deg e^i=-1$ and $e^ie^j=-e^je^i$.  The Koszul differential $\tilde{\bdf}$ is a derivation of the ring $A^e[e^i]$ generated by
$$
\tilde{\bdf}(e^i) = x_i\otimes 1-1\otimes x_i\in A^e.
$$
 Applying the functor $\hom_{A^e}(\cdot,A[G])$, we get the $G$-twisted Koszul cochain complex
\begin{align}
    \p_K\colon K^\bullet(A,A[G])= A[e_i][G] \to K^{\bullet+1}(A,A[G]).
\end{align}
 Here $e_i$ is the dual basis of $e^i$ such that $\deg(e_i)=1, e_ie_j=-e_j e_i$. The Koszul differential $\p_H$ is
 \begin{align}\label{eqn: koszul differential}
 \p_K (\phi g)=\sum_{i=1}^N (x_i-\leftidx{^g}x_i) e_i \phi g, \quad \phi \in A[e_i],  g\in G.
 \end{align}

\begin{defn}
We define the curving differential on $G$-twisted Koszul cochains
$$
   \tilde{\df}_W\colon K^{p}(A,A[G])\to K^{p-1}(A,A[G])
$$
by
\begin{equation}\label{eq:horizontal diff on kus}
\tilde{\df}_W(a e_{i_1}\cdots e_{i_p} g)= \sum_{k=1}^p (-1)^{k-1}a\rho_{i_k}(g)\big(\p_{x_{i_k}}^g(W)\big)  e_{i_1}\cdots \widehat{e_{i_k}}\cdots e_{i_p}  g.
\end{equation}
Here $a\in A, g\in G$ and $1\leq i_1<\cdots <i_p\leq N$.
\end{defn}

In \cite{SW11}, Shepler and Witherspoon introduced chain maps between these two resolutions,
\[\begin{tikzpicture}
\matrix (a) [matrix of math nodes, row sep=3.5em, column sep=2.0em, text height=1.7ex, text depth=0.9ex,font=\scriptsize]
{ \cdots&\bigoplus\limits_{1\leqslant i<j<k\leqslant N}  e^i e^j e^k A^e&\bigoplus\limits_{1\leqslant i<j\leqslant N} e^i e^j A^e&\bigoplus\limits_{1\leqslant i\leqslant N} e^i A^e & A^e & A &0\\
\cdots&A\otimes \overline{A}^{\otimes 3}\otimes A&A\otimes \overline{A}^{\otimes 2}\otimes A&A\otimes \overline{A}\otimes A&A\otimes A & A &0,\\};
\path[latex-]
(a-1-7) edge (a-1-6);
\path[latex-]
(a-1-6) edge node[above,font=\scriptsize] {$m$} (a-1-5);
\path[latex-]
(a-1-5) edge node[above,font=\scriptsize] {$\tilde{\bdf}$} (a-1-4);
\path[latex-]
(a-1-4) edge node[above,font=\scriptsize] {$\tilde{\bdf}$} (a-1-3);
\path[latex-]
(a-1-3) edge node[above,font=\scriptsize] {$\tilde{\bdf}$} (a-1-2);
\path[latex-]
(a-1-2) edge node[above,font=\scriptsize] {$\tilde{\bdf}$} (a-1-1);
\path[latex-]
(a-2-7) edge (a-2-6);
\path[latex-]
(a-2-6) edge node[above,font=\scriptsize] {$m$} (a-2-5);
\path[latex-]
(a-2-6.97)  edge node[left,font=\scriptsize] {$\simeq$} (a-1-6.-97);
\path[-latex]
(a-2-6.83)  edge node[right,font=\scriptsize] {$\simeq$} (a-1-6.-83);
\path[latex-]
(a-2-5) edge node[above,font=\scriptsize] {$\bdf$} (a-2-4);
\path[latex-]
(a-2-5.97)  edge node[left,font=\scriptsize] {$\simeq$} (a-1-5.-97);
\path[-latex]
(a-2-5.83)  edge node[right,font=\scriptsize] {$\simeq$} (a-1-5.-83);
\path[latex-]
(a-2-4) edge node[above,font=\scriptsize] {$\bdf$} (a-2-3);
\path[latex-]
(a-2-4.97)  edge node[left,font=\scriptsize] {$\Phi_1$} (a-1-4.-97);
\path[-latex]
(a-2-4.83)  edge node[right,font=\scriptsize] {$\Upsilon_1$} (a-1-4.-83);
\path[latex-]
(a-2-3) edge node[above,font=\scriptsize] {$\bdf$} (a-2-2);
\path[latex-]
(a-2-3.97)  edge node[left,font=\scriptsize] {$\Phi_2$} (a-1-3.-97);
\path[-latex]
(a-2-3.83)  edge node[right,font=\scriptsize] {$\Upsilon_2$} (a-1-3.-83);
\path[latex-]
(a-2-2) edge node[above,font=\scriptsize] {$\bdf$} (a-2-1);
\path[latex-]
(a-2-2.97)  edge node[left,font=\scriptsize] {$\Phi_3$} (a-1-2.-97);
\path[-latex]
(a-2-2.83)  edge node[right,font=\scriptsize] {$\Upsilon_3$} (a-1-2.-83);
\end{tikzpicture}\]

\begin{itemize}
\item
The chain map $\Phi$ from Koszul resolution to bar resolution is given by
\begin{align}\label{definition-Phi}
\Phi_p\colon\quad
 e^{i_1}e^{i_2}\cdots e^{i_p}a\otimes b &\mapsto  \sum\limits_{\sigma\in S_p}(-1)^{|\sigma|}a\otimes x_{i_{\sigma(1)}}\otimes\cdots\otimes x_{i_{\sigma(p)}}\otimes b. \quad a, b\in A.
\end{align}
Let $\Phi^*$ denote the induced map on cochains
\begin{align}
\Phi^*_p\colon C^p(\overline{A},A[G])&\rightarrow K^p(A,A[G]),\\
 \phi\quad &\mapsto  \sum_{i_1<\cdots<i_p}\sum\limits_{\sigma\in S_p}(-1)^{|\sigma|}\phi\big(x_{i_{\sigma(1)}}\otimes\cdots\otimes x_{i_{\sigma(p)}}\big)e_{i_1}\cdot e_{i_p}.\nonumber
\end{align}

\item The chain map $\Upsilon$ from bar resolution to Koszul resolution is constructed in the following steps:
\begin{description}
  \item[Step 1] Let $a\otimes a_1\otimes \cdots \otimes a_p\otimes b\in A\otimes \overline{A}^{\otimes p}\otimes A$ where $a_k = \prod\limits_{i=1}^N x_{i}^{\gamma_i^{k}}$ ($\gamma_i^{k}\geqslant 0$).
  \item[Step 2] For each $1\leqslant i_1<i_2<\cdots<i_p\leqslant N$ denoted by $\bm{I}$, let us define $\mathcal{S}(\bm{I})$ to be the set of sequence $\bm{s}= (s_1,s_2,\cdots s_p)$ such that
      \[0\leqslant s_k < \gamma_{i_k}^k.\]
      Given $s\in \mathcal{S}(\bm{I})$, we define a splitting of each $a_k$ into two parts:
      \[\begin{dcases}
      a_{k,\bm{I},\bm{s}}^{2}& = x_{1}^{\gamma_{1}^{k}}\cdots x_{(i_k-1)}^{\gamma_{(i_k-1)}^k} x_{i_k}^{s_k},\\
      a_{k,\bm{I},\bm{s}}^{1}& = x_{i_k}^{\gamma_{i_k}^{k}-1-s_k} x_{(i_k+1)}^{\gamma_{(i_k+1)}^k}\cdots x_{N}^{\gamma_{N}^k}.
      \end{dcases}\]
  \item[Step 3] $\Upsilon$ is defined by
  \begin{equation}\label{notation-Upsilon}
  \Upsilon_p(a\otimes a_1\otimes \cdots \otimes a_p\otimes b) = \sum_{\bm{I},\bm{s}\in \mathcal{S}(\bm{I})}\lr{aa^1_1\cdots a^{1}_p}\otimes \lr{a^{2}_p\cdots a^{2}_1b} e^{i_1}e^{i_2}\cdots e^{i_p},\end{equation}
  where $a_k^1$ and $a_k^2$ are short for $a_{k,\bm{I},\bm{s}}^{1}$ and $a_{k,\bm{I},\bm{s}}^{2}$. The proof for $\Upsilon$ being a chain map can be found in the Appendix of \cite{SW11}.
\end{description}

Let $\Upsilon^*$ denote the induced  map on cochains
\begin{equation}
\Upsilon^*_p\colon K^p(A,A[G])\rightarrow C^p(\overline{A},A[G]).
\end{equation}

For $\psi=a e_{i_1}\cdots e_{i_p}g$ where $1\leqslant i_1<i_2<\cdots<i_p\leqslant N$ denoted by $\bm{I}, a\in A, g\in G$, we have
\begin{align}\label{eq:cochain map upsilon star}
\Upsilon^*_p(\psi)(a_1\otimes \cdots\otimes a_p)=&\psi\big(\Upsilon(1\otimes a_1\otimes \cdots\otimes a_p\otimes 1)\big)\nonumber\\
=& \sum_{\bm{s}\in\mathcal{S}(\bm{I})}a^1_{1}a^1_{2}\cdots a^1_{p}a\lr{\leftidx{^g}a^2_1}\lr{\leftidx{^g}a^2_2}\cdots \lr{\leftidx{^g}a^2_p}g\nonumber\\
=& \rho_{i_1}(g)(\p_{x_{i_1}}^g a_1)\rho_{i_2}(g)(\p_{x_{i_2}}^g a_2)\cdots \rho_{i_p}(g)(\p_{x_{i_p}}^g a_p)ag.
\end{align}
\end{itemize}

\begin{lem}\label{lem: Upsilon are cochain maps}
$\Upsilon^*$ is compatible with the curving differential
\begin{equation} \quad \df_W \circ \Upsilon^* = \Upsilon^* \circ \tilde{\df}_W.
\end{equation}
\end{lem}

\begin{proof} It can be checked directly.
\end{proof}

\begin{thm}[\cite{SW12}]\label{thm:bar and kus are quasi-isomorphic} The composition
\begin{equation}\label{eq:bar to kus to bar is identity }\Upsilon \circ \Phi = \id,\end{equation}
is the identity on Koszul resolution.
\end{thm}

This theorem implies that the other composition $\Phi\circ \Upsilon$ will be homotopic to the identity $\id$ on bar resolution. Let us describe such a homotopy $\mathsf{H}$ explicitly. It will allow us to compute various algebraic structures on Hochschild cohomology and orbifold Landau-Ginzburg models.
\begin{defn}
We define $\mathsf H_p\colon A\otimes \overline{A}^{\otimes p}\otimes A\to A\otimes \overline{A}^{\otimes (p+1)}\otimes A$ by
\begin{align}\label{eq:homotopy between resolutions}
     &\mathsf H_p(a_0\otimes a_1\otimes \cdots \otimes a_p\otimes a_{p+1})\nonumber\\
     =& \sum_{i=1}^{p+1} (-1)^{i}a_0\otimes \cdots \otimes a_{i-1}\otimes \Phi_{p-i+1}\circ\Upsilon_{p-i+1}(1\otimes a_i\otimes \cdots \otimes a_p\otimes a_{p+1}).
  \end{align}
\end{defn}
\begin{prop}\label{prop-homotopy} $\mathsf H$ gives a homotopy between $\id$ and $\Phi\circ \Upsilon$ on the bar resolution
\begin{align}
 \id-\Phi\circ \Upsilon= \bdf \circ \mathsf H+ \mathsf H\circ  \bdf.
\end{align}
\[\begin{tikzpicture}
\matrix (a) [matrix of math nodes, row sep=3.5em, column sep=2.0em, text height=1.7ex, text depth=0.9ex,font=\scriptsize]
{\cdots&A\otimes \overline{A}^{\otimes 3}\otimes A&A\otimes \overline{A}^{\otimes 2}\otimes A&A\otimes \overline{A}\otimes A&A\otimes A &0\\
\cdots&\bigoplus\limits_{1\leqslant i<j<k\leqslant N} e^i e^j e^k A^e&\bigoplus\limits_{1\leqslant i<j\leqslant N} e^i e^j A^e&\bigoplus\limits_{1\leqslant i\leqslant N} e^i A^e & A^e & 0\\
\cdots&A\otimes \overline{A}^{\otimes 3}\otimes A&A\otimes \overline{A}^{\otimes 2}\otimes A&A\otimes \overline{A}\otimes A&A\otimes A  &0,\\};
\path[latex-]
(a-2-6) edge (a-2-5);
\path[latex-]
(a-2-5) edge node[below,font=\scriptsize] {$\tilde{\bdf}$} (a-2-4);
\path[latex-]
(a-2-4) edge node[below,font=\scriptsize] {$\tilde{\bdf}$} (a-2-3);
\path[latex-]
(a-2-3) edge node[below,font=\scriptsize] {$\tilde{\bdf}$} (a-2-2);
\path[latex-]
(a-2-2) edge node[below,font=\scriptsize] {$\tilde{\bdf}$} (a-2-1);
\path[latex-]
(a-1-6) edge (a-1-5);
\path[latex-]
(a-1-5) edge node[below,font=\scriptsize] {$\bdf$} (a-1-4);
\path[latex-]
(a-1-4) edge node[below,font=\scriptsize] {$\bdf$} (a-1-3);
\path[latex-]
(a-1-3) edge node[below,font=\scriptsize] {$\bdf$} (a-1-2);
\path[latex-]
(a-1-2) edge node[below,font=\scriptsize] {$\bdf$} (a-1-1);
\path[latex-]
(a-3-6) edge (a-3-5);
\path[latex-]
(a-3-5) edge node[below,font=\scriptsize] {$\bdf$} (a-3-4);
\path[latex-]
(a-3-4) edge node[below,font=\scriptsize] {$\bdf$} (a-3-3);
\path[latex-]
(a-3-3) edge node[below,font=\scriptsize] {$\bdf$} (a-3-2);
\path[latex-]
(a-3-2) edge node[below,font=\scriptsize] {$\bdf$} (a-3-1);
\path[-]
(a-2-5.85)  edge (a-1-5.-85);
\path[-]
(a-2-5.95)  edge (a-1-5.-95);
\path[-]
(a-3-5.85)  edge (a-2-5.-85);
\path[-]
(a-3-5.95)  edge (a-2-5.-95);
\path[latex-]
(a-3-4)  edge node[left,font=\scriptsize] {$\Phi_1$} (a-2-4);
\path[latex-]
(a-2-4)  edge node[left,font=\scriptsize] {$\Upsilon_1$} (a-1-4);
\path[latex-]
(a-3-3)  edge node[left,font=\scriptsize] {$\Phi_2$} (a-2-3);
\path[latex-]
(a-2-3)  edge node[left,font=\scriptsize] {$\Upsilon_2$} (a-1-3);
\path[latex-]
(a-3-2)  edge node[left,font=\scriptsize] {$\Phi_3$} (a-2-2);
\path[latex-]
(a-2-2)  edge node[left,font=\scriptsize] {$\Upsilon_3$} (a-1-2);
\path[latex-,dashed,semithick]
(a-3-4.70)  edge node[above,outer sep= 7pt,font=\scriptsize] {$\mathsf H_0$} (a-1-5.-110);
\path[latex-,dashed,semithick]
(a-3-3.70)  edge node[above,outer sep= 7pt,font=\scriptsize] {$\mathsf H_1$} (a-1-4.-110);
\path[latex-,dashed,semithick]
(a-3-2.70)  edge node[above,outer sep= 7pt,font=\scriptsize] {$\mathsf H_2$} (a-1-3.-110);
\path[latex-,dashed,semithick]
(a-3-1.130)  edge node[above,outer sep= 7pt,font=\scriptsize] {$\mathsf H_3$} (a-1-2.-110);
\end{tikzpicture}\]
Dually,  $\mathsf H$ induces a homotopy $\mathsf H^*$ between $\id$ and $\Upsilon^*\circ \Phi^*$ on $G$-twisted cochains
\begin{align}\label{cochain-homotopy-eqn}
   \id -\Upsilon^*\circ \Phi^*=\mathsf H^*\circ \p_H+\p_H\circ \mathsf H^*   \colon C^\bullet(\overline{A}, A[G])\to C^\bullet(\overline{A},A[G]).
\end{align}
\end{prop}
\begin{proof} Let $s$ be the homotopy for bar resolution,
\[\begin{tikzpicture}
\matrix (a) [matrix of math nodes, row sep=3.5em, column sep=2.0em, text height=1.7ex, text depth=0.9ex,font=\scriptsize]
{\cdots&A\otimes \overline{A}^{\otimes 3}\otimes A&A\otimes \overline{A}^{\otimes 2}\otimes A&A\otimes \overline{A}\otimes A&A\otimes A &A&0,\\};
\path[latex-]
(a-1-7) edge (a-1-6);
\path[latex-]
(a-1-6) edge node[above,font=\scriptsize] {$m_2$} (a-1-5);
\path[latex-]
(a-1-5) edge node[above,font=\scriptsize] {$\bdf$} (a-1-4);
\path[latex-]
(a-1-4) edge node[above,font=\scriptsize] {$\bdf$} (a-1-3);
\path[latex-]
(a-1-3) edge node[above,font=\scriptsize] {$\bdf$} (a-1-2);
\path[latex-]
(a-1-2) edge node[above,font=\scriptsize] {$\bdf$} (a-1-1);
\path[latex-]
(a-1-5.-8) edge node[below,font=\scriptsize] {$s$} (a-1-6.-165);
\path[latex-]
(a-1-4.-6) edge node[below,font=\scriptsize] {$s$} (a-1-5.-172);
\path[latex-]
(a-1-3.-5) edge node[below,font=\scriptsize] {$s$} (a-1-4.-175);
\path[latex-]
(a-1-2.-5) edge node[below,font=\scriptsize] {$s$} (a-1-3.-175);
\path[latex-]
(a-1-1.-15) edge node[below,font=\scriptsize] {$s$} (a-1-2.-175);
\end{tikzpicture}\]
where $
  s(a_0\otimes \cdots \otimes a_k)= 1\otimes a_0\otimes \cdots \otimes a_k.
$ It satisfies
$$
   \bdf \circ s +s \circ \bdf = \id.
$$
Using the fact that $\Phi\circ \Upsilon$ commutes with $b$, we have
\begin{align*}
  &( \bdf \circ \mathsf H+ \mathsf H\circ  \bdf)(a_0\otimes a_1\otimes \cdots\otimes a_p\otimes a_{p+1})\\
  =&-\sum_{i=1}^{p+1} a_0\otimes \cdots \otimes \left(a_{i-1} \Phi\circ\Upsilon(1\otimes a_i\otimes \cdots \otimes a_p\otimes a_{p+1})\right)\\
  &+a_0\otimes \cdots \otimes a_p\otimes m_2(\Phi\circ \Upsilon(1\otimes a_{p+1}))+\sum_{i=1}^{p} a_0\otimes \cdots \otimes \Phi\circ \Upsilon(a_i\otimes \cdots\otimes a_{p+1})\\
  =&(\id-\Phi\circ\Upsilon)(a_0\otimes a_1\otimes \cdots\otimes a_p\otimes a_{p+1}).
\end{align*}
The last statement follows from the fact that $\mathsf H$ is an ${A}$-bimodule homomorphism.
\end{proof}

Therefore we have a homotopy retraction between two cochains
\[\begin{tikzpicture}
\matrix (a) [matrix of math nodes, row sep=3.5em, column sep=4em, text height=2ex, text depth=0.9ex]
{\big(C^{\bullet}(\overline{A},A[G]),\p_H\big) & \big(K^{\bullet}(A,A[G]),\p_K\big).\\};
\path[-latex]
(a-1-1.2) edge node[above,font=\scriptsize] {$\Phi^*$} (a-1-2.178);
\path[latex-]
(a-1-1.-2) edge node[below,font=\scriptsize] {$\Upsilon^*$} (a-1-2.-178);
\draw[-latex]
(a-1-1.-170)  arc (-40:-320:11pt);
\node[left,outer sep=20pt,font=\scriptsize] at (a-1-1.-180) {$\mathsf{H}^*$};
\end{tikzpicture}\]

\begin{lem}\label{lem: special homotopy}
The homotopy $\mathsf{H}^*$ satisfies
\begin{equation}\label{eq: special homotopy retraction}
\mathsf{H}^* \circ \Upsilon^* = 0,\quad \Phi^* \circ \mathsf{H}^* = 0, \quad\text{and } \quad\mathsf{H}^* \circ \mathsf{H}^* = 0.
\end{equation}
In other words, the triple $(\Upsilon^*,\Phi^*,\mathsf{H}^*)$ is a special homotopy retraction (see for example \cite{C04}). 
\end{lem}

\begin{proof} We have
\begin{align*}
& \Upsilon_{p+1} \circ \mathsf{H}_p(1 \otimes a_1 \otimes \cdots \otimes a_p \otimes 1)\\
=& \sum_{i=1}^{p+1} (-1)^{i} \Upsilon_{p+1}(1\otimes \cdots \otimes a_{i-1}\otimes \Phi_{p-i+1}\circ\Upsilon_{p-i+1}(1\otimes a_i\otimes \cdots \otimes a_p\otimes 1))\\
=&0,
\end{align*}
where the last equality follows from  \eqref{definition-Phi} and the ordering \eqref{notation-Upsilon}.

\begin{align*}
& \mathsf{H}\circ \Phi (1\otimes 1 e^{i_1}e^{i_2}\cdots e^{i_p})\\
=& \sum_{\sigma\in S_p}\sum_{i=1}^{p+1} (-1)^{i+\abs{\sigma}} 1\otimes \cdots \otimes x_{\sigma(i-1)}\otimes \Phi_{p-i+1}\circ\Upsilon_{p-i+1}(1\otimes x_{\sigma(i)}\otimes \cdots \otimes x_{\sigma(p)}\otimes 1),
\end{align*}
which is zero as a reduced chain since  $\Phi_{p-i+1}\circ\Upsilon_{p-i+1}(1\otimes x_{\sigma(i)}\otimes \cdots \otimes x_{\sigma(p)}\otimes 1)$ contributes $1$ at some middle position. Similarly,
\begin{align*}
& \mathsf{H} \circ \mathsf{H}(1 \otimes a_1 \otimes \cdots \otimes a_p \otimes 1)\\
=& \sum_{i=1}^{p+1} (-1)^{i} \mathsf{H}(1\otimes \cdots \otimes a_{i-1}\otimes \Phi_{p-i+1}\circ\Upsilon_{p-i+1}(1\otimes a_i\otimes \cdots \otimes a_p\otimes 1))
\end{align*}
is zero as a reduced chain.
\end{proof}

We will be interested in studying the $(\p_H+ \df_W)$-cohomology in terms of Koszul complex.  Viewing $\df_W$ as a small perturbation, standard homological perturbation lemma (together with Lemma \ref{lem: Upsilon are cochain maps} and Lemma \ref{lem: special homotopy}) allows us to construct a new homotopy retract $({\Upsilon}^*,\tilde{\Phi}^*,\tilde{\mathsf{H}}^*)$

\[\begin{tikzpicture}
\matrix (a) [matrix of math nodes, row sep=3.5em, column sep=4em, text height=2ex, text depth=0.9ex]
{\big(C^{\bullet}(\overline{A},A[G]),\p_H+\df_W\big) & \big(K^{\bullet}(A,A[G]),\p_K + \tilde{\df}_W\big).\\};
\path[-latex]
(a-1-1.2) edge node[above,font=\scriptsize] {$\tilde{\Phi}^*$} (a-1-2.178.5);
\path[latex-]
(a-1-1.-2) edge node[below,font=\scriptsize] {${\Upsilon}^*$} (a-1-2.-178.5);
\draw[-latex]
(a-1-1.-170)  arc (-40:-320:11pt);
\node[left,outer sep=20pt,font=\scriptsize] at (a-1-1.-180) {$\tilde{\mathsf{H}}^*$};
\end{tikzpicture}\]
Here the perturbed homotopy $\tilde{\mathsf{H}}^*$ and retract $\tilde{\Phi}^*$ are defined by
\begin{equation}\label{def: perturbed operator}
\tilde{\Phi}^* = \Phi^* - \Phi^*  (\id + \df_W \mathsf{H}^*)^{-1}\df_W \mathsf{H}^*, \quad \tilde{\mathsf{H}}^* = \mathsf{H}^* (\id + \df_W \mathsf{H}^*)^{-1},
\end{equation}
where $(\id + \df_W \mathsf{H}^*)^{-1}=\sum_{m\geq 0}(-1)^m (\df_W\mathsf{H}^*)^m$. It can be checked directly that
\begin{align}\label{perturbed-homotopy}
\begin{cases}
\tilde \Phi^*\circ \ \Upsilon^*=\id\\
 \id - \Upsilon^*\circ \tilde \Phi^*=\tilde{\mathsf{H}}^*\circ (\p_H+\df_W)+(\p_H+\df_W)\circ \tilde{\mathsf{H}}^*
 \end{cases}
\end{align}

We will need the following property to compute the cup product $\cup$.

\begin{lem}\label{lem:homotopy splitting form}
For any $\phi\in C^{p}(\overline{A},A[G])$ and $\psi\in C^{q}(\overline{A},A[G])$, we have
\begin{equation}\label{eq:homotopy splitting form}
\mathsf H^*(\phi\cup\psi) = (-1)^p\phi\cup \mathsf H^*(\psi)+\mathsf H^*\big(\phi\cup(\Upsilon^*\circ\Phi^*(\psi))\big).
\end{equation}
\end{lem}

\begin{proof} Let us introduce the notation $ (\cdots]$ for the input such that
$$
  \phi(a_0\otimes \cdots \otimes a_p]\coloneqq\phi(a_0\otimes \cdots \otimes a_{p-1}) a_p, \quad \phi \in C^p(A,A[G]).
$$

According to (\ref{eq:homotopy between resolutions}),
\begin{align*}
&(-1)^p\phi\cup \mathsf H^*(\psi)(a_1\otimes\cdots\otimes a_p\otimes a_{p+1}\otimes\cdots\otimes a_{p+q-1})\nonumber\\
=& \sum_{i=1}^{q-1}(-1)^{p+i}\phi(a_1\otimes \cdots\otimes a_p)\psi\big(a_{p+1}\otimes\cdots \otimes a_{p+i-1}\otimes
\Phi \!\circ\!\Upsilon (1\otimes a_{p+i}\otimes\cdots\otimes a_{p+q-1}\otimes 1)\big]\nonumber\\
=& \sum_{i=1}^{q-1}(-1)^{p+i}(\phi\cup\psi)(a_1\otimes \cdots\otimes a_{p+i-1}\otimes \Phi_{q-i}\!\circ\!\Upsilon_{q-i}(1\otimes
a_{p+i}\otimes\cdots \otimes a_{p+q-1}\otimes 1)],
\end{align*}
and
\begin{align*}
&\mathsf H^*\big(\phi \cup (\Upsilon^* \circ \Phi^*(\psi))\big)(a_1\otimes\cdots\otimes a_p\otimes a_{p+1}\otimes\cdots\otimes a_{p+q-1})\nonumber\\
=&\sum_{i=1}^{p}(-1)^{i}(\phi \cup (\Upsilon^*\circ\Phi^*(\psi))\big(a_1\otimes \cdots\otimes a_{i-1}\otimes
\Phi_{p+q-i}\!\circ\!\Upsilon_{p+q-i}(1\otimes a_{i}\otimes\cdots\otimes a_{p+1}\otimes\cdots \otimes a_{p+q-1}\otimes 1)\big]\nonumber\\
&+\sum_{i=1}^{q-1}(-1)^{p+i}\phi(a_1\otimes\cdots\otimes a_p)\Upsilon^*\circ\Phi^*(\psi)\big(a_{p+1}\otimes \cdots\otimes a_{p+i-1}\otimes
\Phi_{q-i}\!\circ\!\Upsilon_{q-i}(1\otimes a_{p+i}\otimes\cdots\otimes a_{p+q-1}\otimes 1)\big]\nonumber\\
=& \sum_{i=1}^{p}(-1)^{i}(\phi\cup\psi)\big(a_1\otimes \cdots\otimes a_{i-1}\otimes \Phi_{p+q-i}\circ\Upsilon_{p+q-i}(1\otimes
a_{i}\otimes\cdots\otimes a_{p+1}\otimes\cdots)\big).
\end{align*}
Here the last equality holds because the second summand vanishes (recall the definition of $\Upsilon$, we pick an increasing splitting for
terms in a sequence and put all the latter halves into the first term to make a new sequence, in which we cannot find an new increasing
splitting), and $\Upsilon^*\circ\Phi^*$ will do nothing in the first summand because they will acts on a sequence of terms like
$x_{i_{\sigma(1)}}\otimes x_{i_{\sigma(2)}}\cdots $ with $i_1<i_2<\cdots$.

Sum them up and we will get $\mathsf H^*(\phi\cup\psi)(a_1\otimes\cdots\otimes a_{p+q-1})$.
\end{proof}

\begin{ex}Consider the case $N=2$ and write $x=x_1,y=x_2$.  For $\phi_k\in C^k(\overline{A},A[G])$, we have
\begin{align*}
&\mathsf H_0^*(\phi_1) = 0,\\
&\mathsf H_1^*(\phi_2)(x^ay^b) =  -\sum_{0\leqslant s <a}\phi_2(x^sy^b\otimes x) x^{a-s-1}-\sum_{0\leqslant t<b}\phi_2(y^t\otimes y)x^ay^{b-t-1},\\
&\mathsf H_2^*(\phi_3)(x^{a_1}y^{b_1}\otimes x^{a_2}y^{b_2}) \nonumber\\
=&  -\!\!\sum_{\mbox{\tiny$\begin{array}{cc}0\!\leqslant \!s\!<\!a_1\\0\!\leqslant\! t\!<\!b_2\end{array}$}}\phi_3(x^sy^{b_1+t}\otimes(x\otimes y-y\otimes x))x^{a_1+a_2-s-1}y^{b_2-t-1}\nonumber\\
&+\sum_{0\leqslant s<a_2}\phi_3(x^{a_1}y^{b_1}\otimes x^{s}y^{b_2}\otimes x)x^{a_2-s-1}\nonumber\\
&+\sum_{0\leqslant t<b_2}\phi_3(x^{a_1}y^{b_1}\otimes y^t\otimes y)x^{a_2}y^{b_2-t-1}.
\end{align*}
\end{ex}

\boldmath
\subsection{$G$-twisted cohomology}
\unboldmath

Let us now compute the cohomology
\[\hch_c(A_W,A_W[G])\]
which is the state space of orbifold Landau-Ginzburg model.

\begin{defn} Let $g\in G$, $\fix(g)$ be the fixed locus of $g$. We let
\begin{align}
 \mathsf{PV}^\bullet(\fix(g))=A_g\otimes \wedge^\bullet T \fix(g)
 \end{align}
denote algebraic polyvector fields on $\fix(g)$. Here $T \fix(g)$ are algebraic vector fields on $\fix(g)$.
\end{defn}

As in the proof of Theorem \ref{thm: comparison of curved Hochschild}, there is a spectral sequence converging to $\mathsf{HH}_c(A_W, A_W[G])$ whose $E_1$ page is $\mathsf{HH}_c(A, A[G])$. This can be computed by the Koszul resolution
$$
\mathsf{HH}^\bullet_c(A, A[G])\simeq  \mathsf{H}(K^\bullet({A},A[G]), \p_K).
$$
Using \eqref{eqn: koszul differential}, we find (see also \cite{NPPT06} and \cite{HT10})

\begin{lem}\label{lem:polyvector field and hochschild coh}
There is an natural isomorphism between graded vector spaces:
\[\begin{array}{crcl}
\Theta_g\colon& \mathsf{PV}^\bullet(\fix(g))[N-N_g]&\rightarrow& \mathsf{HH}^\bullet_c(A, Ag)\simeq  \mathsf{H}(K^\bullet(\overline{A},Ag), \p_K)\\
&a \p_{j_1}\wedge\p_{j_2}\wedge\cdots \wedge\p_{j_p}&\mapsto& a e_{j_1}e_{j_2}\cdots e_{j_p}e_{\bm{I}_g} g.
\end{array}\]
Here $\p_{i} = \p/\p x_i$ of degree $\deg\p_i =1$. $\bm{I}_g=\{i_1<i_2<\cdots< i_{N-N_g}\}$  is the moving index of $g$ and $e_{\bm{I}_g}\coloneqq e_{i_1}\cdots e_{i_{N-N_g}}$.
\end{lem}

Our assumption on $W$ implies that $W_g$ has an isolated singularity at the origin (see \cite{FJR13} for details).  
The complex in the $E_2$-page is precisely the Koszul resolution for the critical locus of $W_g$. Therefore the spectral sequence converges at $E_2$-page whose cohomology in the $g$-sector is $\jac(W_g)[N-N_g]$.

In conclusion, we have (see also {Theorem 4.2.} and {Theorem 6.3.} in \cite{CT13})

\begin{thm}\label{thm: sectors equals Jacobi rings}
$\Theta_g$ in Lemma \ref{lem:polyvector field and hochschild coh} induces a natural isomorphism between  $\Z/2\Z$-graded vector spaces
\begin{equation}\Theta_g\colon  \jac(W_g)[\overline{N}-\overline{N_g}] \cong \hch_c(A_W,A_Wg) .\end{equation}
\end{thm}

\begin{ex}[Fermat type]\label{eg-fermat-type} $A=\C[x_1], W = x_1^n$. Let $g\in G_W$ and $g\neq e$. Then
$$
\mathsf{HH}_c(A_W,A_Wg)=\C e_1g 
$$
is one-dimensional and has odd parity.
\end{ex}

\begin{ex}[Loop type]\label{eg-loop-type} $A=\C[x_1,\cdots, x_N], W={x_{1}}^{n_1}x_2+{x_2}^{n_2}x_3+\cdots + {x_N}^{n_N}x_1$.  Let $g\in G_W$ and $g\neq e$. Assume $g(x_i)=\lambda_i x_i$. Then $\lambda_i\neq 1$ for any $i$. Therefore
$$
\mathsf{HH}_c(A_W,A_Wg)=\C e_1e_2\cdots e_Ng 
$$
is one-dimensional and has the same parity as $N$.
\end{ex}

\begin{ex}[Chain type]\label{eg-chain-type} $A=\C[x_1, \cdots x_N], W= {x_{1}}^{n_1}x_2+{x_2}^{n_2}x_3+\cdots + {x_N}^{n_N}$.  Let $g\in G_W$ and $g\neq e$. Assume $g(x_i)=\lambda_i x_i$. There exists $l_g$ such that $\lambda_i\neq 1$ for $i\leq l_g$ and $\lambda_i=1$ for $i>l_g$. Hence
$$
\mathsf{HH}_c(A_W,A_Wg)=\jac\left(x_{l_g+1}^{n_{l_g+1}}x_{l_g+2}+\cdots+ x_{N-1}^{n_{N-1}}x_N+x_N^{n_N}\right) e_1\cdots e_{l_g}g, 
$$
whose parity is the same as $l_g$.
\end{ex}

\boldmath
\subsection{${G}$-Frobenius algebraic structure}\label{section-G-Frobenius}
\unboldmath

\begin{defn}[\cite{K03,K06}]\label{def: G-Frob algebra}
A \emph{$G$-Frobenius algebra} on $\C$ consists of $(G,H,\cup,\eta, \rho,\mathfrak{1}_{\e},\chi)$, where
\begin{enumerate}
\item $G$ is a finite group;
\item $H$ is a finite dimensional $\Z$ (or $\Z/2\Z$)-graded vector space with a sector decomposition
\begin{align}
H = \bigoplus_{g \in G}H_g;
\end{align}
  \item $\chi \in \hom(G,\C^*)$ is a character of $G$,
  \item $\rho \colon G \rightarrow \mathsf{Aut}(H)$ defines a $G$-action on $H$ satisfying
  \begin{description}
    \item[$\bm{\rmnum{1}}$)] $\rho(g)\colon H_h\to H_{ghg^{-1}}$;
        \item[$\bm{\rmnum{2}}$)] $\rho(g)\colon H_g\to H_g$ is the scalar multiplication by $\chi(g)^{-1}$;
  \end{description}
  \item $\eta$ is a non-degenerate bilinear form on $H$ satisfying
  \begin{description}
    \item[$\bm{\rmnum{1}}$)] for $\ha_g\in H_g,\ha_h \in H_h$, $\eta(\ha_g,\ha_h)=0$ unless $gh = \e$,
    \item[$\bm{\rmnum{2}}$)] for $g \in G$ and $\ha,\hb \in H$,
    \begin{equation}\label{eq: (c)2}
    \eta(\rho(g)\ha,\rho(g)\hb) = \chi(g)^{-2}\eta(\ha,\hb);
    \end{equation}
  \end{description}
  \item $(H,\cup,\mathfrak{1}_{\e})$ is an associative algebra with the identity $\mathfrak{1}_{\e}$ satisfying
  \begin{description}
    \item[$\bm{\rmnum{1}}$)] $\cup$ is compatible with the sector decomposition,
    $$
       \cup\colon H_g \times H_h\to H_{gh},
    $$
    \item[$\bm{\rmnum{2}}$)] $\cup$ is $G$-equivariant and $\mathfrak{1}_{\e}$ is $G$-invariant,
    \item[$\bm{\rmnum{3}}$)] $\cup$ is twisted commutative,
    \begin{equation} \label{eq: (d)3}
    \ha_g \cup \rho(g^{-1}) \ha_h = (-1)^{\abs{\ha_g}\abs{\ha_h}} \ha_h \cup \ha_g, \quad \forall \ha_g\in H_g,\ha_h \in H_h,
    \end{equation}
    \item[$\bm{\rmnum{4}}$)] $\cup$ is compatible with $\eta$,
    \begin{equation} \label{eq: (d)4}
    \eta(\ha\cup\hb,\gamma) = \eta(\ha,\hb \cup \gamma), \quad \forall \ha,\hb,\gamma\in H,
    \end{equation}

    \item[$\bm{\rmnum{5}}$)] the following projective trace axiom holds,
    \begin{equation}\label{eq: (d)5}
    \chi(h)\mathsf{Tr}_s(\mathsf{L}_{\ha}\rho(h)|_{H_g}) = \chi(g)^{-1}\mathsf{Tr}_s(\rho(g^{-1})\mathsf{L}_{\ha}|_{H_h}), \quad \forall g,h \in G, \ha\in H_{ghg^{-1}h^{-1}},
    \end{equation}
    where $\mathsf{L}_{\ha}$ is an operator on $H$ given by
    \begin{equation}
    \mathsf{L}_{\ha}(\hb) \coloneqq \ha \cup \hb,
    \end{equation}
    and $\mathsf{Tr}_s$ denotes the trace on $\Z$-graded vector spaces and denotes the super trace on $\Z/2\Z$-graded vector spaces.
  \end{description}
\end{enumerate}
\end{defn}

\begin{defn}\label{def: special G-Frob algebra}
A $G$-Frobenius algebra $H$ is called special if
\begin{itemize}
\item there exists generator $\mathfrak{1}_g \in H_g$ for any $g\in G$ such that $H_g=H_\e \cup \mathfrak{1}_g$,
\item there exists $\rho_{g,h}\in \C^*$ for any $g,h\in G$ such that,
\begin{equation}\label{eq:group action 2-cocycle}\rho(g)(\mathfrak{1}_h) = \rho_{g,h}\mathfrak{1}_{ghg^{-1}}.
\end{equation}
\end{itemize}
\end{defn}

In this subsection we show that the state space of orbifold Landau-Ginzburg model
\[\mathsf{HH}_c(A_W,A_W[G])\]
carries a natural structure of special $G$-Frobenius algebra. This will be constructed in several steps.

Firstly, we have the sector decomposition
$$
\mathsf{HH}_c(A_W,A_W[G])=\bigoplus_{g\in G} H_g, \quad \text{where}\quad H_g= \mathsf{HH}_c(A_W, A_W g),
$$
together with compatible $G$-action and twisted commutative cup product $\cup$ by Theorem \ref{thm: dga on G-twisted Hochschild}.

Define a character $\chi\colon G\rightarrow \C^*$ of $G$ by
\begin{equation}\label{eq: character}
\chi(g) = \det(g)= \lambda_1^g\lambda_2^g\cdots \lambda^g_N,\quad \text{if}\quad \leftidx{^g}x_i=\lambda_i^g x_i.
\end{equation}

\begin{lem}\label{lem: g acts on g-sector by scaling}
For any $g \in G$ and $\ha_g \in H_g$, we have
\begin{equation}\label{eq: (b)2'}
g^*(\ha_g) = \chi(g)^{-1} \ha_g.
\end{equation}
\end{lem}

\begin{proof} This can be checked explicitly using Lemma \ref{lem:polyvector field and hochschild coh}.
\end{proof}

Let us now construct the generator $\mathfrak{1}_g \in H_g$ for each $g \in G$. Define
\begin{align}\label{eq: generator}
\mathfrak{1}_g= \Theta_g(1)\in H_g
\end{align}
where $\Theta_g$ is defined in Theorem \ref{thm: sectors equals Jacobi rings} and $1$ is the identity in $ \jac(W_g)$. The parity of $\mathfrak{1}_g$ is the same as $N-N_g$. It is easy to see that $\mathfrak{1}_{\e}$ is the unit of $\mathsf{HH}_c(A_W,A_W[G])$. For $g,h \in G$, define
\begin{equation}\label{eq: group action co-cycle definition}
\rho_{g,h} = \prod_{i\in \bm{I}_h} \big(\lambda^g_i\big)^{-1},
\end{equation}
where $\bm{I}_h$ is the moving index of $h$ and $\lambda^g_i$ is defined by $\leftidx{^g}x_i = \lambda^g_ix_i$. It is checked that
\begin{equation}\label{eq:group action 2-cocycle'}
g^*(\mathfrak{1}_h) = \rho_{g,h} \mathfrak{1}_h = \rho_{g,h} \mathfrak{1}_{ghg^{-1}}.
\end{equation}
Let
\begin{align}
\Pi_g \colon \jac(W) \rightarrow \jac(W_g)
\end{align}
denote the natural restriction map. This is well-defined since $W$ is $g$-invariant.

\begin{lem}\label{lem-generator-1g}
For any $f \in \jac(W)$
\begin{equation}\label{eq: cup with untwisted sector}
[\Theta_{\e}(f)] \cup \mathfrak{1}_g = [\Theta_g\big(\Pi_g(f)\big)].
\end{equation}
In particular, $H_g$ is a cyclic $H_{\e}$-module generated by $\mathfrak{1}_g$.
\end{lem}

\begin{proof} Identity (\ref{eq: cup with untwisted sector}) follows by checking
\begin{equation*}
\Upsilon^*(\Theta_{\e}(f)) \cup \Upsilon^*(\Theta_{g}(1)) = \Upsilon^*(\Theta_g(f)).
\end{equation*}
The last statement follows from the surjectivity of $\Pi_g$.
\end{proof}

We now construct a bilinear form $\eta$ on $\mathsf{HH}_c(A_W,A_W[G])$. On the identity sector
$$
H_{\e}\cong \jac(W),
$$
the pairing $\eta_{\e}$ is defined by the residue
\begin{equation}
\eta_{\e} ([\Theta_{\e}(f_1)],[\Theta_{\e}(f_2)]) = \res_{\C^N}
\begin{bmatrix} f_1f_2\df x_1\wedge \df x_2 \wedge \cdots \wedge \df x_N\\ \pd{W}{x_1}\pd{W}{x_2}\cdots \pd{W}{x_N}
\end{bmatrix}.
\end{equation}
Since $W$ has an isolated singularity at the origin, the residue pairing on $\jac(W)$ is non-degenerate.  We extend $\eta_{\e}$ onto twisted sectors by  defining
\begin{align}
\eta_g \colon  H_g \otimes H_{g^{-1}} & \rightarrow  \C,\nonumber\\
(\alpha, \beta) & \mapsto \eta_{\e}(\alpha\cup \beta, \mathfrak{1}_{\e}).
\end{align}
Then the bilinear form $\eta$ is defined by
\begin{equation}\label{eq: non-deg bilinear form}
\eta = \sum_{g\in G} \eta_{g}.
\end{equation}

\begin{lem}
 The  bilinear form $\eta$ on $\mathsf{HH}_c(A_W,A_W[G])$ is compatible with the cup product
\begin{equation} \label{eq: (d)4'}
    \eta(\ha\cup\hb,\gamma) = \eta(\ha,\hb \cup \gamma), \quad\forall \ha,\hb,\gamma\in H
\end{equation}
and satisfies the $G$-equivariance condition
\begin{equation}\label{eq: (c)2'}
    \eta(g^*(\ha),g^*(\hb)) = \chi(g)^{-2}\eta(\ha,\hb), \quad \forall \ha, \hb\in H .
\end{equation}
\end{lem}

\begin{proof} By construction,
$$
\eta(\ha\cup\hb,\gamma) = \eta(\ha,\hb \cup \gamma)= \eta(\ha\cup\hb\cup \gamma, 1_{\e}).
$$
Equation (\ref{eq: (d)4'}) follows. Equation \eqref{eq: (c)2'} follows from the $G$-equivariance of $\cup$, the $G$-invariance of $W$ and the property of residue.
\end{proof}
The discussion above is summarized as follows.
\begin{thm}\label{thm: invertible LG gives a G-Frob algebra.}
Let $W$ be a non-degenerate invertible polynomial. Then the state space $\mathsf{HH}_c(A_W,A_W[G])$ of orbifold Landau-Ginzburg model together with the $G$-action \eqref{eq: group action on twisted cochains}, cup product $\cup$, bilinear form $\eta$ \eqref{eq: non-deg bilinear form}, character $\chi$ \eqref{eq: character} and generators $\mathfrak{1}_g$ \eqref{eq: generator} form a special $\Z/2\Z$-graded $G$-Frobenius algebra.
\end{thm}

\begin{proof} The proof of \eqref{eq: (d)5} is via direct calculation by Kaufmann in \cite{K03}. 
Hence, We only need to check the non-degeneracy of  $\eta$ on $g$-sectors where $g\neq \e$ for elementary invertible polynomials. If $W$ is of Fermat type or loop type,  $\mathsf{HH}_c(A_W,A_Wg)$ is one-dimensional (Example \ref{eg-fermat-type} and Example \ref{eg-loop-type}) and $\eta(\mathfrak{1}_g, \mathfrak{1}_{g^{-1}})\neq 0$ by Theorem \ref{thm: cup product formula for General invertible polynomial}. If $W$ is of chain type, it can be checked directly that $\eta$ on $\mathsf{HH}_c(A_W,A_Wg)$ is proportional to the residue pairing on $\jac(W_g)$ using Example \ref{eg-chain-type}, Lemma \ref{lem-generator-1g} and Theorem \ref{thm: cup product formula for General invertible polynomial}. 
\end{proof}

\begin{rem}
We expect that the pairing $\eta$ on $\mathsf{HH}_c(A_W,A_Wg)$ and $\mathsf{HH}_c(A_W,A_Wg^{-1})$ is the same as the residue pairing in $\jac(W_g)$ as long as $W_g$ has an isolated singularity. Alternatively, there is categorical construction of pairing (Mukai pairing) from the dg-category of matrix factorization\cite{CW10, S08, S16}. The Mukai pairing of G-equivariant case is explicitly computed in \cite{PV12}. If one uses the trivial volume form to identity Hochschild homology with cohomology, then our pairing here coincides with their result (see Theorem 4.2.1 in \cite{PV12} and Theorem \ref{thm: cup product formula for General invertible polynomial} below). The non-degeneracy of Mukai pairing is proved by Shklyarov in \cite{S08} for all homological smooth dg-algebra.
\end{rem}

\subsection{Quantum differential operator and cup product}

In this section we establish an explicit formula for cup product of $G$-twisted Hochschild cohomology. This is achieved via the following established quasi-isomorphisms (formula \eqref{perturbed-homotopy} and Theorem \ref{thm: sectors equals Jacobi rings})
\[\begin{tikzpicture}
\matrix (a) [matrix of math nodes, row sep=3.5em, column sep=4em, text height=2ex, text depth=0.9ex]
{\big(C^{\bullet}(\overline{A},Ag),\p_H+\df_W\big) & \big(K^{\bullet}(A,Ag),\p_K+\tilde{\df}_W\big)& (\jac(W_g)[\overline{N}-\overline{N_g}],0).\\};
\path[-latex]
(a-1-1) edge node[above,font=\scriptsize] {$\tilde{\Phi}^*$} (a-1-2);
\path[latex-]
(a-1-1.-3) edge node[below,font=\scriptsize] {$\Upsilon^*$} (a-1-2.-177);
\path[-latex]
(a-1-2) edge node[above,font=\scriptsize] {$\mathsf{p}$} (a-1-3);
\path[latex-]
(a-1-2.-3) edge node[below,font=\scriptsize] {$\mathsf{i}$} (a-1-3.-177);
\end{tikzpicture}\]
Here $p$ is the natural projection onto the appropriate components of polyvectors.

Given $\alpha_g\in \jac(W_g)[\overline{N}-\overline{N_g}], \beta_h\in \jac(W_h)[\overline{N}-\overline{N_h}]$, their cup product in the $gh$-sector can be computed by
\begin{align}\label{cup-product-formula}
  \mathsf{p}\tilde{\Phi}^*\left( (\Upsilon^*\mathsf{i} (\alpha_g))\cup (\Upsilon^*\mathsf{i} (\beta_h)) \right).
\end{align}

For invertible polynomials, it turns out that $\cup$ on cohomology is determined by the cup product between $g$-sector and $g^{-1}$-sector. In this case $(\Upsilon^*\mathsf{i} (\alpha_g))\cup (\Upsilon^*\mathsf{i} (\beta_h))$ is a cocycle in the identity $\e$-sector. It is easy to see that formula \eqref{cup-product-formula} is reduced to the following
\begin{lem} For a $(\p_H\!+\!\df_W)$-closed element in the identity $\e$-sector
\[\phi=\sum_{k=0}^l\phi_{2k}, \quad \phi_{2k}\in C^{2k}(\overline{A},A\e),\]
we have
\begin{equation}\label{formula-homotopy-cup-product}
\mathsf{p} \circ \tilde \Phi^*({\phi}) = \Big[\sum_{k=0}^l (-1)^k(\df_W\mathsf{H}^*)^k \phi_{2k}\Big].
\end{equation}
Here $[~~ ]$ on the right hand side represents its class in $\jac(W)$.
\end{lem}

We first establish some properties of the cup product $\cup$ and the homotopy operator $\mathsf H^*$.

\begin{defn}
Given $g=(q_1,\cdots, q_N)\in (\C^*)^N$, we define the decomposition
$$
g= \prod_{i=1}^Ng^{(i)},
$$
where $g^{(i)}=(1,\cdots, q_i,\cdots, 1)$ is called the $i$'th component of $g$. We  define $\tilde G$ to be the group generated by $g^{(i)}$'s for all $g\in G$.
\end{defn}

It would be convenient to extend our Hochschild cochains to $C^\bullet(A,A[\tilde G])$ and identify $C^\bullet(A,A[G])$ as a sub-algebra under the natural embedding
$$
   C^\bullet(A,A[G])\into C^\bullet(A,A[\tilde G]).
$$
We will also simply write
$
   \phi_1\phi_2
$
for the cup product $\phi_1\cup \phi_2$ in this subsection.

Recall the quantum differential operator $\p_{x_i}^g$ defined by \eqref{eq: q-diff operator}.
\begin{defn}
For $g\in G$, we define the first order quantum differential operators $\p_i^g \in C^1(A,Ag^{(i)})$ by
\begin{align}
  \p_i^g\colon f \to \p_{x_i}^g(f) g^{(i)},
\end{align}
and the second order quantum differential operators $\p_{i,j}^{g,h} \in C^1(A,A[\tilde G])$ by
\begin{align}
\p^{g,h}_{i,j}=
 \begin{dcases}\p^g_{i} \lrf{\p^h_{j}} & i\neq j\quad (\text{mixed type}),\\
  -\mathsf H^*(\p_i^g\p_i^h) & i=j \quad (\text{pure type}).
\end{dcases}
\end{align}
Here $\mathsf H^*$ is the homotopy on $G$-twisted Hochschild cochains defined in Proposition \ref{prop-homotopy}.

\end{defn}

\begin{ex}\label{ex-second-order} Let $\leftidx{^g}x_i = \epsilon_1x_i$ and $\leftidx{^h}x_i = \epsilon_2 x_i$. Recall definition \eqref{eqn-quantum-number}.  We have
  \begin{equation}\p^{g,h}_{i,i}\big(x_i^n\big) = \frac{\epsilon_1^{n-1}[n]_{\epsilon_2}-[n]_{\epsilon_1\epsilon_2}}{\epsilon_1-1}x_i^{n-2}g^{(i)}h^{(i)}.
  \end{equation}
In parcular, if $\epsilon_1\epsilon_2 = 1$, i.e., $g^{(i)} = (h^{(i)})^{-1}$,
  \begin{equation}\p^{g,h}_{i,i}\big(x_i^n\big) = \frac{[n]_{\epsilon_1}-n}{\epsilon_1-1}x_i^{n-2}\e.
  \end{equation}
\end{ex}

Using quantum differential operators, the cochain map \eqref{eq:cochain map upsilon star}
$$
   \Upsilon^*\colon K^\bullet(A,A[G])\to C^\bullet(\bar A,A[G])
$$
can be expressed in terms of cup product
\begin{align}
\Upsilon^*(e_{i_1}\cdots e_{i_p}g)= g^{(1)}\cdots g^{(i_1-1)}\p_{i_1}^g g^{(i_1+1)}\cdots g^{(i_2-1)}\p_{i_2}^g g^{(i_2+1)}\cdots  g^{(i_p-1)}\p_{i_p}^g g^{(i_p+1)}\cdots g^{(N)},
\end{align}
where $1\leqslant i_1<i_2<\cdots<i_p\leqslant N$, and $g^{(i)}$ is viewed naturally as an element in $C^0(\bar A, A g^{(i)})$.

\begin{lem}\label{lem:q-diff operators}~~
\begin{description}
  \item[(a)] First order quantum differential operators are $\p_H$-closed,
  \begin{equation}\label{eq: 1st q-diff op is b-closed}\p_H(\p^g_i) = 0.
  \end{equation}
  \item[(b)] Second order quantum differential operators are symmetric
  \begin{equation} \p^{g,h}_{i_1,i_2} = \p^{h,g}_{i_2,i_1}.
  \end{equation}
  and satisfy
  \begin{equation} \p_H(\p^{g,h}_{i,j}) =\begin{cases} -\p^g_{i}\p^h_{j}-\p^h_{j}\p^g_{i} & i\neq j\\
   -\p^{g}_i\p^{h}_i & i=j.
  \end{cases}
  \end{equation}
\end{description}
\end{lem}

\begin{proof} (a) is equivalent to the twisted Leibnitz rule
$$
  \p_i^g(ab)=\p_i^g(a)b+ a\p_i^g(b)=\p_{x_i}^g(a) (\leftidx{^{g^{(i)}}}b)g^{(i)}+a \p_{x_i}^g(b) g^{(i)}, \quad \forall a,b\in A.
$$
The first equation in (b) is by direct check. The second equation  follows from  (\ref{eq: twisted commutative on cochains}), (\ref{eq: 1st q-diff op is b-closed}) and \eqref{cochain-homotopy-eqn}.
\end{proof}

The following lemma gives some useful reorganizing rules for the homotopy $\mathsf H^*$.

\begin{lem}\label{lem:homotopy cal}
Let $\bm{I} = \{i_1,i_2,\cdots i_p\}$, $i_2<i_3<\cdots <i_p$ and $i_k\leqslant i_1<i_{k+1}$. Then we have the following formula for the homotopy operator
  \begin{eqnarray}\label{eq: homotopy cal}
  \mathsf H^*(\p^{g_1}_{i_1}\p^{g_2}_{i_2}\cdots \p^{g_p}_{i_p}) &=& -\p^{g_1,g_2}_{i_1,i_2}\p^{g_3}_{i_3}\cdots
  \p^{g_p}_{i_p}-\p^{g_2}_{i_2}\p^{g_1,g_3}_{i_1,i_3}\p^{g_4}_{i_4}\cdots \p^{g_p}_{i_p}\nonumber\\
  &&-\cdots -\p^{g_2}_{i_2}\cdots \p^{g_{k-1}}_{i_{k-1}}\p^{g_1,g_{k}}_{i_1,i_{k}}\p^{g_{k+1}}_{i_{k+1}}\cdots \p^{g_{p}}_{i_{p}}.
  \end{eqnarray}
More generally, if on the left hand side we insert several $0$-cochains $h^{(j_1)}_1, \cdots, h^{(j_q)}_q$ into the product $\p^{g_1}_{i_1}\p^{g_2}_{i_2}\cdots \p^{g_p}_{i_p}$ at positions after $\p^{g_1}_{i_1}$, then on the right hand side each term is modified by inserting $h^{(j_1)}_1, \cdots, h^{(j_q)}_q$ in ascending order as follows: if $i_k< j_s<i_{k+1} (k\geq 2)$ , then $h^{(j_s)}$ is inserted between quantum differential operators indexed by $i_k$ and
      $i_{k+1}$; if $i_k=j_s<i_{k+1}$, then we combine $h_s^{(j_s)}$ and $\p_{i_k}^{g_k}$ into $\p_{i_k}^{g_k h_s^{(j_s)}}$.
\end{lem}

\begin{proof}Under the assumption $i_2<i_3<\cdots <i_p$,  we have only one term survive in (\ref{eq:homotopy between resolutions}),
\begin{equation*}
\mathsf H^*(\p^{g_1}_{i_1}\p^{g_2}_{i_2}\cdots \p^{g_p}_{i_p})(a_1\otimes a_2\otimes \cdots \otimes a_{p-1}) = -\p^{g_1}_{i_1}\p^{g_2}_{i_2}\cdots
\p^{g_p}_{i_p} \big(\Phi\circ \Upsilon (1\otimes a_1 \otimes \cdots\otimes a_{p-1}\otimes 1)\big].
\end{equation*}
Here the notation $ (\cdots]$ means that
$$
  \phi(a_0\otimes \cdots \otimes a_p]\coloneqq\phi(a_0\otimes \cdots \otimes a_{p-1}) a_p, \quad \phi \in C^p(A,A[G]).
$$
Let $\bm{J}=\{i_2,\cdots,i_{p-1}\}$, $g = g_1^{(i_1)}\cdots g_{p}^{(i_p)}$ and we keep the notations in  \eqref{notation-Upsilon}. Then
\begin{align*}
 & \mathsf H^*(\p^{g_1}_{i_1}\p^{g_2}_{i_2}\cdots \p^{g_p}_{i_p})(a_1\otimes a_2\otimes \cdots \otimes a_{p-1})\nonumber\\
 =& -\sum_{\bm{s}\in \mathcal{S}(\bm{J})}\p^{g_1}_{i_1}\p^{g_2}_{i_2}\cdots \p^{g_p}_{i_p}\left(a_1^1\cdots a_{p-1}^1\otimes x_{i_2}\otimes\cdots\otimes x_{i_{p}}\otimes a_1^2\cdots a_{p-1}^2   \right] \\
 =&-\sum_{\bm{s}\in \mathcal{S}(\bm{J})} \p_{x_{i_1}}^{g_1} (a_1^1\cdots a_{p-1}^1) (\leftidx{^g}(a_1^2\cdots a_{p-1}^2)) g \\
= &-\sum_{j=1}^{k-1} \p^{g_2}_{i_2}\cdots \p^{g_j}_{i_j}\p^{g_1g_{j+1}}_{i_1i_{j+1}}\p^{g_{j+2}}_{i_{j+2}} \cdots \p^{g_p}_{i_p} (a_1\otimes
a_2\otimes \cdots \otimes a_{p-1}).
\end{align*}
The general case with $0$-cochain insertions is proved similarly.

\end{proof}

The next lemma gives certain vanishing conditions for $\mathsf H^*$ on twisted cochains.

\begin{lem}\label{lem:vanishing condition for homotopy} Let $\phi\in C^\bullet(A, A[G])$ be a $G$-twisted cochain expressed in terms of cup product of first and second order quantum differential
operators and elements of $\tilde G$ (viewed as $0$-th cochain). Then $\mathsf H^*\phi=0$  in the following cases:
\begin{description}
  \item[(a)]   $
     \phi =\cdots \p_{i,j}^{g,h}.
  $  has a second order quantum differential operator appearing at the last position;

  \item[(b)] $\phi$ is of the form of arbitrary insertion of elements of $\tilde G$ into the expression
  \[
  \cdots \p^{g_k,g'_{k}}_{i_k,i'_k}\p^{g_{k+1}}_{i_{k+1}}\cdots\p^{g_{p-1}}_{i_{p-1}} \p^{g_p}_{i_p},\]
  with $i_k<i_{k+1}<\cdots <i_{p}$ or $i'_k<i_{k+1}<\cdots <i_{p}$;
  \item[(c)] $\phi$ is of the form of arbitrary insertion of elements of $\tilde G$ into the expression
      \[\p^{g_1}_{i_1}\cdots\p^{g_{p-1}}_{i_{p-1}} \p^{g_p}_{i_p},\]
      with $i_1<i_{2}<\cdots <i_{p}$.
\end{description}
\end{lem}

\begin{proof} They are checked directly by \eqref{notation-Upsilon} and (\ref{eq:homotopy between resolutions}).
\end{proof}

We are ready to compute cup product in terms of \eqref{cup-product-formula} and \eqref{def: perturbed operator}. We need to understand how $\tilde{\Phi}^* = \Phi^* - \Phi^*  (\id + \df_W \mathsf{H}^*)^{-1}\df_W \mathsf{H}^*$ acts on cochains of quantum differential operators.

\begin{lem} Consider the following cup product of cochains of quantum differential operators
\begin{equation}\label{eq:usual version cochain}
\phi=a_1g^{(1)}\cdots g^{(i_1-1)}\p^{g}_{i_1}g^{(i_1+1)}\cdots \p^{g}_{i_p}g^{(i_p+1)}\cdots g^{(N)}\cup a_2h^{(1)}\cdots h^{(j_1-1)}\p^{h}_{j_1}h^{(j_1+1)}\cdots
\p^{h}_{j_q}h^{(j_q+1)}\cdots h^{(N)}
\end{equation}
where $i_1<i_2<\cdots< i_p$ and  $j_1<j_2<\cdots <j_q$. $a_1,a_2\in A, g,h\in G$.  Then
\begin{align}\label{eq:homotopy of usual version cochain}
\mathsf H^*(\phi)=&\sum_{p',j_k} \pm a_1\leftidx{^g}a_2 g^{(1)}\cdots \p^{g}_{i_1}g^{(i_1+1)}\cdots \p^{g}_{i_{p'}}g^{(i_{p'}+1)}\cdots g^{(i_{(p'\!+1)}-1)} \nonumber\\
&\qquad\cup \tilde{h}^{(1)}\cdots\p^{\tilde{h}}_{j_1}\tilde{h}^{(j_1+1)}\cdots
\p^{g,\tilde{h}}_{i_{(p'\!+1)},j_{k}} \tilde{h}^{(j_k+1)} \cdots \p^{\tilde{h}}_{i_{(p'\!+2)}} \tilde{h}^{(i_{(p'\!+2)}+1)} \cdots \tilde{h}^{(N)}.
\end{align}
Here the summation is over all $p'<p, j_k$ such that $i_{(p'\!+1)} \geqslant j_{k}$ and $i_{\tilde{p}}\neq j_l$ for all $\tilde{p}\geqslant p'+2, 1\leqslant l\leqslant q$.
\begin{equation*}
\tilde{h}^{(j)} =
\begin{dcases}
h^{(j)}& \text{if }1\leqslant j\leqslant i_{(p'\!+1)},\\
h^{(j)}g^{(j)}& \text{if }i_{(p'\!+1)}<j\leqslant N.
\end{dcases}
\end{equation*}
\end{lem}
\begin{proof}Assume $j_l\leq i_p<j_{l+1}$ for some $1\leq l<q$ or $j_q\leq i_p$. If not, we consider $i_{p-1}$ and so on. Let
\begin{align*}
\alpha&=a_1g^{(1)}\cdots g^{(i_1-1)}\p^{g}_{i_1}g^{(i_1+1)}\cdots g^{(i_p-1)}\\ \beta&=\p^{g}_{i_p}g^{(i_p+1)}\cdots g^{(N)}\cup a_2h^{(1)}\cdots h^{(j_1-1)}\p^{h}_{j_1}h^{(j_1+1)}\cdots
\p^{h}_{j_q}h^{(j_q+1)}\cdots h^{(N)}.
\end{align*}
We apply Lemma \ref{lem:homotopy splitting form} to  $\phi=\alpha\cup \beta$
$$
\mathsf H^*(\phi) = \pm \alpha\cup \mathsf H^*(\beta)+\mathsf H^*\big(\alpha\cup(\Upsilon^*\circ\Phi^*(\beta))\big).
$$
The first term is computed by Lemma \ref{lem:homotopy cal}. In the second term, $\Upsilon^*\circ\Phi^*(\beta)=0$ if $i_p=j_l$; otherwise the role of $\Upsilon^*\circ\Phi^*$ amounts to change terms in $\beta$ into the order
$$
\Upsilon^*\circ\Phi^*(\beta)=\pm \cdots \p^{h}_{j_1}\cdots
\p^{h}_{j_l} \cdots \p_{i_p}^g \cdots \p_{j_l+1}^h\cdots \p_{j_q}^h \cdots.
$$
Now we apply the same process to $\mathsf H^*\big(\alpha\cup(\Upsilon^*\circ\Phi^*(\beta))\big)$ but consider $i_{p-1}$. Recursively we arrive at the lemma.
\end{proof}

\begin{thm}\label{thm:graph summation}Consider the cup product between a $g$-sector and $g^{-1}$-sector
$$
\phi=\Upsilon^*(e_{i_1}\cdots e_{i_p}g)\cup \Upsilon^*(e_{j_1}\cdots e_{j_q}g^{-1})
$$
where $i_1<i_2<\cdots< i_p$ and  $j_1<j_2<\cdots <j_q$. $a_1,a_2\in A$. Then $\mathsf{p} \circ \tilde \Phi^*({\phi})$  is computed as follows.
\begin{description}
\item[(1)] if $p\neq q$, then $\mathsf{p} \circ \tilde \Phi^*({\phi})=0$.
\item[(2)] if $p=q$, consider the following subset of the $p$-th permutation group
$$
   V=\{\sigma\in S_p\mid j_{\sigma(k)}\leq i_{k}, \forall 1\leq k\leq p\}.
$$
Let $g^{(i,j)}=g^{(i)}g^{(i+1)}\cdots g^{(j)}$ for $i\leq j$. Then $\mathsf{p} \circ \tilde \Phi^*({\phi})$ equals
$$
(-1)^{{p(p-1)\over 2}}\sum_{\sigma\in V} \sgn{\sigma} g^{(j_{\sigma(1)},i_1-1)} \p_{i_1,j_{\sigma(1)}}^{g,g^{-1}}(W) (g^{-1})^{(j_{\sigma(1)}+1,i_1)}\cdots g^{(j_{\sigma(p)},i_p-1)} \p_{i_p,j_{\sigma(p)}}^{g,g^{-1}}(W) (g^{-1})^{(j_{\sigma(p)}+1,i_p)}
$$
as elements in $\jac(W)$.
\end{description}
\end{thm}
\begin{proof} We consider the case $p=q$. The case for $p\neq q$ follows by the same steps.
Explicitly,
\begin{align}
\phi=&a_1g^{(1)}\cdots g^{(i_1-1)}\p^{g}_{i_1}g^{(i_1+1)}\cdots \p^{g}_{i_p}g^{(i_p+1)}\cdots g^{(N)}\nonumber\\
\quad & \cup a_2(g^{-1})^{(1)}\cdots (g^{-1})^{(j_1-1)}\p^{g^{-1}}_{j_1}(g^{-1})^{(j_1+1)}\cdots
\p^{g^{-1}}_{j_p}(g^{-1})^{(j_p+1)}\cdots (g^{-1})^{(N)}
\end{align}
and by \eqref{formula-homotopy-cup-product}
$$
\mathsf{p} \circ \tilde \Phi^*({\phi})=(-1)^p(\df_W\mathsf{H}^*)^p\phi.
$$
$\mathsf H^*(\phi)$ is computed by \eqref{eq:homotopy of usual version cochain}. We consider the action of $\df_W$ on each term in \eqref{eq:homotopy of usual version cochain}. There are two cases:
\begin{description}
  \item[Case 1] $W$ is acted on by some first order quantum differential operators. In this case, there will be a second order quantum differential
      operator left in the cochain. If we take a further homotopy $\mathsf H^*$, it will
      become zero by {Lemma \ref{lem:vanishing condition for homotopy}(a),(b)}.
  \item[Case 2] $W$ is acted on by some second order quantum differential operators. In this case,
 a further homotopy $\mathsf H^*$
  will bring the cochain back to the form of (\ref{eq:usual version cochain}) by {Lemma \ref{lem:homotopy cal}}. The relevant second order differential operator is combined by  $\p^{g}_{i_{(p'\!+1)}}$ and $\p^{h}_{j_k}$ where $i_{(p'\!+1)}\geqslant j_k$.  \end{description}
 Repeating the above steps we arrive at the theorem.
\end{proof}

\begin{ex}For $g|_1\p^g_{2}g|_3\p^g_{4}\p^g_5\cup \p^h_{1}h|_2\p^h_3\p^h_4h|_5$ with $g = h^{-1}$, there are two relevant index maps

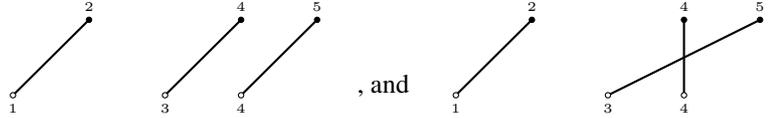
\[\tikz[baseline=1]{
\draw[thick] (1,1) -- (0,0);
\draw[thick] (4,1) -- (3,0);
\draw[thick] (3,1) -- (2,0);
\filldraw[fill=black] (1,1) node[above,font=\tiny] {$2$} circle (1pt);
\filldraw[fill=black] (4,1) node[above,font=\tiny] {$5$} circle (1pt);
\filldraw[fill=black] (3,1) node[above,font=\tiny] {$4$} circle (1pt);
\filldraw[fill=white] (0,0) node[below,font=\tiny] {$1$} circle (1pt);
\filldraw[fill=white] (3,0) node[below,font=\tiny] {$4$} circle (1pt);
\filldraw[fill=white] (2,0) node[below,font=\tiny] {$3$} circle (1pt);
}\quad\text{, and }\quad
\tikz[baseline = 1]{
\draw[thick] (1,1) -- (0,0);
\draw[thick] (4,1) -- (2,0);
\draw[thick] (3,1) -- (3,0);
\filldraw[fill=black] (1,1) node[above,font=\tiny] {$2$} circle (1pt);
\filldraw[fill=black] (4,1) node[above,font=\tiny] {$5$} circle (1pt);
\filldraw[fill=black] (3,1) node[above,font=\tiny] {$4$} circle (1pt);
\filldraw[fill=white] (0,0) node[below,font=\tiny] {$1$} circle (1pt);
\filldraw[fill=white] (3,0) node[below,font=\tiny] {$4$} circle (1pt);
\filldraw[fill=white] (2,0) node[below,font=\tiny] {$3$} circle (1pt);
}.
\]
The first one contributes
$$
-g^{(1)} \p_{2,1}^{g,g^{-1}}(W)(g^{-1})^{(2)}g^{(3)}\p_{4,3}^{g,g^{-1}}(W)\p_{5,4}^{g,g^{-1}}(W)(g^{-1})^{(5)}.
$$
The second one contributes
$$
g^{(1)} \p_{2,1}^{g,g^{-1}}(W) (g^{-1})^{(2)} \p_{4,4}^{g,g^{-1}}(W)g^{(3)}g^{(4)}\p_{5,3}^{g,g^{-1}}(W)(g^{-1})^{(4)}(g^{-1})^{(5)}.
$$
\end{ex}

\subsection{Invertible polynomials and cup product formula}
As an application of the method we have developed, we give explicit cup product formula for a general invertible polynomial $W$ and a finite group $G\subseteq G_W$. We can write $W = W_1 \oplus W_2 \oplus \cdots \oplus W_r$
and $G = G_1\times G_2 \times \cdots \times G_r$, where $W_k$ is an elementary invertible polynomial with dimension $N_k$ and $G_k$ is a
finite subgroup of $G_{W_k}$. Also $A=A_1\otimes \cdots A_r$ is a tensor product of elementary ones.

First of all, we have the following K\"unneth type formula.

\begin{prop}\label{prop-kunneth} As $\Z/2\Z$-graded algebras
$$
    \mathsf{HH}_c(A_W,A_W[G])= \mathsf{HH}_c((A_1)_{W_1},(A_1)_{W_1}[G_1])\otimes \cdots \otimes  \mathsf{HH}_c((A_r)_{W_r},(A_r)_{W_r}[G_r])
$$
where the tensor product on the right hand side is the graded tensor product.
\end{prop}
\begin{proof} This is a direct consequence of Theorem \ref{thm: sectors equals Jacobi rings} and the fact that
$$
  \jac(W)=\jac(W_1)\otimes \cdots \otimes \jac(W_r).
$$

\end{proof}

This proposition reduces the problem to an elementary invertible polynomial only. Let
$$
\mathfrak{1}_g\in \mathsf{HH}_c(A_W, A_W g)
$$
be the generator defined in \eqref{eq: generator}.

\begin{lem}\label{lem-generator-cup} The cup product on $\hch_c(A_W, A_W[G])$ is completely determined by the generators
$$
\mathfrak{1}_g \cup \mathfrak{1}_h, \quad g,h\in G.
$$
\end{lem}
\begin{proof}This follows from Lemma \ref{lem-generator-1g}.

\end{proof}

\begin{defn}\label{defn-Hession}
 Let $W(x_i)$ be an elementary invertible polynomial. Let $g\neq \e\in G_W$, $g(x_i)=\lambda_i x_i$ and $l_g=\sharp\{\lambda_i \mid \lambda_i\neq 1\}$. We define the $g$-twisted Hessian, denoted by $\hess^g(W)$, as follows.
\begin{description}
\item[(a)] If $W = x_1^n$ is of Fermat type, then
\begin{equation*}
\hess^{g}(W)=\f{n}{1-\lambda_1}x_1^{n-2}.
\end{equation*}
\item[(b)] If $W=x_{1}^{n_1}x_2+x_2^{n_2}x_3+\cdots + x_N^{n_N}x_1$ is of loop type, then
  \begin{equation*}
  \hess^{g}(W)= \f{(-1)^{N+1}+n_1n_2\cdots
n_N}{(1-\lambda_1)(1-\lambda_2)\cdots(1-\lambda_N)}x_1^{n_1-1}x_2^{n_2-1}\cdots x_N^{n_N-1}.
  \end{equation*}
  Note that $l_g=N$ for any $g\neq e$ in the loop case.
\item[(c)] If $W= x_{1}^{n_1}x_2+x_2^{n_2}x_3+\cdots + x_N^{n_N}$ is of chain type, then
   \begin{equation*}
      \hess^g(W) =\begin{cases}\f{n_1n_2\cdots
      n_{l_g}}{(1-\lambda_1)\cdots(1-\lambda_{l_g})}x_1^{n_1-2}x_2^{n_2-1}\cdots x_{l_g}^{n_{l_g}-1}x_{l_g+1} & \text{if}\ l_g<N  \\
      \f{n_1n_2\cdots
      n_{N}}{(1-\lambda_1)\cdots(1-\lambda_{N})} x_1^{n_1-2}x_2^{n_2-1}\cdots x_{N}^{n_N-1} & \text{if}\ l_g=N
      \end{cases}
      \end{equation*}
Note that $\lambda_{l_g+1}=\cdots =\lambda_{N}=1$ for any $g\neq e$ in the chain case.
\end{description}
\end{defn}

Our main theorem in this section is the following.

\begin{thm}\label{thm: cup product formula for General invertible polynomial}
Let $A$ be a polynomial ring and $W$ be an elementary invertible polynomial. Let $G\subseteq G_W$. Then  we have the following explicit cup product formula on $\mathsf{HH}_c(A_W,A_W[G])$.
\begin{description}
\item[(1)]  For $g_1,g_2 \in G$ and $g_1,g_2\neq \e, g_2\neq g_1^{-1}$,
\begin{align}\mathfrak{1}_{g_1}\cup \mathfrak{1}_{g_2}=0.\end{align}
\item[(2)] For $g\in G$ and $g\neq \e$,
\begin{align}\mathfrak{1}_g\cup \mathfrak{1}_{g^{-1}} = (-1)^{\f{l_g(l_g-1)}{2}}\hess^g(W)\mathfrak{1}_{\e}.
\end{align}
Here $l_g$ and $\hess^g(W)$ are defined in Definition \ref{defn-Hession}.
\end{description}
By Proposition \ref{prop-kunneth} and Lemma \ref{lem-generator-cup}, this determines the full cup product for all invertible polynomials.
\end{thm}

\begin{proof}(1) By Proposition \ref{lem:subgp lem}, we can assume $G=G_W$. Let $g_1,g_2 \in G_W$ and $g_1,g_2\neq \e, g_2\neq g_1^{-1}$.
\begin{description}
\item[(a)] $W = x_1^n$ is of Fermat type. By Example \ref{eg-fermat-type}, $\mathfrak{1}_{g_1}$ and $\mathfrak{1}_{g_2}$ have both odd parity, therefore $\mathfrak{1}_{g_1}\cup \mathfrak{1}_{g_2}$ is an even element in the $g_1g_2$-sector. Since $g_1g_2\neq 1$,  the $g_1g_2$-twisted sector is generated by $1_{g_1g_2}$ which is odd. It follows that
$$
\mathfrak{1}_{g_1}\cup \mathfrak{1}_{g_2}=0.
$$

\item[(b)] $W={x_{1}}^{n_1}x_2+{x_2}^{n_2}x_3+\cdots + {x_N}^{n_N}x_1$ is of loop type. Then $l_{g_1} = l_{g_2} = l_{g_1g_2} = N$. By the $G_W$-equivariance of the cup product, we have
\begin{equation*}
h^*(\mathfrak{1}_{g_1} \cup \mathfrak{1}_{g_2}) = h^*(\mathfrak{1}_{g_1}) \cup h^*(\mathfrak{1}_{g_2}),\quad \forall g\in G_W,
\end{equation*}
i.e.
\begin{equation*}
\chi(h)^{-1} (\mathfrak{1}_{g_1} \cup \mathfrak{1}_{g_2}) = \chi(h)^{-2}(\mathfrak{1}_{g_1} \cup \mathfrak{1}_{g_2}),
\end{equation*}
where $\chi(h)$ is defined in (\ref{eq: character}).
Since $G_W\bigcap \mathrm{SL}(N;\mathbb{C})$ is a proper subgroup of $G_W$ (see {Proposition 20} in \cite{BTW16}), we can always find
an element $h$ such that $\chi(h) \neq 1$. It follows that
$
\mathfrak{1}_{g_1} \cup \mathfrak{1}_{g_2} = 0.
$
\item[(c)] $W= {x_{1}}^{n_1}x_2+{x_2}^{n_2}x_3+\cdots + {x_N}^{n_N}$ is of chain type. $x_1$ has nontrivial weight under $g_1, g_2$ and $g_1g_2$. Consider
\[h=\text{diag} (e^{2\pi \sqrt{-1}\over n_1},1,1\cdots 1)\in G_W.\]
The equation
$
h^*(\mathfrak{1}_{g_1} \cup \mathfrak{1}_{g_2}) = h^*(\mathfrak{1}_{g_1}) \cup h^*(\mathfrak{1}_{g_2})
$
implies that
\begin{equation*}
\exp\big(-\f{2\pi\sqrt{-1}}{n_1}\big) (\mathfrak{1}_{g_1} \cup \mathfrak{1}_{g_2}) =
\exp\big(-\f{4\pi\sqrt{-1}}{n_1}\big)(\mathfrak{1}_{g_1} \cup \mathfrak{1}_{g_2}).
\end{equation*}
Since $n_1>1$, we conclude $
\mathfrak{1}_{g_1} \cup \mathfrak{1}_{g_2} = 0.
$
\end{description}
The rest of this paper is devoted to prove part (2) of this theorem.

\end{proof}

\begin{rem}
We follow the argument in \cite{BTW16} in the proof of part (1) above.
\end{rem}

\subsubsection*{Fermat type}
$W = x_1^n$ and $G = \lra{\sigma}\cong \mathbb{Z}/n\mathbb{Z}$, where $n\geqslant 2$ and $\sigma = \exp(\f{2\pi\sqrt{-1}}{n})$. For $1\leq k< n$, let us first find the representative of $\mathfrak{1}_{\sigma^k}$ in Koszul cochains. This amounts to extend $e_1\sigma^k$ into a cochain annihilated by $\p_K+ \df_W$.  It is easy to see that $\df_W (e_1\sigma^k)=0$, hence
$$
\mathfrak{1}_{\sigma^k}=e_1\sigma^k
$$
already does the job. Since $\Upsilon^*(e_1\sigma^k)=\p_1^{\sigma^k}$, we find
$$
\mathfrak{1}_{\sigma^k}\cup \mathfrak{1}_{\sigma^{-k}}=\mathsf{p}\tilde{\Phi}^*(\p_1^{\sigma^k} \p_1^{\sigma^{-k}})\stackrel{\text{Thm}\ \ref{thm:graph summation}}{=}\p_{1,1}^{\sigma^k,\sigma^{-k}}(W)\stackrel{\text{Ex}\ \ref{ex-second-order}}{=}{n\over 1-\sigma^k}x^{n-2}_1.
$$
This proves part (2) of Theorem \ref{thm: cup product formula for General invertible polynomial} in the Fermat case.

\subsubsection*{Loop type}
$W=x_{1}^{n_1}x_2+x_2^{n_2}x_3+\cdots + x_N^{n_N}x_1$. Let us first assume $N\geq 3$.

Let $g=\text{diag}(\lambda_1, \cdots, \lambda_N)\in G_W$, $\lambda_i=e^{2\pi \sqrt{-1}q_i}$ where
$$
n_iq_i+q_{i+1}\in \mathbb{Z}, \quad \forall 1\leq i\leq N, \quad \text{where}\quad q_{N+1}\equiv q_1.
$$
We require $g\neq e$, then $0<q_i<1$ for any $i$.  Let us first find the representative of $\mathfrak{1}_{g}$ in Koszul cochains. This amounts to extend $\kappa^g_0=e_1e_2\cdots e_N g$ into $\kappa^g_{0}+\kappa^g_{-2}+\cdots$ where $\kappa^g_{-2k}\in K^{N-2k}(A,Ag)$ satisfying
\begin{equation*} \p_K(\kappa^g_{-2k-2}) + \tilde{\df}_W(\kappa^g_{-2k})=0.\quad \text{Then}\quad \mathfrak{1}_{g}=\Big[\sum_{0\leq 2k\leq N}\kappa_{-2k}^g\Big].
\end{equation*}
For  $\bm{I}=\{i_1<i_2<\cdots i_m\}$, let us denote $e_{\bm{I}}\coloneqq e_{i_1}\cdots e_{i_k}$. In the loop case, we have
\begin{equation}\label{eq:curving diff on kus loop}
\tilde{\df}_W(ae_{\bm{I}}g)= \sum_{k=1}^{p} (-1)^{k-1}ap^g_{i_k}e_{\bm{I}\setminus\{i_k\}}g,
\end{equation}
where
\begin{equation}
p^g_i\coloneqq
\begin{dcases}
x_N^{n_N}+[n_1]_{\lambda_1}x_1^{n_1-1}x_2&\text{if } i=1,\\
\lambda_{i-1}^{n_{i-1}}x_{i-1}^{n_{i-1}}+[n_i]_{\lambda_i}x_i^{n_i-1}x_{i+1}&\text{if }1<i<N,\\
\lambda_{N-1}^{n_{N-1}}x_{N-1}^{n_{N-1}}+\lambda_1[n_N]_{\lambda_N}x_N^{n_N-1}x_1&\text{if }i=N.
\end{dcases}
\end{equation}

Let us consider the following index set
\begin{equation*}
\mathcal{B}_s \coloneqq \left\{i_1<i_2<\cdots <i_s\mid i_{k+1}-i_k>1, \text{ for } 1\leqslant k <s, \text{
and }i_1+N-i_s>1\right\},
\end{equation*}
and introduce the notation
\begin{equation}
b^g_i\coloneqq \begin{dcases}
\f{\lambda_i^{n_i}}{1-\lambda_i}x_i^{n_i-1}&\text{if } 1\leqslant i<N,\\
\f{(-1)^{N-1}}{1-\lambda_N}x_N^{n_N-1}&\text{if } i=N.
\end{dcases}
\end{equation}
For each $\bm{b} = \{i_1<i_2<\cdots<i_s\}\in \mathcal{B}_s$, define
\begin{equation}\bm{I}_{\bm{b}}\coloneqq
\begin{dcases}
\{i_1,i_1+1,i_2,i_2+1,\cdots i_s,i_s+1\}&\text{while } i_s\neq N,\\
\{1,i_1,i_1+1,i_2,i_2+1,\cdots i_s\} &\text{while } i_s=N.
\end{dcases}
\end{equation}
\begin{lem}\label{lem:loop closed cochain}
Let $\bm{I}_g=\{1,2,\cdots, N\}$ and define
\begin{equation}
\kappa^g_{-2s}\coloneqq \sum_{\bm{b}\in \mathcal{B}_s}\Big(\prod_{i\in\bm{b}}b^g_i\Big)e_{\bm{I}_g\setminus \bm{I_b}}g.
\end{equation}
Then
\[\sum_{N\geqslant 2s\geqslant 0}\kappa^{g}_{-2s}\]
gives a Koszul representative for $\mathfrak{1}_g$.
\end{lem}

\begin{proof}For convenience, we will identify index $i$ with $i+N$. For ever $p^g_{i}$, we denote
\[p_i^g = \big(p^g_i\big)_1+\big(p_i^g\big)_2,\quad \text{where}\quad \big(p^g_{j}\big)_1\!\!\propto\!\! x_{j-1}^{n_{j-1}}, \quad \big(p^g_{j}\big)_2\!\!\propto\!\! x_{j}^{n_{j}-1}x_{j+1}.
\]
For any index set $\bm{I}\subset \bm{I}_g$ with $|\bm{I}| = 2s+1$, denote
\begin{equation*}
\mathcal{H}_{\bm{I}} = \{(\bm{b},j)\mid \bm{b}\in \mathcal{B}_s,j \notin \bm{I}_{\bm{b}} \text{ and } \bm{I}_{\bm{b}}\bigcup \{j\} = \bm{I}\}.
\end{equation*}
For $\bm{I}$ such that $\mathcal{H}_{\bm{I}}\neq \emptyset$, let $\bm{I_0} = \{i_{\bm{I}},i_{\bm{I}}+1,\cdots
i_{\bm{I}}+2\ell_{\bm{I}}\}\subseteq \bm{I}$, such that $i_{\bm{I}}-1,i_{\bm{I}}+2\ell_{\bm{I}}+1\notin \bm{I}$.Then
$$
\mathcal{H}_{\bm{I}}=\{(\bm{b}(\ell),j(\ell))\}_{\ell=0}^{\ell_{\bm{I}}}, \quad \text{where}\quad j(\ell)=i_{\bm{I}}+2\ell.
$$
Note that for $0\leqslant \ell\leqslant\ell_{\bm{I}}-1$,
\begin{equation*}
\big(p^g_{j(\ell)})_2b^g_{j(\ell)+1}+b^g_{j(\ell)}\big(p^g_{j(\ell+1)}\big)_1= 0.
\end{equation*}
Therefore
\begin{align*}\tilde{\df}_W(\kappa^g_{-2s}) &= \sum_{\bm{b}\in
\mathcal{B}_s}\sum_{j_k\notin \bm{I}_{\bm{b}}}(-1)^{k-1}p^g_{j_k}\Big(\prod_{i\in\bm{b}}b^g_i\Big)e_{\bm{I}_g\setminus (\bm{I_b}\bigcup
\{j_k\})} g\nonumber\\
&= \sum_{|\bm{I}| =
2s+1}\sum_{(\bm{b},j)\in\mathcal{H}_{\bm{I}}}(-1)^{\sgn{\bm{I}}}p^g_{j}\Big(\prod_{i\in\bm{b}}b^g_i\Big)e_{\bm{I}_g\setminus\bm{I}} g\nonumber\\
&= \sum_{|\bm{I}| = 2s+1}\sum_{\ell = 0}^{\ell_{\bm{I}}}(-1)^{\sgn{\bm{I}}}\Big(\big(p^g_{j(\ell)}\big)_1+\big(p^g_{j(\ell)}\big)_2\Big)
\Big(\prod_{i\in\bm{b}(\ell)}b^g_i\Big)e_{\bm{I}_g\setminus\bm{I}} g\nonumber\\
&= \sum_{|\bm{I}| = 2s+1}(-1)^{\sgn{\bm{I}}}\Big(\big(p^g_{i_{\bm{I}}}\big)_1\prod_{i\in\bm{b}(0)}b^g_i
+\big(p^g_{i_{\bm{I}}+2\ell_{\bm{I}}}\big)_2\prod_{i\in\bm{b}(\ell_{\bm{I}})}b^g_i\Big)e_{\bm{I}_g\setminus\bm{I}}g,
\end{align*}
where $\sgn{\bm{I}} = \sharp\big(\{1,2,\cdots i_{\bm{I}}-1\}\setminus \bm{I}\big)$. Observe
\begin{align*}
\big(p^g_{i+1}\big)_1 = (1-\lambda_i)x_ib^g_i, \quad
 -\big(p^g_{i}\big)_2 =(1-\lambda_{i+1})x_{i+1}b^g_i,
\end{align*}
we have
\begin{align*}
-\p_K(\kappa^g_{-2s-2}) &= \sum_{|\bm{I}| =
2s+1}(-1)^{\sgn{\bm{I}}}\Big((1-\lambda_{i_{\bm{I}}-1})x_{i_{\bm{I}}-1}\prod_{i\in\bm{b}(\bm{I})^+}b^g_i\nonumber\\
&\quad -(1-\lambda_{i_{\bm{I}}+2\ell_{\bm{I}}+1})x_{i_{\bm{I}}+2\ell_{\bm{I}}+1}\prod_{i\in\bm{b}(\bm{I})^-}b^g_i\Big)e_{\bm{I}_g\setminus
\bm{I}}g,
\end{align*}
where  $\bm{I}_{\bm{b}(\bm{I})^+} = \bm{I}\bigcup\{i_{\bm{I}}-1\}$ and
$\bm{I}_{\bm{b}(\bm{I})^-} = \bm{I}\bigcup\{i_{\bm{I}}+2\ell_{\bm{I}}+1\}$.  We find
\begin{equation*}
\p_K(\kappa^g_{-2s-2}) + \tilde{\df}_W(\kappa^g_{-2s})=0.
\end{equation*}
\end{proof}

Now we can compute $\mathfrak{1}_g\cup \mathfrak{1}_{g^{-1}}$ by
\begin{equation}\label{eq:loop product}
\mathsf{p} \circ \tilde \Phi^*\left(\Upsilon^*(\kappa^g_0+\kappa^g_{-2}+\cdots)\cup\Upsilon^*(\kappa^{g^{-1}}_0+\kappa^{g^{-1}}_{-2}+\cdots)\right).
\end{equation}
The following Lemma is a direct consequence of {Theorem \ref{thm:graph summation}} and {Lemma \ref{lem:loop closed cochain}}.
\begin{lem}\label{lem:loop graph summation}
$$
\mathfrak{1}_g\cup \mathfrak{1}_{g^{-1}}=\sum_{\Gamma}\val{\Gamma}
$$ is a sum of `values' for all  graphs $\Gamma$
with
\begin{description}
  \item[(a)] a set of vertices  $V(\Gamma)$ indexed by $1,2,\cdots N$ (coloured by `$+$') and $1,2,\cdots N$ (coloured by`$-$')
     ;
  \item[(b)] a set of edges $E(\Gamma)$ of three types such that each vertex has valency 1:
  \begin{description}
    \item[Type `$++$':] edges having sources coloured $+$ and targets coloured $+$, indexed by
        $\pplr{v_{\text{source}}}{v_{\text{target}}}$, such that for each edge $\pplr{i}{j}$, $j\equiv i+1 \mod N$,
    \item[Type `$--$':] edges having sources coloured $-$ and targets coloured $-$, indexed by
        $\mmlr{v_{\text{source}}}{v_{\text{target}}}$, such that for each edge $\mmlr{i}{j}$, $j\equiv i+1 \mod N$,
    \item[Type `$+-$':] edges having sources coloured $+$ and targets coloured $-$, indexed by
        $\pmlr{v_{\text{source}}}{v_{\text{target}}}$, such that for each edge $\pmlr{i}{j}$, $i=j$ or $i =j+1$ or $(i,j) = (N,1)$,
  \end{description}
\end{description}
Given such a graph $\Gamma$, we assign the value of each edge by
\begin{align*}
\val{\pmlr{i}{i}}~~~~~~~&= \f{[n_{i}]_{\lambda_{i}}-n_i}{\lambda_i-1}x_{i}^{n_{i}-2}x_{i+1},\\
\val{\pmlr{i}{i-1}}&= [n_{i-1}]_{\lambda_{i-1}}x_{i-1}^{n_{i-1}-1},\\
\val{\pmlr{N}{1}}~~~~&= [n_{N}]_{\lambda_{N}}x_{N}^{n_{N}-1},
\end{align*}
\begin{align*}
\val{\pplr{i}{i+1}}&= b_i^g = \begin{dcases}
\f{\lambda_i^{n_i}}{1-\lambda_i}x_i^{n_i-1}&\text{while } 1\leqslant i<N,\\
\f{(-1)^{N-1}}{1-\lambda_N}x_N^{n_N-1}&\text{while } i=N,
\end{dcases}\\
\val{\mmlr{i}{i+1}}&= \leftidx{^g}b_i^{g^{-1}} = \begin{dcases}
-\f{1}{1-\lambda_i}x_i^{n_i-1}&\text{while } 1\leqslant i<N,\\
\f{(-1)^{N}\lambda_{N}^{n_N}}{1-\lambda_N}x_N^{n_N-1}&\text{while } i=N,
\end{dcases}
\end{align*}
and the value of a graph $\Gamma$ is defined by
\begin{equation*}
\val{\Gamma}\coloneqq (-1)^{\sgn{\Gamma}+\f{(s_{\Gamma}-1)s_{\Gamma}}{2}}\prod_{e\in E(\Gamma)} \val{e}.
\end{equation*}
Here  $s_\Gamma$ is the number of
edges of type `$+-$' of $\Gamma$. If the type `$+-$' edges have source index  $i_1<i_2<\cdots< i_{s_\Gamma}$ and target index $j_1<j_2<\cdots <j_{s_\Gamma}$ such that it connects $i_k$ to $j_{\sigma(k)}$, then $(-1)^{\sgn{\Gamma}}$ is the sign of the permutation $\sigma$.

\end{lem}

\begin{ex}
For $N=3$, we need to consider the following graphs.
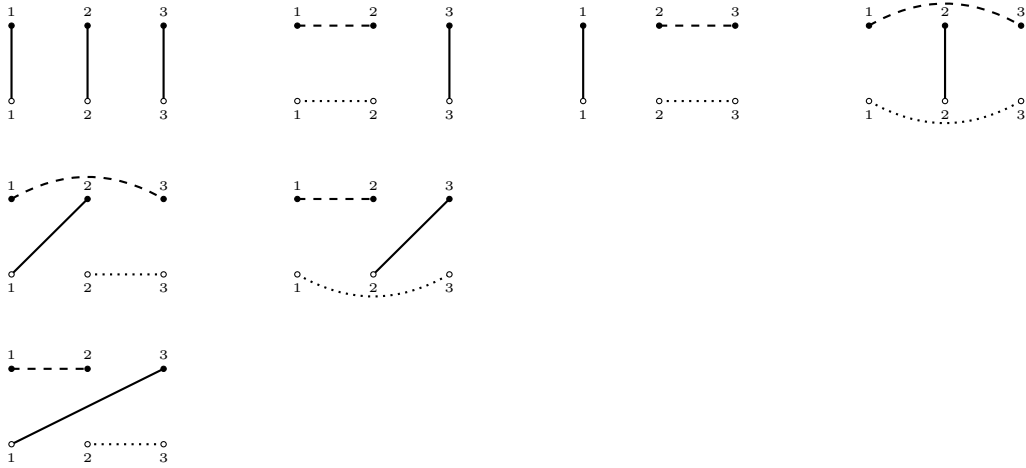
\begin{align*}
&\tikz[baseline=0]{
\draw[thick] (0,1) -- (0,0);
\draw[thick] (1,1) -- (1,0);
\draw[thick] (2,1) -- (2,0);
\filldraw[fill=black] (0,1) node[above,font=\tiny] {$1$} circle (1pt);
\filldraw[fill=white] (0,0) node[below,font=\tiny] {$1$} circle (1pt);
\filldraw[fill=black] (1,1) node[above,font=\tiny] {$2$} circle (1pt);
\filldraw[fill=white] (1,0) node[below,font=\tiny] {$2$} circle (1pt);
\filldraw[fill=black] (2,1) node[above,font=\tiny] {$3$} circle (1pt);
\filldraw[fill=white] (2,0) node[below,font=\tiny] {$3$} circle (1pt);
}\qquad\qquad
\tikz[baseline=0]{
\draw[thick,dashed] (0,1) -- (1,1);
\draw[thick,dotted] (0,0) -- (1,0);
\draw[thick] (2,1) -- (2,0);
\filldraw[fill=black] (0,1) node[above,font=\tiny] {$1$} circle (1pt);
\filldraw[fill=white] (0,0) node[below,font=\tiny] {$1$} circle (1pt);
\filldraw[fill=black] (1,1) node[above,font=\tiny] {$2$} circle (1pt);
\filldraw[fill=white] (1,0) node[below,font=\tiny] {$2$} circle (1pt);
\filldraw[fill=black] (2,1) node[above,font=\tiny] {$3$} circle (1pt);
\filldraw[fill=white] (2,0) node[below,font=\tiny] {$3$} circle (1pt);
}\qquad\qquad
\tikz[baseline=0]{
\draw[thick,dashed] (1,1) -- (2,1);
\draw[thick,dotted] (1,0) -- (2,0);
\draw[thick] (0,1) -- (0,0);
\filldraw[fill=black] (0,1) node[above,font=\tiny] {$1$} circle (1pt);
\filldraw[fill=white] (0,0) node[below,font=\tiny] {$1$} circle (1pt);
\filldraw[fill=black] (1,1) node[above,font=\tiny] {$2$} circle (1pt);
\filldraw[fill=white] (1,0) node[below,font=\tiny] {$2$} circle (1pt);
\filldraw[fill=black] (2,1) node[above,font=\tiny] {$3$} circle (1pt);
\filldraw[fill=white] (2,0) node[below,font=\tiny] {$3$} circle (1pt);
}\qquad\qquad
\tikz[baseline=0]{
\draw[thick,dashed] (0,1) to[bend left] (2,1);
\draw[thick] (1,1) -- (1,0);
\draw[thick,dotted] (0,0) to[bend right] (2,0);
\filldraw[fill=black] (0,1) node[above,font=\tiny] {$1$} circle (1pt);
\filldraw[fill=white] (0,0) node[below,font=\tiny] {$1$} circle (1pt);
\filldraw[fill=black] (1,1) node[above,font=\tiny] {$2$} circle (1pt);
\filldraw[fill=white] (1,0) node[below,font=\tiny] {$2$} circle (1pt);
\filldraw[fill=black] (2,1) node[above,font=\tiny] {$3$} circle (1pt);
\filldraw[fill=white] (2,0) node[below,font=\tiny] {$3$} circle (1pt);
}\\
\\
&\tikz[baseline=0]{
\draw[thick,dashed] (0,1) to[bend left] (2,1);
\draw[thick] (1,1) -- (0,0);
\draw[thick,dotted] (1,0) -- (2,0);
\filldraw[fill=black] (0,1) node[above,font=\tiny] {$1$} circle (1pt);
\filldraw[fill=white] (0,0) node[below,font=\tiny] {$1$} circle (1pt);
\filldraw[fill=black] (1,1) node[above,font=\tiny] {$2$} circle (1pt);
\filldraw[fill=white] (1,0) node[below,font=\tiny] {$2$} circle (1pt);
\filldraw[fill=black] (2,1) node[above,font=\tiny] {$3$} circle (1pt);
\filldraw[fill=white] (2,0) node[below,font=\tiny] {$3$} circle (1pt);
}\qquad\qquad
\tikz[baseline=0]{
\draw[thick] (2,1) -- (1,0);
\draw[thick,dashed] (0,1) -- (1,1);
\draw[thick,dotted] (0,0) to[bend right] (2,0);
\filldraw[fill=black] (0,1) node[above,font=\tiny] {$1$} circle (1pt);
\filldraw[fill=white] (0,0) node[below,font=\tiny] {$1$} circle (1pt);
\filldraw[fill=black] (1,1) node[above,font=\tiny] {$2$} circle (1pt);
\filldraw[fill=white] (1,0) node[below,font=\tiny] {$2$} circle (1pt);
\filldraw[fill=black] (2,1) node[above,font=\tiny] {$3$} circle (1pt);
\filldraw[fill=white] (2,0) node[below,font=\tiny] {$3$} circle (1pt);
}\qquad\qquad\\
\\
&\tikz[baseline=0]{
\draw[thick] (2,1) -- (0,0);
\draw[thick,dashed] (0,1) -- (1,1);
\draw[thick,dotted] (2,0) -- (1,0);
\filldraw[fill=black] (0,1) node[above,font=\tiny] {$1$} circle (1pt);
\filldraw[fill=white] (0,0) node[below,font=\tiny] {$1$} circle (1pt);
\filldraw[fill=black] (1,1) node[above,font=\tiny] {$2$} circle (1pt);
\filldraw[fill=white] (1,0) node[below,font=\tiny] {$2$} circle (1pt);
\filldraw[fill=black] (2,1) node[above,font=\tiny] {$3$} circle (1pt);
\filldraw[fill=white] (2,0) node[below,font=\tiny] {$3$} circle (1pt);
}\qquad\qquad
\end{align*}
\end{ex}

There are three types of graphs in Lemma \ref{lem:loop graph summation}:
\begin{description}
  \item[Type $\Rmnum{1}$:]Graphs with all the type `$+-$' edges indexed by $\pmlr{i}{i}$. Then, the other edges comes in pairs;
  \item[Type $\Rmnum{2}$:]Graphs with all the type `$+-$' edges indexed by $\pmlr{i}{i-1}$, for $i\neq 1$. Then, the other edges comes in pairs;
  \item[Type $\Rmnum{3}$:] Graphs with all the type `$+-$' edges indexed by $\pmlr{i}{i-1}$ and by $\pmlr{N}{1}$. Then, there exists edges
      $\pplr{1}{2}$ and $\mmlr{N-1}{N}$ and the other edges comes in pairs.
\end{description}
We can simplify the graphs by gluing vertices with the same index but different colours, then the three types of graphs look like as follows,
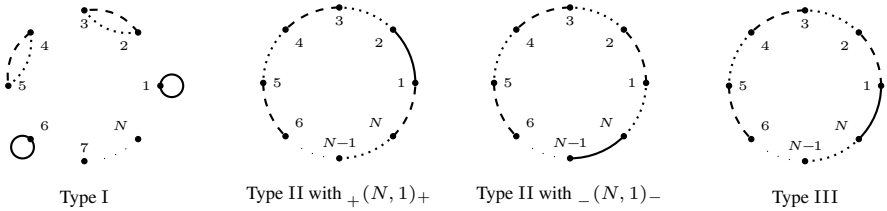
\begin{equation*}
\tikz{
\draw[thick] (0:1) arc (180:540:1.5mm);
\draw[thick,dashed] (45:1) to[bend right] (90:1);
\draw[thick,dotted] (45:1) to[bend left] (90:1);
\filldraw[fill=black] (0:1) node[left,font=\tiny] {$1$} circle (1pt);
\filldraw[fill=black] (45:1) node[anchor= north east,font=\tiny] {$2$} circle (1pt);
\filldraw[fill=black] (90:1) node[anchor= north,font=\tiny] {$3$} circle (1pt);
\draw[thick,dashed] (135:1) to[bend right] (180:1);
\draw[thick,dotted] (135:1) to[bend left] (180:1);
\filldraw[fill=black] (135:1) node[anchor= north west,font=\tiny] {$4$} circle (1pt);
\filldraw[fill=black] (180:1) node[anchor= west,font=\tiny] {$5$} circle (1pt);
\draw[thick] (225:1) arc (45:405:1.5mm);
\filldraw[fill=black] (225:1) node[anchor=south west,font=\tiny] {$6$} circle (1pt);
\draw[loosely dotted] (285:1) arc (-75:-50:1cm);
\filldraw[fill=black] (270:1) node[anchor=south,font=\tiny] {$7$} circle (1pt);
\filldraw[fill=black] (315:1) node[anchor=south east,font=\tiny] {$N$} circle (1pt);
\node[font=\scriptsize] at (0,-1.5) {Type $\Rmnum{1}$};
}\qquad
\tikz{
\draw[thick,dashed] (-45:1) arc (-45:0:1);
\draw[thick,dotted] (-90:1) arc (-90:-45:1);
\draw[loosely dotted] (240:1) arc (-120:-95:1cm);
\draw[thick] (0:1) arc (0:45:1cm);
\draw[thick,dotted] (45:1) arc (45:90:1cm);
\draw[thick,dashed] (90:1) arc (90:135:1cm);
\draw[thick,dotted] (135:1) arc (135:180:1cm);
\draw[thick,dashed] (180:1) arc (180:225:1cm);
\filldraw[fill=black] (0:1) node[left,font=\tiny] {$1$} circle (1pt);
\filldraw[fill=black] (45:1) node[anchor= north east,font=\tiny] {$2$} circle (1pt);
\filldraw[fill=black] (90:1) node[anchor= north,font=\tiny] {$3$} circle (1pt);
\filldraw[fill=black] (135:1) node[anchor= north west,font=\tiny] {$4$} circle (1pt);
\filldraw[fill=black] (180:1) node[anchor= west,font=\tiny] {$5$} circle (1pt);
\filldraw[fill=black] (225:1) node[anchor=south west,font=\tiny] {$6$} circle (1pt);
\filldraw[fill=black] (270:1) node[anchor=south,font=\tiny] {$N\!\!-\!\!1$} circle (1pt);
\filldraw[fill=black] (315:1) node[anchor=south east,font=\tiny] {$N$} circle (1pt);
\node[font=\scriptsize] at (0,-1.5) {Type $\Rmnum{2}$ with $\pplr{N}{1}$};
}\quad
\tikz{
\draw[thick,dotted] (-45:1) arc (-45:0:1);
\draw[thick] (-90:1) arc (-90:-45:1);
\draw[loosely dotted] (240:1) arc (-120:-95:1cm);
\draw[thick,dashed] (0:1) arc (0:45:1cm);
\draw[thick,dotted] (45:1) arc (45:90:1cm);
\draw[thick,dashed] (90:1) arc (90:135:1cm);
\draw[thick,dotted] (135:1) arc (135:180:1cm);
\draw[thick,dashed] (180:1) arc (180:225:1cm);
\filldraw[fill=black] (0:1) node[left,font=\tiny] {$1$} circle (1pt);
\filldraw[fill=black] (45:1) node[anchor= north east,font=\tiny] {$2$} circle (1pt);
\filldraw[fill=black] (90:1) node[anchor= north,font=\tiny] {$3$} circle (1pt);
\filldraw[fill=black] (135:1) node[anchor= north west,font=\tiny] {$4$} circle (1pt);
\filldraw[fill=black] (180:1) node[anchor= west,font=\tiny] {$5$} circle (1pt);
\filldraw[fill=black] (225:1) node[anchor=south west,font=\tiny] {$6$} circle (1pt);
\filldraw[fill=black] (270:1) node[anchor=south,font=\tiny] {$N\!\!-\!\!1$} circle (1pt);
\filldraw[fill=black] (315:1) node[anchor=south east,font=\tiny] {$N$} circle (1pt);
\node[font=\scriptsize] at (0,-1.5) {Type $\Rmnum{2}$ with $\mmlr{N}{1}$};
}\qquad
\tikz{
\draw[thick] (-45:1) arc (-45:0:1);
\draw[thick,dotted] (-90:1) arc (-90:-45:1);
\draw[loosely dotted] (240:1) arc (-120:-95:1cm);
\draw[thick,dashed] (0:1) arc (0:45:1cm);
\draw[thick,dotted] (45:1) arc (45:90:1cm);
\draw[thick,dashed] (90:1) arc (90:135:1cm);
\draw[thick,dotted] (135:1) arc (135:180:1cm);
\draw[thick,dashed] (180:1) arc (180:225:1cm);
\filldraw[fill=black] (0:1) node[left,font=\tiny] {$1$} circle (1pt);
\filldraw[fill=black] (45:1) node[anchor= north east,font=\tiny] {$2$} circle (1pt);
\filldraw[fill=black] (90:1) node[anchor= north,font=\tiny] {$3$} circle (1pt);
\filldraw[fill=black] (135:1) node[anchor= north west,font=\tiny] {$4$} circle (1pt);
\filldraw[fill=black] (180:1) node[anchor= west,font=\tiny] {$5$} circle (1pt);
\filldraw[fill=black] (225:1) node[anchor=south west,font=\tiny] {$6$} circle (1pt);
\filldraw[fill=black] (270:1) node[anchor=south,font=\tiny] {$N\!\!-\!\!1$} circle (1pt);
\filldraw[fill=black] (315:1) node[anchor=south east,font=\tiny] {$N$} circle (1pt);
\node[font=\scriptsize] at (0,-1.5) {Type $\Rmnum{3}$};
}
\end{equation*}
Note that for $1\leqslant i<N$, we have
\begin{equation*}
\val{\pmlr{i+1}{i}} = \f{1-\lambda_i^{n_i}}{1-\lambda_i}x_i^{n_i-1} = -\val{\pplr{i}{i+1}}-\val{\mmlr{i}{i+1}},
\end{equation*}
which means we can replace all these edges of type `$+-$' in graphs of type $\Rmnum{2}$ by edges of type `$++$' of type `$--$' and multiply by $(-1)^{s_\Gamma}$. Hence, graphs with
\[
\tikz{\draw[thick,dotted](0,0) --(1,0);
\draw[thick,dashed](1,0)--(2,0);
\draw[loosely dotted] (-0.6,0) -- (-0.2,0);
\draw[loosely dotted] (2.2,0) -- (2.6,0);
\filldraw[fill=black] (0,0) node[above,font=\scriptsize] {$i\!\!-\!\!1$} circle (1pt);
\filldraw[fill=black] (1,0) node[above,font=\scriptsize] {$i$} circle (1pt);
\filldraw[fill=black] (2,0) node[above,font=\scriptsize] {$i\!\!+\!\!1$} circle (1pt);}
\]
for $2\leqslant i\leqslant N-1$ will be cancelled and we are left with the graphs
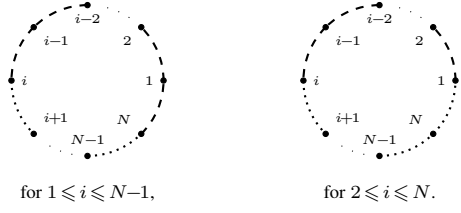
\begin{equation*}
\tikz{
\draw[thick,dashed] (-45:1) arc (-45:0:1);
\draw[thick,dotted] (-90:1) arc (-90:-45:1);
\draw[loosely dotted] (240:1) arc (-120:-95:1cm);
\draw[thick,dashed] (0:1) arc (0:45:1cm);
\draw[loosely dotted] (75:1) arc (75:50:1cm);
\draw[thick,dashed] (90:1) arc (90:135:1cm);
\draw[thick,dashed] (135:1) arc (135:180:1cm);
\draw[thick,dotted] (180:1) arc (180:225:1cm);
\filldraw[fill=black] (0:1) node[left,font=\tiny] {$1$} circle (1pt);
\filldraw[fill=black] (45:1) node[anchor= north east,font=\tiny] {$2$} circle (1pt);
\filldraw[fill=black] (90:1) node[anchor= north,font=\tiny] {$i\!\!-\!\!2$} circle (1pt);
\filldraw[fill=black] (135:1) node[anchor= north west,font=\tiny] {$i\!\!-\!\!1$} circle (1pt);
\filldraw[fill=black] (180:1) node[anchor= west,font=\tiny] {$i$} circle (1pt);
\filldraw[fill=black] (225:1) node[anchor=south west,font=\tiny] {$i\!\!+\!\!1$} circle (1pt);
\filldraw[fill=black] (270:1) node[anchor=south,font=\tiny] {$N\!\!-\!\!1$} circle (1pt);
\filldraw[fill=black] (315:1) node[anchor=south east,font=\tiny] {$N$} circle (1pt);
\node[font=\scriptsize] at (0,-1.5) {for $1\!\leqslant \!i\!\leqslant\! N\!\!-\!\!1$,};
}\quad\qquad\qquad
\tikz{
\draw[thick,dotted] (-45:1) arc (-45:0:1);
\draw[thick,dotted] (-90:1) arc (-90:-45:1);
\draw[loosely dotted] (240:1) arc (-120:-95:1cm);
\draw[thick,dashed] (0:1) arc (0:45:1cm);
\draw[loosely dotted] (75:1) arc (75:50:1cm);
\draw[thick,dashed] (90:1) arc (90:135:1cm);
\draw[thick,dashed] (135:1) arc (135:180:1cm);
\draw[thick,dotted] (180:1) arc (180:225:1cm);
\filldraw[fill=black] (0:1) node[left,font=\tiny] {$1$} circle (1pt);
\filldraw[fill=black] (45:1) node[anchor= north east,font=\tiny] {$2$} circle (1pt);
\filldraw[fill=black] (90:1) node[anchor= north,font=\tiny] {$i\!\!-\!\!2$} circle (1pt);
\filldraw[fill=black] (135:1) node[anchor= north west,font=\tiny] {$i\!\!-\!\!1$} circle (1pt);
\filldraw[fill=black] (180:1) node[anchor= west,font=\tiny] {$i$} circle (1pt);
\filldraw[fill=black] (225:1) node[anchor=south west,font=\tiny] {$i\!\!+\!\!1$} circle (1pt);
\filldraw[fill=black] (270:1) node[anchor=south,font=\tiny] {$N\!\!-\!\!1$} circle (1pt);
\filldraw[fill=black] (315:1) node[anchor=south east,font=\tiny] {$N$} circle (1pt);
\node[font=\scriptsize] at (0,-1.5) {for $2\!\leqslant \!i\!\leqslant \!N$.};
}
\end{equation*}
multiplied by $(-1)^{\f{(N-1)(N-2)}{2}}$ (for these $\Gamma$, $\sgn{\Gamma}=0$ and $s_\Gamma = N-2$).

We can do the same thing to graphs of type $\Rmnum{3}$. Then the rest graphs are
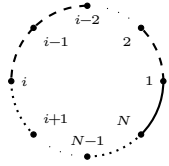
\begin{equation*}
\tikz{
\draw[thick] (-45:1) arc (-45:0:1);
\draw[thick,dotted] (-90:1) arc (-90:-45:1);
\draw[loosely dotted] (240:1) arc (-120:-95:1cm);
\draw[thick,dashed] (0:1) arc (0:45:1cm);
\draw[loosely dotted] (75:1) arc (75:50:1cm);
\draw[thick,dashed] (90:1) arc (90:135:1cm);
\draw[thick,dashed] (135:1) arc (135:180:1cm);
\draw[thick,dotted] (180:1) arc (180:225:1cm);
\filldraw[fill=black] (0:1) node[left,font=\tiny] {$1$} circle (1pt);
\filldraw[fill=black] (45:1) node[anchor= north east,font=\tiny] {$2$} circle (1pt);
\filldraw[fill=black] (90:1) node[anchor= north,font=\tiny] {$i\!\!-\!\!2$} circle (1pt);
\filldraw[fill=black] (135:1) node[anchor= north west,font=\tiny] {$i\!\!-\!\!1$} circle (1pt);
\filldraw[fill=black] (180:1) node[anchor= west,font=\tiny] {$i$} circle (1pt);
\filldraw[fill=black] (225:1) node[anchor=south west,font=\tiny] {$i\!\!+\!\!1$} circle (1pt);
\filldraw[fill=black] (270:1) node[anchor=south,font=\tiny] {$N\!\!-\!\!1$} circle (1pt);
\filldraw[fill=black] (315:1) node[anchor=south east,font=\tiny] {$N$} circle (1pt);
}
\end{equation*}
for $2\leqslant i\leqslant N-1$ and multiplied by $(-1)^{\f{(N-2)(N-3)}{2}}$ (for these $\Gamma$, $\sgn{\Gamma}= N-3$ and
$s_\Gamma = N-2$). Since
\begin{equation*}
(-1)^N\val{\pmlr{N}{1}} = -\val{\pplr{N}{1}} - \val{\mmlr{N}{1}},
\end{equation*}
the sum of graphs of type $\Rmnum{2}$ and type $\Rmnum{3}$ gives only two graphs
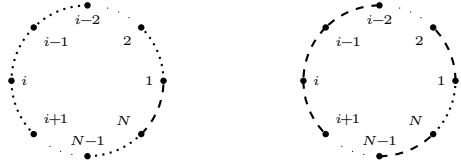
\begin{equation*}
\tikz{
\draw[thick,dashed] (-45:1) arc (-45:0:1);
\draw[thick,dotted] (-90:1) arc (-90:-45:1);
\draw[loosely dotted] (240:1) arc (-120:-95:1cm);
\draw[thick,dotted] (0:1) arc (0:45:1cm);
\draw[loosely dotted] (75:1) arc (75:50:1cm);
\draw[thick,dotted] (90:1) arc (90:135:1cm);
\draw[thick,dotted] (135:1) arc (135:180:1cm);
\draw[thick,dotted] (180:1) arc (180:225:1cm);
\filldraw[fill=black] (0:1) node[left,font=\tiny] {$1$} circle (1pt);
\filldraw[fill=black] (45:1) node[anchor= north east,font=\tiny] {$2$} circle (1pt);
\filldraw[fill=black] (90:1) node[anchor= north,font=\tiny] {$i\!\!-\!\!2$} circle (1pt);
\filldraw[fill=black] (135:1) node[anchor= north west,font=\tiny] {$i\!\!-\!\!1$} circle (1pt);
\filldraw[fill=black] (180:1) node[anchor= west,font=\tiny] {$i$} circle (1pt);
\filldraw[fill=black] (225:1) node[anchor=south west,font=\tiny] {$i\!\!+\!\!1$} circle (1pt);
\filldraw[fill=black] (270:1) node[anchor=south,font=\tiny] {$N\!\!-\!\!1$} circle (1pt);
\filldraw[fill=black] (315:1) node[anchor=south east,font=\tiny] {$N$} circle (1pt);
}\quad\qquad\qquad
\tikz{
\draw[thick,dotted] (-45:1) arc (-45:0:1);
\draw[thick,dashed] (-90:1) arc (-90:-45:1);
\draw[loosely dotted] (240:1) arc (-120:-95:1cm);
\draw[thick,dashed] (0:1) arc (0:45:1cm);
\draw[loosely dotted] (75:1) arc (75:50:1cm);
\draw[thick,dashed] (90:1) arc (90:135:1cm);
\draw[thick,dashed] (135:1) arc (135:180:1cm);
\draw[thick,dashed] (180:1) arc (180:225:1cm);
\filldraw[fill=black] (0:1) node[left,font=\tiny] {$1$} circle (1pt);
\filldraw[fill=black] (45:1) node[anchor= north east,font=\tiny] {$2$} circle (1pt);
\filldraw[fill=black] (90:1) node[anchor= north,font=\tiny] {$i\!\!-\!\!2$} circle (1pt);
\filldraw[fill=black] (135:1) node[anchor= north west,font=\tiny] {$i\!\!-\!\!1$} circle (1pt);
\filldraw[fill=black] (180:1) node[anchor= west,font=\tiny] {$i$} circle (1pt);
\filldraw[fill=black] (225:1) node[anchor=south west,font=\tiny] {$i\!\!+\!\!1$} circle (1pt);
\filldraw[fill=black] (270:1) node[anchor=south,font=\tiny] {$N\!\!-\!\!1$} circle (1pt);
\filldraw[fill=black] (315:1) node[anchor=south east,font=\tiny] {$N$} circle (1pt);
}
\end{equation*}
multiplied by $(-1)^{\f{(N-1)(N-2)}{2}} = (-1)^{\f{(N-1)N}{2}}(-1)^{N-1}$. Let us consider the matrix
\begin{equation}\label{eq:quantum hess loop}
\mathrm{H}^g_W= \begin{pmatrix}
h_{11}&h_{12}&&&&h_{1N}\\
h_{21}&h_{22}&h_{23}&&&\\
&h_{32}&h_{33}&h_{34}&&\\
&&&\ddots&&\\
&&&&\ddots&h_{(N\!-\!1)N}\\
h_{N1}&&&&h_{N(N\!-\!1)}&h_{NN}
\end{pmatrix},
\end{equation}
where the entries are defined by
\begin{equation*}h_{ij} \coloneqq \begin{dcases}
\val{\pmlr{i}{i}}& \text{if } i=j,\\
\val{\pplr{i}{i+1}}& \text{if } N\geqslant j=i+1\geqslant 2,\\
\val{\mmlr{N}{1}} & \text{if } (i,j) = (N,1),\\
\val{\mmlr{i-1}{i}} & \text{if }N-1\geqslant j=i-1\geqslant 1,\\
\val{\pplr{N}{1}} & \text{if } (i,j) = (1,N),\\
0 & \text{otherwise}
\end{dcases}
\end{equation*}
According to the discussion above, the sum of Type I, II, III graphs leads to
\begin{align}\label{det-formula-loop}
\mathfrak{1}_g\cup\mathfrak{1}_{g^{-1}} = (-1)^{\f{(N-1)N}{2}}\det \mathrm{H}^{g}_W.
\end{align}
Note that $\det \mathrm{H}^{g}_W$ is valued in $\jac(W)$ where
\begin{equation*}
x_{i-1}^{n_{i-1}}+n_ix_i^{n_i-1}x_{i+1}=0,\quad 1\leq i\leq N.
\end{equation*}
Here we always identify indices $N=0, N+1=1$.

The two graphs of type $\Rmnum{2}$ and type $\Rmnum{3}$ contribute the following two terms in $\det \mathrm{H}^{g}_W$
\begin{align*}&(-1)^{\f{(N-1)N}{2}}(-1)^{N-1}(h_{12}h_{23}\cdots h_{N1}+h_{21}h_{32}\cdots h_{1N})\nonumber\\
 =&(-1)^{\f{(N-1)N}{2}}\f{-[n_1]_{\lambda_1}[n_2]_{\lambda_2}\cdots [n_N]_{\lambda_N}-1}{(\lambda_1-1)(\lambda_2-1)\cdots
 (\lambda_N-1)}x_1^{n_1-1}\cdots x_N^{n_N-1}.
\end{align*}

To compute the contribution of type $\Rmnum{1}$ graphs, we observe that $h_{i(i+1)}$ and $h_{(i+1)i}$ always come in pairs. The contribution of type I graphs does not change if we replace $h_{i(i+1)}, h_{(i+1)i}$ by other values $\tilde h_{i(i+1)}, \tilde h_{(i+1)i}$ as long as $h_{i(i+1)}h_{(i+1)i}=\tilde h_{i(i+1)}\tilde h_{(i+1)i}$ holds for all $1\leq i\leq N$. In $\jac(W)$, we have
\begin{align*}h_{i(i+1)}h_{(i+1)i} &= -\f{\lambda_i^{n_i}}{(1-\lambda_i)^2}x_i^{2n_i-2}\nonumber\\
&=  -\f{(1-\lambda_i^{n_i})n_{i+1}}{(1-\lambda_i)^2(1-\lambda_{i+1})}x_i^{n_i-2}x_{i+1}^{n_{i+1}-1}x_{i+2}\nonumber\\
&= \lr{\f{[n_i]_{\lambda_i}}{\lambda_i-1}x_i^{n_{i}-2}x_{i+1}}\lr{-\f{n_{i+1}}{\lambda_{i+1}-1}x_{i+1}^{n_{i+1}-2}x_{i+2}}.
\end{align*}
Let us replace $h_{i(i+1)}, h_{(i+1)i}$ by
\begin{equation*}\begin{dcases}\tilde{h}_{i(i+1)} = \f{[n_i]_{\lambda_i}}{\lambda_i-1}x_i^{n_{i}-2}x_{i+1},\\
\tilde{h}_{(i+1)i} =-\f{n_{i+1}}{\lambda_{i+1}-1}x_{i+1}^{n_{i+1}-2}x_{i+2}.
\end{dcases}\end{equation*}
Observe that
\begin{equation*}
h_{ii} = \tilde{h}_{i(i+1)}+\tilde{h}_{i(i-1)}.
\end{equation*}
It leads to cancellation of sum of type $\Rmnum{1}$ graphs that we are left with only two terms
\begin{align*}
&(-1)^{\f{(N-1)N}{2}}\big(\tilde{h}_{12}{\tilde h_{23}}\cdots \tilde{h}_{N1}+\tilde{h}_{21}\tilde{h}_{32}\cdots\tilde{h}_{1N}\big)\nonumber\\
=&(-1)^{\f{(N-1)N}{2}}\f{[n_1]_{\lambda_1}[n_2]_{\lambda_2}\cdots [n_N]_{\lambda_N}\!+\!(-1)^Nn_1n_2\cdots
n_N}{(\lambda_1-1)(\lambda_2-1)\cdots (\lambda_N-1)}x_1^{n_1-1}\cdots x_N^{n_N-1}.
\end{align*}

The sum of all the above contributions proves part (2) of Theorem \ref{thm: cup product formula for General invertible polynomial} in the loop case for $N\geq 3$.

\begin{rem}
In $N=2$ case, we can choose
\begin{equation*}
\kappa_{-2}^g \coloneqq b^g_1+b^g_2 =  \f{\lambda_1^{n_1}}{1-\lambda_1}x_1^{n_1-1}g-\f{1}{1-\lambda_2}x_2^{n_2-1}g,
\end{equation*}
and define matrix
\begin{equation*}
\mathrm{H}^g_W= \begin{pmatrix}
\f{[n_1]_{\lambda_1}-n_1}{\lambda_1-1}x_1^{n_1-2}x_2 &
\f{\lambda_1^{n_1}}{1-\lambda_1}x_1^{n_1-1}-\f{1}{1-\lambda_2} x_2^{n_2-1}\\[1.2em]
-\f{1}{1-\lambda_1} x_1^{n_1-1}+\f{\lambda_2^{n_2}}{1-\lambda_2} x_2^{n_2-1} &
\f{[n_2]_{\lambda_2}-n_2}{\lambda_2-1}x_1 x_2^{n_2-2}
\end{pmatrix}.
\end{equation*}
We leave it to the reader to check directly that
\begin{equation*}
\mathfrak{1}_g\cup\mathfrak{1}_{g^{-1}}  = \f{1-n_1n_2}{(\lambda_1-1)(\lambda_2-1)}x_1^{n_1-1}x_2^{n_2-1}.
\end{equation*}
\end{rem}
This proves part (2) of Theorem \ref{thm: cup product formula for General invertible polynomial} in the loop case for $N=2$.

\subsubsection*{Chain type}
$W= x_{1}^{n_1}x_2+x_2^{n_2}x_3+\cdots + x_N^{n_N}$ ($N\geq 2$). Let $g=\text{diag}(\lambda_1, \cdots, \lambda_N)\in G_W$, $\lambda_i=e^{2\pi \sqrt{-1}q_i}$ where
\begin{equation*} 
\begin{dcases}
\lambda_i^{n_i}\lambda_{i+1}=1& \text{if } 1\leqslant i <l_g,\\
 \lambda_i^{n_i}=1& \text{if } i=l_g,\\
\lambda_{i}=1& \text{if } l_g<i\leqslant N,
\end{dcases}
\end{equation*}
Here $\lambda_i\neq 1$ for $i\leq l_g$. Let us find a Koszul representative of $\mathfrak{1}_g$.

Let $\bm{I}_g = \{1,2,\cdots l_g\}$. For  $\bm{I}=\{i_1<i_2<\cdots i_m\}$, let us denote $e_{\bm{I}}\coloneqq e_{i_1}\cdots e_{i_k}$. In the chain case,
\begin{equation*}
\tilde{\df}_W(age_{\bm{I}})= \sum_{k=1}^{p} (-1)^{k-1}ap^g_{i_k}ge_{\bm{I}\setminus\{i_k\}},
\end{equation*}
where
\begin{equation*}
p^g_i=
\begin{dcases}
[n_1]_{\lambda_1}x_1^{n_1-1}x_2&\text{if } i=1,\\
\lambda_{i-1}^{n_{i-1}}x_{i-1}^{n_{i-1}}+[n_i]_{\lambda_i}x_i^{n_i-1}x_{i+1}&\text{if }1<i<l_g,\\
\lambda_{l_g-1}^{n_{l_g-1}}x_{l_g-1}^{n_{l_g-1}}&\text{if }i=l_g.
\end{dcases}
\end{equation*}

Let us define the index set
\begin{equation*}
\mathcal{B}_s \coloneqq \left\{\{i_1<i_2<\cdots <i_s\}\subset \bm{I}_g\mid i_{k+1}-i_k>1, \text{ for } 1\leqslant k <s, \text{and }
i_s<l_g\right\},
\end{equation*}
and
\begin{equation*}
b^g_i= \f{\lambda_i^{n_i}}{1-\lambda_i}x_i^{n_i-1}, \quad 1\leqslant i\leqslant l_g-1.
\end{equation*}
For each $\bm{b} = \{i_1<i_2<\cdots<i_s\}\in \mathcal{B}_s$, let
\begin{equation*}\bm{I}_{\bm{b}}=
\{i_1,i_1+1,i_2,i_2+1,\cdots i_s,i_s+1\}.
\end{equation*}
Then
\[\sum_{N\geqslant 2s\geqslant 0}\kappa^{g}_{-2s}\]
gives a Koszul representative for $\mathfrak{1}_g$ where
\begin{equation*}
\kappa^g_{-2s}= \sum_{\bm{b}\in \mathcal{B}_s}\Big(\prod_{i\in\bm{b}}b^g_i\Big)ge_{\bm{I}_g\setminus \bm{I_b}}.
\end{equation*}

The proof is the same as {Lemma \ref{lem:loop closed cochain}}. As in the loop case,
\begin{equation*}
\mathfrak{1}_g\cup \mathfrak{1}_{g^{-1}}=\mathsf{p} \circ \tilde \Phi^*\left(\Upsilon^*(\kappa^g_0+\kappa^g_{-2}+\cdots)\cup\Upsilon^*(\kappa^{g^{-1}}_0+\kappa^{g^{-1}}_{-2}+\cdots)\right).\end{equation*}
is a sum over all the graphs $\Gamma$ with
\begin{description}
  \item[(a)] a set of vertices $V(\Gamma)$ indexed by $1,2,\cdots l_g$;
  \item[(b)] a set of edges $E(\Gamma)$ of three types such that each vertex has valency 2:
  \begin{description}
    \item[Type `$++$':] edges having sources coloured $+$ and targets coloured $+$, indexed by
        $\pplr{v_{\text{source}}}{v_{\text{target}}}$, satisfying that for each edge $\pplr{i}{j}$, $j=i+1$,
    \item[Type `$--$':] edges having sources coloured $-$ and targets coloured $-$, indexed by
        $\mmlr{v_{\text{source}}}{v_{\text{target}}}$, satisfying that for each edge $\mmlr{i}{j}$, $j=i+1$,
    \item[Type `$+-$':] edges having sources coloured $+$ and targets coloured $-$, indexed by
        $\pmlr{v_{\text{source}}}{v_{\text{target}}}$, satisfying that for each edge $\pmlr{i}{j}$, $i=j$ or $i =j+1$,
  \end{description}
\end{description}
Given such a graph $\Gamma$, we assign the value of each edge by
\begin{align*}
\val{\pmlr{i}{i}}&=
\begin{dcases} \f{[n_{i}]_{\lambda_{i}}-n_i}{\lambda_i-1}x_{i}^{n_{i}-2}x_{i+1} &\text{if } 1\leqslant i<l_g,\\
\f{-n_{l_g}}{\lambda_{l_g}-1}x_{l_g}^{n_{l_g}-2}x_{l_g+1}&\text{if } i=l_g<N,\\
\f{-n_{l_g}}{\lambda_{l_g}-1}x_{l_g}^{n_{l_g}-2}&\text{if } i=l_g=N,
\end{dcases}\\
\val{\pmlr{i}{i-1}}&= [n_{i-1}]_{\lambda_{i-1}}x_{i-1}^{n_{i-1}-1},\\
\val{\pplr{i}{i+1}}&= b_i^g = \f{\lambda_i^{n_i}}{1-\lambda_i}x_i^{n_i-1},\\
\val{\mmlr{i}{i+1}}&= \leftidx{^g}b_i^{g^{-1}} = -\f{1}{1-\lambda_i}x_i^{n_i-1},
\end{align*}
and the value of a graph $\Gamma$ is defined as
\begin{equation*}
\val{\Gamma}= (-1)^{\sgn{\Gamma}+\f{(s_{\Gamma}-1)s_{\Gamma}}{2}}\prod_{e\in E(\Gamma)} \val{e}.
\end{equation*}
Here $s_\Gamma$ is the number of edges of type `$+-$' of $\Gamma$. If the type `$+-$' edges have source index  $i_1<i_2<\cdots< i_{s_\Gamma}$ and target index $j_1<j_2<\cdots <j_{s_\Gamma}$ such that it connects $i_k$ to $j_{\sigma(k)}$, then $(-1)^{\sgn{\Gamma}}$ is the sign of the permutation $\sigma$.

\begin{ex}
For $l_g=3$, we need to consider the following graphs.
\begin{equation*}
\tikz[baseline=0]{
\draw[thick] (0,0) arc (180:-180:1.5mm);
\draw[thick] (1,0) arc (180:-180:1.5mm);
\draw[thick] (2,0) arc (180:-180:1.5mm);
\filldraw[fill=black] (0,0) node[left,font=\tiny] {$1$} circle (1pt);
\filldraw[fill=black] (1,0) node[left,font=\tiny] {$2$} circle (1pt);
\filldraw[fill=black] (2,0) node[left,font=\tiny] {$3$} circle (1pt);
}\quad,\quad
\tikz[baseline=0]{
\draw[thick] (0,0) arc (180:-180:1.5mm);
\draw[thick,dotted] (1,0) to[bend right] (2,0);
\draw[thick,dashed] (1,0) to[bend left] (2,0);
\filldraw[fill=black] (0,0) node[left,font=\tiny] {$1$} circle (1pt);
\filldraw[fill=black] (1,0) node[left,font=\tiny] {$2$} circle (1pt);
\filldraw[fill=black] (2,0) node[right,font=\tiny] {$3$} circle (1pt);
}\quad\text{and}\quad
\tikz[baseline=0]{
\draw[thick] (2,0) arc (180:-180:1.5mm);
\draw[thick,dotted] (0,0) to[bend right] (1,0);
\draw[thick,dashed] (0,0) to[bend left] (1,0);
\filldraw[fill=black] (0,0) node[left,font=\tiny] {$1$} circle (1pt);
\filldraw[fill=black] (1,0) node[right,font=\tiny] {$2$} circle (1pt);
\filldraw[fill=black] (2,0) node[left,font=\tiny] {$3$} circle (1pt);
}\quad.
\end{equation*}
\end{ex}

Graphs for chain types are those of type $\Rmnum{1}$ graphs for the loop case.
\[
\tikz{
\draw[thick] (0,0) arc (90:450:1.5mm);
\draw[thick] (1,0) arc (90:450:1.5mm);
\draw[thick,dotted] (2,0) to[bend right] (3,0);
\draw[thick,dashed] (2,0) to[bend left] (3,0);
\draw[thick] (4,0) arc (90:450:1.5mm);
\draw[thick,dotted] (5,0) to[bend right] (6,0);
\draw[thick,dashed] (5,0) to[bend left] (6,0);
\draw[loosely dotted] (6.3,0) -- (6.7,0);
\filldraw[fill=black] (0,0) node[above,font=\tiny] {$1$} circle (1pt);
\filldraw[fill=black] (1,0) node[above,font=\tiny] {$2$} circle (1pt);
\filldraw[fill=black] (2,0) node[above,font=\tiny] {$3$} circle (1pt);
\filldraw[fill=black] (3,0) node[above,font=\tiny] {$4$} circle (1pt);
\filldraw[fill=black] (4,0) node[above,font=\tiny] {$5$} circle (1pt);
\filldraw[fill=black] (5,0) node[above,font=\tiny] {$6$} circle (1pt);
\filldraw[fill=black] (6,0) node[above,font=\tiny] {$7$} circle (1pt);
\filldraw[fill=black] (7,0) node[above,font=\tiny] {$l_g$} circle (1pt);
}
\]
Similar to the discussion for the loop case, let us define the matrix

\begin{equation}\label{eq:quantum hess chain}
\mathrm{H}^g_W= \begin{pmatrix}
h_{11}&h_{12}&&&&\\
h_{21}&h_{22}&h_{23}&&&\\
&h_{32}&h_{33}&h_{34}&&\\
&&&\ddots&&\\
&&&&\ddots&h_{(l_g\!-\!1)l_g}\\
&&&&h_{l_g(l_g\!-\!1)}&h_{l_gl_g}
\end{pmatrix}
\end{equation}
where the entries are defined by
\begin{equation*}h_{ij} = \begin{dcases}
\val{\pmlr{i}{i}}& \text{if } i=j,\\
\val{\pplr{i}{i+1}}& \text{if } l_g \geqslant j=i+1\geqslant 2,\\
\val{\mmlr{i-1}{i}} & \text{if }l_g-1\geqslant j=i-1\geqslant 1,\\
0 & \text{otherwise}.
\end{dcases}
\end{equation*}
Then we have the analogue of \eqref{det-formula-loop} in the chain case
\begin{align}\label{det-formula-chian}
\mathfrak{1}_g\cup\mathfrak{1}_{g^{-1}} = (-1)^{\f{(l_g-1)l_g}{2}} \det \mathrm{H}^g_W.
\end{align}
The computation of this determinant is also the same as in the loop case for type $\Rmnum{1}$ graphs. In $\jac(W)$,
\begin{equation*}
\begin{dcases}
x_{i-1}^{n_{i-1}}+n_i x_i^{n_i-1}x_{i+1}=0& \text{for } 2\leqslant i\leqslant N-1,\\
x_{N-1}^{n_{N-1}}+n_N x_N^{n_N-1}=0.
\end{dcases}
\end{equation*}
For $l_g<N$, since $\pplr{i}{i+1}$ and $\mmlr{i}{i+1}$ come in pairs and
\begin{align*}
h_{i(i\!+\!1)}h_{(i\!+\!1)i} &= -\f{\lambda_i^{n_i}}{(1-\lambda_i)^2}x_i^{2n_i-2}\nonumber\\
&=  -\f{(1-\lambda_i^{n_i})n_{i+1}}{(1-\lambda_i)^2(1-\lambda_{i+1})}x_i^{n_i-2}x_{i+1}^{n_{i+1}-1}x_{i+2}\nonumber\\
&= \lr{\f{[n_i]_{\lambda_i}}{\lambda_i-1}x_i^{n_{i}-2}x_{i+1}}\lr{-\f{n_{i+1}}{\lambda_{i+1}-1}x_{i+1}^{n_{i+1}-2}x_{i+2}}.
\end{align*}
Without changing of the value of $\det \mathrm{H}^g_W$, we can replace $h_{i(i\!+\!1)}$ and $h_{(i\!+\!1)i}$ by
\begin{align*}\tilde{h}_{i(i\!+\!1)} &= \f{[n_i]_{\lambda_i}}{\lambda_i-1}x_i^{n_{i}-2}x_{i+1},\\
\tilde{h}_{(i\!+\!1)i} &= -\f{n_{i+1}}{\lambda_{i+1}-1}x_{i+1}^{n_{i+1}-2}x_{i+2}.
\end{align*}
Observe that
\begin{equation*}
h_{ii} =
\begin{dcases}
\tilde{h}_{i(i\!+\!1)}+\tilde{h}_{(i\!-\!1)i}& \text{while }2\leqslant i\leqslant l_g-1,\\
\tilde{h}_{l_g(l_g\!-\!1)}&\text{while }i=l_g.
\end{dcases}
\end{equation*}
It leads to cancellation of sum of graphs that we are left with only two terms
\begin{align*}
&(-1)^{\f{(l_g-1)l_g}{2}}\big((h_{11}-\tilde{h}_{12})\tilde{h_{21}}\cdots \tilde{h}_{l_g(l_g\!-\!1)}\big)\nonumber\\
=&(-1)^{\f{(l_g-1)l_g}{2}}\f{n_1n_2\cdots
n_{l_g}}{(1-\lambda_1)\cdots(1-\lambda_{l_g})}x_1^{n_1-2} x_2^{n_2-1}\cdots x_{l_g}^{n_{l_g}-1}x_{l_g+1}.
\end{align*}
For $l_g=N$, the computation is similar. This proves part (2) of Theorem \ref{thm: cup product formula for General invertible polynomial} in the chain case.

\Addresses

\end{document}